\documentclass[a4paper,11pt,DIV=11,abstract=on]{scrartcl}

\usepackage[T1]{fontenc}
\usepackage[utf8]{inputenc}
\usepackage{amssymb,amsmath,amsthm,bm,mathtools,stmaryrd}
\usepackage{algorithm, algpseudocode}
\usepackage{siunitx} % for numbers
\usepackage{aliascnt}
\usepackage{array}
\mathtoolsset{showonlyrefs}
\usepackage{csquotes}
\usepackage{booktabs}
\usepackage[shortlabels]{enumitem}
\usepackage{cite}

\usepackage{graphicx}
\usepackage{xcolor}
\graphicspath{{../images/},{../matlab/images/}}
\usepackage{subcaption}
\usepackage{tikz,pgfplots}
\pgfplotsset{compat=1.18}

\usepackage{bbm, dsfont} % for mathbbm 1
\usepackage[english,pdfusetitle,colorlinks=true,linkcolor=blue,citecolor=green!50!black,%pagebackref,%urlcolor=blue,
bookmarks=true]{hyperref}

\hypersetup{
  pdfauthor={
    Michael Quellmalz, quellmalz@math.tu-berlin.de;
    Léo Buecher, leo.buecher@student-cs.fr;
    Gabriele Steidl, steidl@math.tu-berlin.de
  }
}

\newtheorem{theorem}{Theorem}[section]

\newaliascnt{lemma}{theorem}
\newtheorem{lemma}[lemma]{Lemma}
\aliascntresetthe{lemma}

\newaliascnt{corollary}{theorem}
\newtheorem{corollary}[corollary]{Corollary}
\aliascntresetthe{corollary}

\newaliascnt{proposition}{theorem}
\newtheorem{proposition}[proposition]{Proposition}
\aliascntresetthe{proposition}

\theoremstyle{definition}

\newaliascnt{remark}{theorem}
\newtheorem{remark}[remark]{Remark}
\aliascntresetthe{remark}

\newtheorem*{remark*}{Remark}

\newcommand{\C}{\ensuremath{\mathbb{C}}}
\newcommand{\D}{\ensuremath{\mathbb{D}}}

\renewcommand{\H}{\ensuremath{\mathbb{H}}}
\newcommand{\II}{\ensuremath{\mathbb{I}}}
\newcommand{\N}{\ensuremath{\mathbb{N}}}
\newcommand{\NN}{\ensuremath{\mathbb{N}_0}}

\newcommand{\R}{\ensuremath{\mathbb{R}}}
\renewcommand{\S}{\ensuremath{\mathbb{S}}}
\newcommand{\T}{\ensuremath{\mathbb{T}}}
\newcommand{\X}{\ensuremath{\mathbb{X}}}
\newcommand{\Y}{\ensuremath{\mathbb{Y}}}
\newcommand{\Z}{\ensuremath{\mathbb{Z}}}

\newcommand{\M}{\ensuremath{\mathcal{M}}}
\newcommand{\U}{\ensuremath{\mathcal{U}}}
\newcommand{\V}{\mathcal{V}}
\newcommand{\cW}{\mathcal{W}}
\newcommand{\cA}{\mathcal{A}}
\newcommand{\cB}{\mathcal{B}}
\newcommand{\cE}{\mathcal{E}}

\newcommand{\cR}{\mathcal{R}}
\newcommand{\cS}{\mathcal{S}}
\newcommand{\cT}{\mathcal{T}}

\newcommand{\cZ}{\mathcal{Z}}
\newcommand{\rR}{\mathrm{R}}

\renewcommand{\P}{\ensuremath{\mathcal{P}}}
\newcommand{\ac}{\ensuremath{\mathrm{ac}}}
\newcommand{\Pac}{\ensuremath{\mathcal{P}_{\ac}}}
\newcommand{\Pe}{\ensuremath{\mathcal{P}_{\mathrm{even}}}}

\newcommand{\SO}{\ensuremath{\mathrm{SO}(3)}}
\newcommand{\SOd}{\ensuremath{\mathrm{SO}(d)}}

\newcommand{\abs}[1]{\ensuremath{\left\vert#1\right\vert}}
\newcommand{\inn}[1]{\ensuremath{\left\langle#1\right\rangle}}
\newcommand{\dx}{\mathrm{d}}
\newcommand{\ds}{\,\mathrm{ds}}
\newcommand{\e}{\mathrm{e}}

\renewcommand{\i}{\mathrm{i}}

\newcommand{\zb}[1]{\ensuremath{\boldsymbol{#1}}}

\DeclareMathOperator{\azi}{azi}
\DeclareMathOperator{\zen}{zen}

\DeclareMathOperator*{\argmin}{arg\,min}

\DeclareMathOperator{\proj}{proj}

\DeclareMathOperator{\Id}{Id}

\DeclareMathOperator{\bary}{Bary}
\DeclareMathOperator{\sph}{\Phi}
\DeclareMathOperator{\eul}{\Psi}
\DeclareMathOperator{\trace}{trace}
\DeclareMathOperator{\SW}{SW}
\DeclareMathOperator{\PSW}{PSW}
\DeclareMathOperator{\RSW}{RSW}

\DeclareMathOperator{\SSW}{SSW}
\DeclareMathOperator{\SOSW}{SOSW}
\DeclareMathOperator{\WS}{W}
\DeclareMathOperator{\W}{W}

\DeclareMathOperator{\CDT}{CDT}

\renewcommand{\d}{\, \mathrm{d}}

\newcommand{\norm}[1]{\left\lVert #1%\smash{#1} 
  \right\rVert}

\newcommand{\ii}[1]{\llbracket #1\rrbracket}

\newcommand{\ba}{{\boldsymbol a}}

\newcommand{\be}{{\boldsymbol e}}

\newcommand{\bg}{{\boldsymbol g}}

\newcommand{\bn}{{\boldsymbol n}}

\newcommand{\bq}{{\boldsymbol q}}
\newcommand{\br}{{\boldsymbol r}}

\newcommand{\bu}{{\boldsymbol u}}
\newcommand{\bv}{{\boldsymbol v}}
\newcommand{\bw}{{\boldsymbol w}}
\newcommand{\bx}{{\boldsymbol x}}
\newcommand{\by}{{\boldsymbol y}}
\newcommand{\bz}{{\boldsymbol z}}

\newcommand{\bA}{{\boldsymbol A}}
\newcommand{\bB}{{\boldsymbol B}}

\newcommand{\bP}{{\boldsymbol P}}
\newcommand{\bQ}{{\boldsymbol Q}}
\newcommand{\bR}{{\boldsymbol R}}
\newcommand{\blambda}{{\boldsymbol \lambda}}
\newcommand{\bpsi}{{\boldsymbol \psi}}

\newcommand{\bxi}{\boldsymbol\xi}

\mathtoolsset{showonlyrefs}

\begin{document}

\def\sectionautorefname{Section}
\def\subsectionautorefname{Section}

\title{Parallelly Sliced Optimal Transport \\
on Spheres and on the Rotation Group}

\author{
Michael Quellmalz\textsuperscript{1}
\and
    Léo Buecher\textsuperscript{1,2}
	\and
	Gabriele Steidl\textsuperscript{1}
}

\maketitle
\date{\today}

\footnotetext[1]{Institute of Mathematics,
	Technische Universität Berlin,
	Stra{\ss}e des 17.\ Juni 136, 
	10623 Berlin, Germany,
	\ttfamily{\{quellmalz, steidl\}@math.tu-berlin.de}
    \url{https://tu.berlin/imageanalysis}
	} 

\footnotetext[2]{CentraleSupélec, Université Paris-Saclay, \ttfamily{leo.buecher@student-cs.fr}}

\begin{abstract}
Sliced optimal transport, which is basically a Radon transform followed by one-dimensional
optimal transport, became popular in various applications due to its efficient
computation. 
In this paper, we deal with sliced optimal transport on the sphere $\S^{d-1}$ and on the rotation group $\SO$. 
We propose a parallel slicing procedure of the sphere which requires again only
optimal transforms on the line.
We analyze the properties of the corresponding
parallelly sliced optimal transport, which provides
in particular a rotationally invariant metric on the spherical probability measures.
For $\SO$, we introduce a new two-dimensional Radon transform and 
develop its singular value decomposition. 
Based on this, we propose a sliced optimal transport on $\SO$.

As Wasserstein distances were extensively used in barycenter computations, 
we derive algorithms to compute the barycenters with respect to our new sliced Wasserstein distances
and provide synthetic numerical examples on the 2-sphere that demonstrate their behavior {for}
both the free and fixed support setting of discrete spherical measures.
In terms of computational speed, they outperform the existing methods for semicircular slicing as well as the regularized Wasserstein barycenters.
\end{abstract}

%-----------------------------------------------
\section{Introduction}
%-----------------------------------------------

Optimal transport (OT) deals with the problem of finding the most efficient way to transport probability measures.
The Wasserstein distance is a metric on the space of probability measures and has received much attention  \cite{PeyCut19,San15,Vil03},
e.g., for neural gradient flows \cite{FZTC2022,HWAH2023,KSB2022,AHS2023} in machine learning.
Since OT on multi-dimensional domains is hard to compute, there exist different modifications that allow an efficient computation,
such as the entropic regularization that yields the Sinkhorn algorithm \cite{Cut13,PeyCut19,Kni08,BaQue22}. 
Sliced OT on Euclidean spaces utilizes the Radon transform to reduce the problem to the real line \cite{NRNRNH23,RabPeyDelBer12,San15,BonRabPeyPfi15},
where OT possesses an analytic solution that can be computed efficiently.
The notion of sliced OT can be generalized to other Radon-like transforms \cite{KolNadSimBadRoh19}.
{Specific geometries have been considered such as spheres \cite{Bon23,QueBeiSte23}, manifolds with constant negative curvature \cite{Bon23a}, or separable Hilbert spaces \cite{Han23}.}

In this paper, we are interested in sliced OT on special manifolds.
A slicing approach on Riemannian manifolds based on eigenfunctions of the Laplacian was proposed in \cite{RusMaj23}.
OT on the sphere has been intensely studied, e.g., the computation of Wasserstein barycenters \cite{StaClaSolJeg17,TheKer22},
the regularity of optimal maps \cite{Loe10},
isometric rigidity of Wasserstein spaces \cite{GehHruTitVir23}
a connection with a Monge--Ampère type equation \cite{HamTur22,McrCotBud18,WelBroBudCul16},
or a variational framework \cite{CuiQi19}.
Sliced OT was generalized to spheres in two different ways:
Bonet et al. \cite{Bon22} introduced a slicing along semicircles to reduce the OT problem to 
one-dimensional circle, see \autoref{fig:slicing} right. This requires only OT computations on circles 
which was examined in \cite{DelSalSob10}.
The respective sliced Wasserstein distance is a metric in the space of probability measures on the 2-sphere \cite{QueBeiSte23}. 
Note that Radon transforms on such semicircles have been considered before in \cite{HiPoQu18,Gro98},
providing an extension of the Funk--Radon transform \cite{Fun13,Hel11,LoRiSpSp11,HiQu15,QueWeiHubErc23}.
A second approach \cite{QueBeiSte23} of sliced spherical Wasserstein distances is based on the vertical slice transform \cite{ZaSc10,HiQu15circav,Rub19}.
This yields a family of measures on the unit interval instead of the circle 
and is therefore faster to compute than the first approach.
However, this vertical slicing approach provides only a metric for even measures on the 2-sphere, i.e. the same values are taken
on the upper and lower hemisphere. 
This is a serious restriction for practical applications.

In this paper, we generalize the vertically sliced OT 
from even measures to arbitrary probability measures 
by constructing a so-called parallelly sliced OT, see \autoref{fig:slicing} for an illustration.
We provide a method for spheres $\S^{d-1}$ in general dimensions $d$.
The key advantages are that the respective sliced Wasserstein distance is a rotationally invariant metric on the spherical probability measures,
and that it is faster to compute than the semicircular sliced Wasserstein distance since we project on intervals instead of circles.
Our numerical tests indicate a speedup between 40 and 100 times.
In \autoref{thm:SSW}, we prove estimates between the spherical Wasserstein distance and its parallelly sliced version.
\begin{figure}[ht]
    \centering
    \includegraphics[height=4cm]{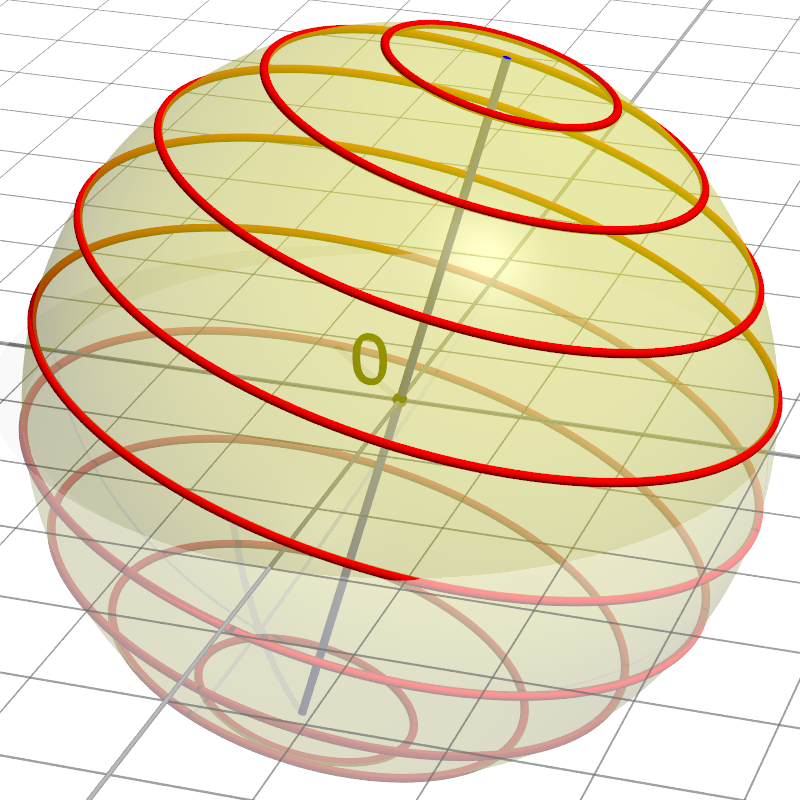}\qquad
    \includegraphics[height=4cm]{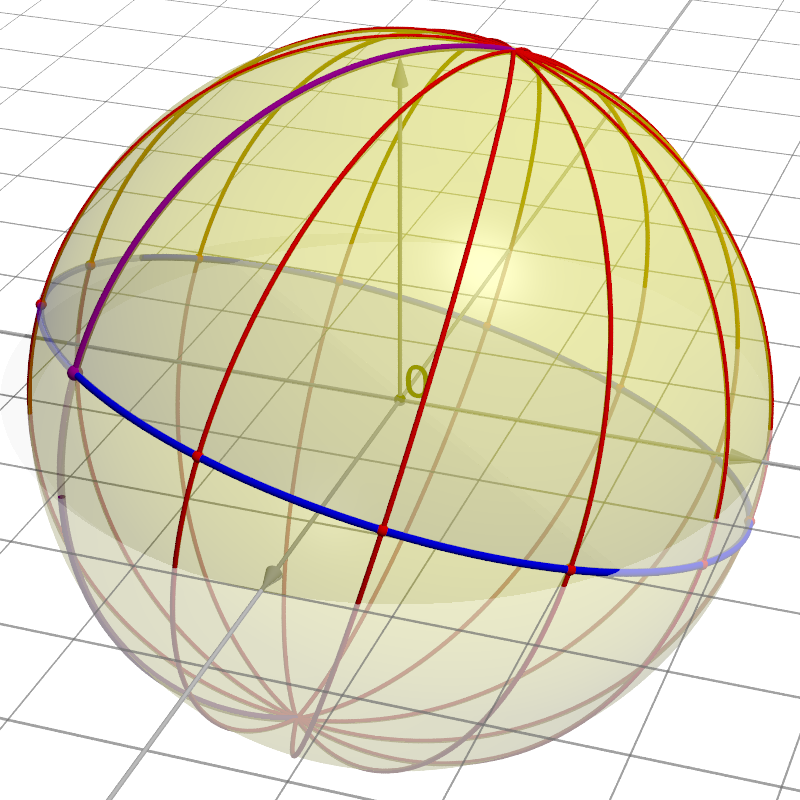}
    \caption{{Our proposed} parallel slicing (left): each red circle is projected to a single point on the blue line segment.
    {The vertical slicing \cite{QueBeiSte23} can be regarded as special case that keeps the blue line direction fixed.}
    Semicircular {slicing from \cite{Bon23}} (right): each red semicircle is projected to one point on the blue circle.}
    \label{fig:slicing}
\end{figure}

Furthermore, we consider OT in the group $\SO$ of three-dimensional rotation matrices,
which has applications {in the synchronization of probability measures on rotations} \cite{BiArSiGu20}.
A Radon transform along one-dimensional geodesics of $\SO$ was proposed in \cite{BoHiPrSc97,HiPoJuScSc07},
but, for the purpose of OT,
we require slicing along two-dimensional submanifolds of $\SO$.
Therefore, we  develop a new two-dimensional Radon transform on $\SO$,
including its singular value decomposition and adjoint operator. 
This paves the way to prove that the corresponding sliced Wasserstein distance fulfills the metric properties on the set of probability measures on $\SO$.

Barycenter computations with respect to Wasserstein distances and their sliced variants
are of increasing interest \cite{PeyCut19,BonRabPeyPfi15,RabPeyDelBer12}.
Therefore, we deal with barycenters with respect to our new sliced Wasserstein distances
and describe their computation 
both in the free and fixed support discrete setting,
as well as so-called Radon Wasserstein barycenter.
As proof of the concept, we give numerical examples of barycenter computations
on the 2-sphere.
We compare our approach with the semicircular slicing \cite{Bon23} as well as the entropy-regularized Wasserstein barycenter computed with PythonOT \cite{POT}. 

\paragraph*{Outline of the paper.}
We provide the basic preliminaries on OT and the manifolds
$\mathbb S^{d-1}$ and $\SO$ in Section \ref{sec:prelim}.
The parallel slice transform for functions and measures, and the corresponding parallelly sliced Wasserstein $p$-distances are introduced in Section \ref{sec:slice_sphere}.
In Section \ref{sec:SO}, we generate sliced Wasserstein distances on $\SO$
based on our new two-dimensional Radon transform on $\SO$.
Barycenter computations are examined in two different ways, namely
sliced Wasserstein barycenters and Radon Wasserstein barycenters
in Section \ref{sec:bary}. 
In the former case, we deal both with free and fixed discretization.
In Section \ref{sec:numerics}, we demonstrate by synthetic numerical examples the
performance of our barycenter algorithms on the 2-sphere. 
In particular, we compare our parallel slicing approach with the slicing method of
Bonet et al. \cite{Bon22}. Here some theoretically expected phenomena are illustrated.
Finally, conclusions are drawn in Section \ref{sec:conc}.
The appendix contains technical proofs.

% -----------------------------------------------
\section{Preliminaries} \label{sec:prelim}
%-----------------------------------------------
In this section, we provide the notation and necessary preliminaries 
on OT, in particular on the interval,
and harmonic analysis on the unit sphere on $\mathbb R^d$. 

%------------------------------------------------------
\subsection{Measures and OT}
%------------------------------------------------------
Let $\X$ be a compact {Riemannian manifold} 
with metric $d\colon \X \times \X \to \R$,
and let $\cB(\X)$ be the Borel $\sigma$-algebra induced by $d$.
We denote by $\M(\X)$ the Banach space of signed, finite measures, and 
by $\P(\X)$ the subset of probability measures
on $\X$. 
The pre-dual space of $\M(\X)$ is the space of continuous functions $C(\X)$.
Let $\Y$ be another compact manifold and $T\colon \X \to \Y$
be measurable. For $\mu \in \M(\X)$, 
we define the \emph{push-forward measure} 
$T_\# \mu \coloneqq \mu \circ T^{-1} \in \M(\mathbb Y)$.

The $p$-\emph{Wasserstein distance}, $p\in[1,\infty)$, of $\mu,\nu \in \P(\mathbb X)$ is given by
\begin{equation} \label{eq:Wp}
    \WS_p^p(\mu,\nu)
    \coloneqq
    \min_{\pi\in\Pi(\mu,\nu)} \int_{\mathbb X^2} d^p(x,y) \d \pi (x,y),
\end{equation}
with
$
\Pi(\mu,\nu) \coloneqq \{\pi \in\M(\mathbb X \times \mathbb X): 
\pi(B \times \mathbb  X) = \mu(B),
\pi(\mathbb  X\times B) = \nu(B)
\ \text{for all }  B \in \cB(\mathbb X)\}
$. 
It defines a metric on $\P(\mathbb X)$.
The metric space $\P^p(\mathbb X) \coloneqq (\P(\mathbb X),W_p)$
is called \emph{$p$-Wasserstein space} and, in case $p=2$, just Wasserstein space.
The $p$-Wasserstein distance is a special case of the more general \emph{optimal transport} (OT) problem,
where $d^p(x,y)$ can be replaced by a more general cost function $c(x,y)$.
For 
$\blambda \in \Delta_M \coloneqq \{\blambda \in [0,1]^M\mid \sum_{i=1}^{M} \lambda_i=1\},$
the \emph{Wasserstein barycenter} of $\mu_i\in\P^2(\X)$, $i\in\ii{M} \coloneqq \{1,\dots,M\}$, is the minimizer  
\begin{equation} \label{eq:W-bary}
    \bary^{\WS}_X(\mu_i, \lambda_i)_{i=1}^{M}
    \coloneqq
    \argmin_{\nu\in\P(\X)}
    \sum_{i=1}^{M}
    \lambda_i\, \WS_2^2(\nu,\mu_i),
\end{equation}
see \cite{AguCar11}.
The Wasserstein barycenter of absolutely continuous measures is unique \cite{KimPas17}.

\paragraph*{OT on the Interval}

If $\X$ is the \emph{unit interval} $\II \coloneqq [-1,1]$ with the distance $d(x,y) = \abs{x-y}$, the OT between two probability measures $\mu,\nu\in \P(\II)$ can be computed easily \cite{PeyCut19,San15,Vil03} 
using the \emph{cumulative distribution function} 
$F_\mu(x) \coloneqq \mu([-1,x])$, $x\in\II$,
which is non-decreasing and right continuous.
Its pseudoinverse, 
the \emph{quantile function}
$
F_\mu^{-1}(r) \coloneqq \min\{ x\in\II\mid F_\mu(x)\ge r\}$,
$r\in[0,1]$,
is non-decreasing and left continuous.
The $p$-Wasserstein distance \eqref{eq:Wp} between $\mu,\nu \in \P^p(\II)$
now equals 
$\WS_p(\mu,\nu)
=
\lVert F_\mu^{-1} - F_\nu^{-1} \rVert_{L^p([0,1])}$.
If $\mu \in \Pac(\II)$, where $\Pac(\II)$ denotes the probability measures that are
\emph{absolutely continuous} with respect to the Lebesgue measure,
then the OT plan $\pi$ in \eqref{eq:Wp}
is uniquely given by 
\begin{equation}  
    \pi = (\Id, T^{\mu,\nu})_\# \mu
    \quad\text{with}\quad 
    T^{\mu,\nu}(x) \coloneqq F_\nu^{-1}(F_\mu(x)),
    \quad
    x \in \II.
\end{equation}

Based on the \emph{OT map} $T^{\mu,\nu}$,
the Wasserstein space $\P^p(\II)$ can be isometrically embedded 
into $L^p_\omega(\II)$ with $\omega \in \Pac(\II)$
\cite{KolParRoh16, PaKoSo18, BeBeSt22},
where $L^p_\omega(\II)$ consists of all $p$-integrable functions with respect to $\omega$.
For a reference measure $\omega \in \Pac(\II)$, the \emph{cumulative distribution transform} (CDT)
is defined by 
$\CDT_\omega \colon \P^p(\II) \to L^p_\omega(\II)$ with
\begin{equation} \label{eq:cdt}
    \CDT_{\omega}[\mu] (x) 
    \coloneqq 
    (T^{\omega,\mu} - \Id)(x)
    =
    \bigl(F_{\mu}^{-1} \circ F_{\omega} \bigr) (x) - x,
    \quad
    x\in\II.
\end{equation}
The CDT is in fact a mapping from $\P^p(\II)$ into the tangent space of $\P^p(\II)$ at $\omega$,
see \cite[§~8.5]{AGS05}.
Due to the relation to the OT map,
the CDT can be inverted by $\mu = \CDT^{-1}_\omega[h] \coloneqq (h + \Id)_\# \omega$ 
{for $h = \CDT_\omega[\mu]$}. 
If $\mu,\omega \in \Pac(\II)$ possess positive density functions $f_\mu$ and $f_\omega$,
then, by the transformation formula for push-forward measures, $f_\mu$ can be recovered by
\begin{equation} \label{eq:icdt}
    f_\mu(x)
    =
    \left( g^{-1} \right)'(x)\, f_\omega(g^{-1}(x))
    \quad
    \text{with} \quad
    g(x) = \CDT_\omega[\mu](x) + x
    ,\quad
    x \in \II.
\end{equation}
For $\mu_i \in \P(\II)$
and an arbitrary reference $\omega \in \Pac(\II)$,
the Wasserstein barycenter \eqref{eq:W-bary} has the form \cite{KolParRoh16}
\begin{equation}
    \label{eq:cdt-bary}
    \bary_\II (\mu_i,\lambda_i)_{i=1}^{M}
    =
    \CDT^{-1}_{\omega}\left(
    \sum_{i=1}^{M} \lambda_i \CDT_\omega [\mu_i] \right).
\end{equation}

\paragraph*{Sliced OT}
Given another compact space $\D$ with a probability measure $u_\D$ and a slicing operator $\cS_{\psi} \colon \X\to\R$ for all $\psi\in\D$,
we define the \emph{sliced $p$-Wasserstein distance}
\begin{equation} \label{eq:SW-general}
    \SW_p^p(\mu,\nu)
    \coloneqq
    \int_\D \WS_p^p((\cS_{\psi})_\#\mu, (\cS_{\psi})_\#\nu) \d u_{\D}(\psi).
\end{equation}
Sliced {Wasserstein} distances on the Euclidean space $\X=\R^d$ with the slicing operator $\cS^{\R^d}_\bpsi \coloneqq \inn{\bpsi,\cdot}$ for $\bpsi\in\D = \S^{d-1}$ are well known \cite{RabPeyDelBer12,BonRabPeyPfi15,San15}.
Sliced OT is closely related with the Radon transform
\begin{equation} \label{eq:Radon}
\cR_\bpsi\colon \P(\R^d)\to \P(\R),\quad 
\mu\mapsto
(\cS^{\R^d}_\bpsi)_\# \mu.
\end{equation}
The Radon transform is often defined for functions on $\R^d$ via an integral, see \cite{NaWue00}.

%---------------------------------------------
\subsection{Sphere}
%---------------------------------------------
Let $d\in\N$ with $d\ge3$.
We define the $(d-1)$-dimensional unit sphere in $\R^d$ by
$$\S^{d-1} \coloneqq \{\bx \in\R^d\mid \norm{\bx}=1\},$$
and denote the canonical unit vectors by $\zb e^j\in\R^d$ for
$j \in \ii{d} \coloneqq \{1,\dots,d\}$.
The geodesic distance on the sphere $\S^{d-1}$ reads as
\begin{equation} \label{eq:Sd-dist}
d(\zb\xi,\zb\eta) \coloneqq \arccos(\inn{\zb\xi,\zb\eta})
,\qquad \forall \bxi,\zb\eta\in\S^{d-1},
\end{equation}
and we denote the volume of $\S^{d-1}$ by
\begin{equation} \label{eq:Sd-vol}
    \abs{\S^{d-1}}
    \coloneqq
    \int_{\S^{d-1}} \d \sigma_{\S^{d-1}}
    =
    \frac{2\pi^{d/2}}{\Gamma(d/2)},
\end{equation}
where $\sigma_{\S^{d-1}}$ is the surface measure on $\S^{d-1}$.
Normalizing $\sigma_{\S^{d-1}}$
yields the \emph{uniform measure} $u_{\S^{d-1}} \coloneqq \abs{\smash{\S^{d-1}}}^{-1} \sigma_{\S^{d-1}}$.
We can write any vector $\bxi\in\S^{d-1}$ as
\begin{equation} \label{eq:xi-decomp}
\bxi =  \begin{pmatrix}\sqrt{1-t^2}\,\zb\eta\\ t \end{pmatrix} 
\qquad \text{for}\quad \zb\eta\in\S^{d-2},\ t\in\II,
\end{equation}
then the surface measure on the sphere $\S^{d-1}$ decomposes as \cite[(1.16)]{AtHa12}
\begin{equation} \label{eq:Sd_measure}
\dx \sigma_{\S^{d-1}}(\bxi)
=
\dx \sigma_{\S^{d-2}}(\zb\eta)
\,(1-t^2)^{\frac{d-3}{2}} \d t.
\end{equation}
For $\S^2$, we denote the bijective spherical coordinate transform by
\begin{equation} \label{eq:sph}
\sph\colon[0,2\pi)\times(0,\pi)\cup\{0\}\times\{0,\pi\}\to\S^2,\quad
(\varphi,\theta)\mapsto (\cos\varphi\sin\theta,\sin\varphi\sin\theta,\cos\theta)^\top.
\end{equation}

\paragraph*{Spherical harmonics}
Let $n\in\NN$. We denote by $\Y_{n,d}$ the space of all polynomials $f\colon\R^d\to\C$ which are harmonic, i.e., the Laplacian $\Delta f$ vanishes everywhere, and homogeneous of degree $n$, i.e., $f(\alpha \bx) = \alpha^n f(\bx)$ for all $\alpha\in\R$ and $\bx\in\R^d$,  restricted to the sphere $\S^{d-1}$.
Setting
\begin{equation} \label{eq:Nnd}
N_{n,d} 
\coloneqq
\dim(\Y_{n,d})
=
\frac{(2n+d-2)\, (n+d-3)!}{n!\, (d-2)!},
\end{equation}
we call an orthonormal basis
$
\{Y_{n,d}^k
\mid k \in \ii{N_{n,d}}\}
$
of $\Y_{n,d}$ a basis of \emph{spherical harmonics} on $\S^{d-1}$ of degree $n$, cf.\ \cite{AtHa12}.
Then $\{ Y_{n,d}^k \mid n\in\NN,\, k\in\ii{N_{n,d}} \}$ forms an orthonormal basis of $L^2(\S^{d-1})$.
In particular,
we can write any $f\in L^2(\S^{d-1})$ as spherical Fourier series
\begin{equation} \label{eq:Y-series}
  f 
  =
  \sum_{n=0}^\infty \sum_{k=1}^{N_{n,d}}
  \hat f_{n,d}^k\, Y_{n,d}^k
  , \quad \text{where} \quad
  \hat f_{n,d}^k
  \coloneqq
  \inn{f,Y_{n,d}^k}_{L^2(\S^{d-1})}.
\end{equation}

The \emph{Legendre polynomial} $P_{n,d}$ of degree $n\in\NN$ in dimension~$d\ge2$ is given by \cite[(2.70)]{AtHa12}
\begin{equation}
P_{n,d}(t)
\coloneqq (-1)^n\, \frac{(d-3)!!}{(2n+d-3)!!}\, (1-t^2)^{\frac{3-d}{2}} \left(\frac{\dx}{\dx t}\right)^n (1-t^2)^{n+\frac{d-3}{2}}
,\qquad t\in[-1,1].
\label{eq:Pnd-Rodrigues}
\end{equation}
Up to normalization, the Legendre polynomials are equal to the Gegenbauer or ultraspherical polynomials, see \cite[(2.145)]{AtHa12}.
The normalized Legendre polynomials
\begin{equation}
\widetilde P_{n,d}(t)
\coloneqq \sqrt\frac{N_{n,d}\abs{\S^{d-2}}}{\abs{\S^{d-1}}}\, P_{n,d}(t)
= \frac{\sqrt{(2n+d-2)\,(n+d-3)!}}{2^{(d-2)/2}\,\sqrt{n!}\,\Gamma(\frac{d-1}{2})}\, P_{n,d}(t)
\label{eq:Pnd-normalized}
\end{equation}
satisfy the orthonormality relation
$
\int_{-1}^1 \widetilde P_{n, d}(t)\, \overline{\widetilde P_{m, d}(t)}\, (1-t^2)^{\frac{d-3}{2}} \d t
= \delta_{n, m},
$
where $\delta$ denotes the Kronecker symbol.

%---------------------------------------------------
\subsection{Rotation Group}
%---------------------------------------------------

We define the \emph{rotation group} 
$$\SOd \coloneqq \{\bQ\in\R^{d\times d}\mid \bQ^\top \bQ=I, \det(\bQ)=1\}.$$
We are especially interested in the 3D rotation group $\SO$.
Every rotation matrix can be written as the rotation around an axis $\bn \in \S^2$ with an angle $\omega \in \T \coloneqq \R/(2\pi\Z)$ , i.e.,
\begin{equation} \label{eq:axan}
  \rR_{\bn}(\omega) \coloneqq 
  (1-\cos\omega)\, \bn \bn^\top + 
  \begin{pmatrix}
    \cos \omega & - n_3 \sin \omega & 
    n_2 \sin \omega\\
    n_3 \sin \omega & \cos \omega & 
    n_1 \sin \omega\\
    - n_2\sin \omega & n_1 \sin \omega & 
    \cos \omega
  \end{pmatrix}
  \in\SO.
\end{equation}
Furthermore, we consider the \emph{Euler angle} parameterization
\begin{equation} \label{eq:Q}
  \eul\colon \T\times [0,\pi]\times \T \to \SO,\quad
  \eul(\alpha,\beta,\gamma)
  \coloneqq
  \rR_{\be^3}(\alpha) \rR_{\be^2}(\beta) \rR_{\be^3}(\gamma).
\end{equation}
The rotationally invariant Lebesgue measure $\sigma_{\SO}$ on $\SO$ is given by
\begin{align}
  \int_{\SO} f(\bQ) \d\sigma_{\SO}(\bQ)
  \label{eq:axan-int}
  &=
  2
  \int_{0}^{\pi}
  \int_{\S^2}
  f(\rR_{\bxi}(\omega))
  (1-\cos(\omega))
  \d\sigma_{\S^2}(\bxi)
  \d\omega,
\end{align}
see \cite[p.\ 8]{hiediss07},
and the uniform measure on $\SO$ is $u_{\SO} \coloneqq (8\pi^{2})^{-1} \sigma_{\SO}$.

The \emph{rotational harmonics} or \emph{Wigner D-functions} 
$D_n^{k,j}$ of degree $n\in\NN$ and orders $k,j\in\{-n,\dots,n\}$ are
defined by 
\begin{equation} \label{eq:D}
  D_n^{k,j} (\eul(\alpha,\beta,\gamma))
  \coloneqq \e^{-\mathrm{i}k\alpha}\, d_n^{k,j} (\cos\beta)\, \e^{-\mathrm{i}j\gamma},
\end{equation}
where the \emph{Wigner d-functions} are given for $t\in[-1,1]$ by 
\begin{align*}
  d_n^{k,j}(t) 
  &\coloneqq
  \frac{(-1)^{n-j}}{2^n} 
  \sqrt\frac{(n+k)!(1-t)^{j-k}}{(n-j)!(n+j)!(n-k)!(1+t)^{j+k}}
  \frac{\dx^{n-k}}{\dx t^{n-k}} \frac{(1+t)^{n+j}}{(1-t)^{-n+j}},
\end{align*}
see \cite[chap.~4]{Varsha88}.
The rotational harmonics satisfy the orthogonality relations
\begin{equation} \label{eq:D_ortho}
  \int_{\SO} D_n^{j,k}(\bQ) D_{n'}^{j',k'}(\bQ) \d \sigma_{\SO}(\bQ)
  =
  \frac{8\pi^2}{2n+1} \delta_{n,n'} \delta_{k,k'} \delta_{j,j'}
\end{equation}
and
  \begin{equation} \label{eq:d_ortho}
    \int_{0}^{\pi} d_n^{j,k}(\beta) d_{n'}^{j,k}(\beta)\, \sin(\beta) \d \beta
    =
    \frac{2}{2n+1} \delta_{n,n'}
\end{equation}
for all $n,n'\in\N_0,$ $j,k=-n,\dots,n,$ and $j',k'=-n',\dots,n'$.
The normalized rotational harmonics
\begin{equation} \label{eq:Dn}
  \widetilde D_n^{j,k}
  \coloneqq
  \sqrt{\frac{2n+1}{8\pi^2}}\, D_n^{j,k}
  ,\qquad n \in \mathbb N_0,\ j,k = -n,\ldots,n,
\end{equation} 
form an orthonormal basis
of $L^2(\SO)$.

%---------------------------------------------------------------------
\section{Sliced OT on the Sphere}\label{sec:slice_sphere}
%---------------------------------------------------------------------
We present a slicing approach for OT on the sphere $\X=\S^{d-1}$ for $d\ge3$.
We define the \emph{parallel slicing operator} for a fixed $\bpsi \in \S^{d-1}$ by
\begin{equation} \label{eq:slice}
  \cS_{\bpsi}^{\S^{d-1}} \colon \S^{d-1} \to \II ,\qquad
  \cS_{\bpsi}^{\S^{d-1}} (\zb \xi)
  \coloneqq
  \inn{\bxi, \bpsi} .
\end{equation}
We will omit the superscript of $\cS_\bpsi$ if no confusion arises.
The corresponding slice is the $(d-2)$-dimensional \emph{subsphere}
\begin{equation} \label{eq:subsphere}
  C_{\bpsi}^{t} 
  \coloneqq
  \cS_{\bpsi}^{-1}(t)
  =
  \{\zb \xi \in \S^{d-1}\mid \cS_{\bpsi} (\zb \xi) = t\}
  =
  \{\zb \xi \in \S^{d-1}\mid \inn{\bpsi,\bxi} = t\},
  \quad
  \ t \in \II,
\end{equation} 
which is the intersection of $\S^{d-1}$
and the hyperplane of $\R^d$ with normal $\bpsi$
and distance~$t$ from the origin.

In this section, we first analyze the respective Radon transform for functions and measures on $\S^{d-1}$, and then show that the sliced Wasserstein distance is a metric on $\P(\S^{d-1})$.

%----------------------------------------------------------------------
\subsection{Parallel Slice Transform of Functions} 
% ---------------------------------------------------------------------
For $f \colon \S^{d-1} \to \R$, we define the \emph{parallel slice transform}
\begin{equation} \label{eq:V}
  \U f(\bpsi,t)
  \coloneqq
  \begin{cases}
	\displaystyle  \frac{1}{\abs{\S^{d-1}} \sqrt{1-t^2}}
  \int_{C_{\bpsi}^{t}}
  f(\zb\xi)
  \, \ds (\bxi),
  &\quad \bpsi\in\S^{d-1},\ t\in (-1,1),\\
  \displaystyle \frac{1}{2}\, \delta_{d,3}\, f(\pm\bpsi), 
  &\quad \bpsi\in\S^{d-1},\ t=\pm1,
  \end{cases}
\end{equation}
where $\mathrm{ds}$ denotes the $(d-2)$-dimensional volume element on $C_{\bpsi}^{t}$
and $\delta$ is the Kronecker symbol.
{The second line in \eqref{eq:V} ensures the continuity of $\U f$ if $f$ is continuous.}
For fixed $\bpsi \in \S^{d-1}$,
the (normalized) restriction 
\begin{equation}
  \label{eq:Vs}
  \U_\bpsi f\coloneqq \abs{\smash{\S^{d-1}}} \, \U f(\bpsi,\cdot)
\end{equation}
belongs to the class of convolution operators \cite{QuHiLo18}, and
is known as the spherical section transform \cite{Rub00},
translation \cite{DaXu2013} or shift operator \cite{Rus93}.
The chosen normalization will become clear in \autoref{prop:V-as-adj}.
The following proposition was shown, e.g., in \cite[Cor.\ 3.3]{quellmalzdiss}.

\begin{proposition}[Integration in $t$]
  \label{prop:V_norm}
  For every $f \in L^1(\S^{d-1})$ and $\bpsi\in\S^{d-1}$,
  we have $\U_\bpsi \in L^1(\II)$ and
  \begin{equation} \label{eq:V_int}
    \int_{\II} \U_\bpsi f(t) \d t
    =
    \int_{\S^{d-1}} f(\zb\xi) \, \dx\sigma_{\S^{d-1}}(\zb\xi).
  \end{equation}
\end{proposition}
% \begin{proof}
%   Since $C_\bpsi^t$ is a $(d-2)$-dimensional sphere, we can express \eqref{eq:V} via an integral over a rotated sphere $\S^{d-2}$.
%   We apply to \eqref{eq:V} for $\bpsi = \be^d$ the substitution \eqref{eq:xi-decomp} and \eqref{eq:Sd_measure}, 
%   and obtain 
%   \begin{equation} \label{eq:Vf-ed}
%   \U f(\be^d,t)
%   =
%   \frac{(1-t^2)^{\frac{d-3}{2}}}{\abs{\S^{d-1}}}
%   \int_{\S^{d-2}} f\ \begin{pmatrix} \sqrt{1-t^2}\,\zb\eta\\ t\end{pmatrix} \d\sigma_{\S^{d-2}}(\zb\eta)
%   ,\qquad \forall t\in[-1,1].
%   \end{equation}
%   Let $\bpsi\in\S^{d-1}$. We choose $\bQ \in \SOd$ such that $\bQ \be^d=\bpsi$.
%   Applying \eqref{eq:Sd_measure} to the function $f\circ \bQ$, we have
%   \begin{align*}
%     \int_{\S^{d-1}} f\circ \bQ(\zb\xi) \, \dx\sigma_{\S^{d-1}}(\zb\xi)
%     =
%     \int_{-1}^{1} \int_{\S^{d-2}} f\circ \bQ \begin{pmatrix} \sqrt{1-t^2}\, \zb\eta\\ t\end{pmatrix} \d \sigma_{\S^{d-2}}(\zb\eta) \, (1-t^2)^{\frac{d-3}{2}} \d t.
%   \end{align*}
%   Using the rotational invariance of the spherical measure $\sigma_{\S^{d-1}}$ together with \eqref{eq:Vs} and \eqref{eq:Vf-ed}
%   yields the assertion.
% \end{proof}

\begin{theorem}[Positivity] \label{thm:pos}
  Let $f\in C(\S^{d-1})$.
  Then we have $f(\bxi) \ge 0$ for all $\bxi\in\S^{d-1}$
  if and only if
  $\U f(\bpsi,t) \ge 0$ for all $\bpsi\in\S^{d-1}$ and $t\in\II$.
\end{theorem}
\begin{proof}
  Let $f\in C(\S^{d-1})$.
  If $f \ge 0$ everywhere,
  then $\U f$ is non-negative everywhere as the integral of a non-negative function.
  Conversely, let $\zb\eta\in\S^{d-1}$ such that $f(\zb\eta) = -\delta < 0$.
  By continuity, there exists $\varepsilon>0$ such that $f(\zb\xi) < -\delta/2$ for all $\zb\xi\in\S^{d-1}$ with $d(\zb\xi,\zb\eta) 
  \le \varepsilon.$
  Because all points in $C_{\zb\eta}^{\cos(\varepsilon)}$ have spherical distance $\varepsilon$ to $\zb\eta$, we conclude that
  \begin{equation*}
  \U f(\zb\eta,\cos(\varepsilon))
  =
  \frac{1}{\abs{\S^{d-1}} \sin(\varepsilon)}
  \int_{C_{\zb\eta}^{\cos(\varepsilon)}}
  f(\zb\xi)
  \, \ds (\zb\xi)
  <0
  . \qedhere
  \end{equation*}
\end{proof}

There is no analogue to \autoref{thm:pos} for the vertical slice nor for the semicircle transform considered in \cite{QueBeiSte23,Bon22} since one can always construct a function that is negative on a small ball and positive outside such that either transform is non-negative everywhere.

\begin{theorem}[Singular value decomposition] \label{prop:V-inj}
  For each $n\in\NN$, let $\{ Y_{n,d}^k \mid k\in\ii{N_{n,d}} \}$ be an orthonormal basis of $\Y_{n,d}$.
  Then \eqref{eq:V} is a compact operator $\U\colon L^2(\S^{d-1}) \to L^2_{w_d}(\S^{d-1} \times \II)$, where {$L^2_{w_d}(\S^{d-1} \times \II)$ is the space of square integrable functions with the weighted norm $g\mapsto(\int_{\S^{d-1}\times\II} g(\bxi,t) (1-t^2)^{(3-d)/2} \d(\bxi,t))^{1/2}$}, %$w_d(\bxi,t)\coloneqq (1-t^2)^{(3-d)/2}$, 
  with the singular value decomposition
  \begin{equation} \label{eq:M-svd}
    \U Y_{n,d}^k (\bpsi,t)
    =
    \lambda_{n,d}^{\U}\, Y_{n,d}^k(\bpsi)\, \widetilde P_{n,d}(t)\, (1-t^2)^{\frac{d-3}{2}}
    ,\qquad \forall n\in\NN,\ k\in \ii{N_{n,d}},
  \end{equation}
  where $\widetilde P_{n,d}$ are given in \eqref{eq:Pnd-normalized} and the singular values are
  \begin{equation} \label{eq:M-sv}
    \lambda_{n,d}^{\U} \coloneqq 
    \frac{2^{\frac{d}{2}}\, \sqrt[4]{\pi}\, \sqrt{n!}\, \Gamma(\frac{d}{2})}{\sqrt{(2n+d-2)\, (n+d-3)!}}.
  \end{equation}
\end{theorem}
\begin{proof}
  The theorem basically follows from the generalized Funk--Hecke formula \cite[(4.2.10)]{BeBuPa68}, which states for $\zb\xi\in\S^{d-1}$, $t\in(-1,1)$, and $Y_{n,d} \in \Y_{n,d}(\S^{d-1})$ that
  \begin{equation*}
  \frac{1}{\abs{\S^{d-2}} ({1-t^2})^{\frac{d-2}{2}}} \int_{\left<\zb\xi,\zb\eta\right> = t} Y_{n, d}(\zb\eta) \d \sigma_{\S^{d-2}}(\zb\eta)
  = \frac{P_{n,d}(t)}{P_{n,d}(1)}\, Y_{n, d}(\zb\xi).
  \end{equation*}
  Inserting the normalization \eqref{eq:Pnd-normalized} yields \eqref{eq:M-svd} with
  \begin{equation*}
    \lambda_{n,d}^{\U}
    =
    \sqrt{\frac{\abs{\S^{d-2}}}{\abs{\S^{d-1}} N_{n,d}}}
    \overset{\eqref{eq:Sd-vol}, \eqref{eq:Nnd}}{=}
    \sqrt{\frac{\pi\, \Gamma(\frac{d}{2})\, n!\, (d-2)!}{\Gamma(\frac{d-1}{2})\, (2n+d-2)\, (n+d-3)!}}.
  \end{equation*}
  The Legendre duplication formula $\Gamma(\frac{d-1}2)\, \Gamma(\frac{d}{2}) = 2^{-d} \sqrt\pi\, \Gamma(d-1)$ yields \eqref{eq:M-sv}.
  Expanding the product \eqref{eq:Nnd} asymptotically for $n\to\infty$, we have
  \begin{equation} \label{eq:Nnd-bound}
    N_{n,d}
    =
    (2n+d-2)\,\frac{ (n+d-3)\, (n+d-2) \cdots (n+1)}{(d-2)!}
    =
    \frac{2}{(d-2)!} \left( n^{d-2} + \mathcal O(n^{d-3}) \right)
  \end{equation}
  and hence the singular values $\lambda_{n,d}^{\U}$ converge to zero.
  Together with the orthonormality of \eqref{eq:Pnd-normalized},
  we deduce that \eqref{eq:M-svd} is a singular value decomposition.
\end{proof}

\begin{theorem}[Adjoint]
  Let $1 \le p,q \le \infty$ with $1/p + 1/q = 1$.
  For $1 \le p < \infty$,
  the adjoint $\U^*\colon L^q(\S^{d-1} \times \II) \to L^q(\S^{d-1})$
  of $\U\colon L^p(\S^{d-1}) \to L^p(\S^{d-1} \times \II)$ is given by
  \begin{equation}
    \label{eq:V*}
    \U^*g(\bxi)
    = 
    \frac{1}{\abs{\S^{d-1}}}
    \int_{\S^{d-1}} g(\inn{\bxi,\bpsi}) \d\sigma_{\S^{d-1}}(\bpsi),
  \end{equation}
  and the adjoint $\U_\bpsi^*\colon L^q(\II) \to L^q(\S^{d-1})$
  of $\U_\bpsi \colon L^p(\S^{d-1}) \to L^p(\II)$ by
  \begin{equation}
    \label{eq:Vs*}
    \U_\bpsi^*g(\bxi)
    = 
    g(\inn{\bxi,\bpsi})
  \end{equation}
  for all $\bxi\in\S^{d-1}$.
  Both adjoint operators map continuous functions to continuous functions.
\end{theorem}

\begin{proof}
  Let $f\in L^p(\S^{d-1})$ and $g\in L^q(\S^{d-1}\times\II)$.
  We have
  \begin{align}
    \langle \U f,g \rangle 
    &=
      \int_{\S^{d-1}} \int_{\II}
      \U f (\bpsi,t) \, g(\bpsi,t) 
      \d t \,\d\sigma_{\S^{d-1}}(\bpsi)
    \\&
       \overset{\eqref{eq:V}}{=}
      \int_{\S^{d-1}} \int_{\II}
      \frac{1}{\abs{\S^{d-1}} \sqrt{1-t^2}}
      \int_{C_{\bpsi}^{t}} f(\bxi)
      \, g(\bpsi,t) \ds(\bxi) \d t \d\sigma_{\S^{d-1}}(\bpsi)
    \\&
      \overset{\eqref{eq:subsphere}}{=}
      \int_{\S^{d-1}} \int_{\II}
      \frac{1}{\abs{\S^{d-1}} \sqrt{1-t^2}}
      \int_{C_{\bpsi}^{t}} f(\bxi)
      \, g(\bpsi,\inn{\bpsi,\bxi}) \ds(\bxi) \d t \d\sigma_{\S^{d-1}}(\bpsi)
    \\&
      \overset{\eqref{eq:V_int}}{=}
      \frac{1}{\abs{\S^{d-1}}}
      \int_{\S^{d-1}}
      \int_{\S^{d-1}} f(\bxi)
      \, g(\bpsi,\inn{\bpsi,\bxi}) \d\sigma_{\S^{d-1}}(\bxi) \d\sigma_{\S^{d-1}}(\bpsi).
  \end{align}
  The adjoint of $\U_\bpsi$ can be established analogously without the outer integral.
  The continuity follows from
  Lebesgue's dominated convergence theorem.
\end{proof}

%------------------------------------------------------------------
\subsection{Parallel Slice Transform of Measures}
\label{sec:vert-slice-meas}
% ------------------------------------------------------------------
We extend the definition \eqref{eq:V} to measures as pushforward of the slicing operator \eqref{eq:slice}.
For $\bpsi \in \S^{d-1}$,
we define
\begin{equation}
  \label{eq:Vs_measure}
  \U_\bpsi \colon \M(\S^{d-1}) \to \M(\II),\quad
  \mu \mapsto (\cS_\bpsi)_\# \mu = \mu \circ \cS_\bpsi^{-1}.
\end{equation}
and $\U\colon \M(\S^2) \to \M(\T \times \II)$ by
\begin{equation}
  \label{def:V_measure}
  \U \mu
  \coloneqq
  T_\# (u_{\S^{d-1}} \times \mu)
  \quad\text{with}\quad
  T(\bpsi, \zb\xi) \coloneqq (\bpsi, \cS_\bpsi(\zb\xi)).
\end{equation}

\begin{proposition}[Connection with adjoint]
  \label{prop:V-as-adj}
  Let $\mu\in\M(\S^{d-1})$.
  The transforms
  \eqref{def:V_measure} and \eqref{eq:Vs_measure}
  satisfy
  \begin{equation}
  \begin{aligned}
    \label{eq:1}
    \langle \U\mu, g \rangle
    &=
      \langle \mu, \U^*g \rangle
      \quad \text{for all } g \in C(\S^{d-1} \times \II) \quad \text{and}
    \\
    \langle \U_\bpsi \mu, g \rangle
    &=
      \langle \mu, \U_\bpsi^*g \rangle
      \quad \text{for all } g \in C(\II),\, \bpsi\in\S^{d-1}
  \end{aligned}
  \end{equation}
  with the adjoint operators from \eqref{eq:V*} and \eqref{eq:Vs*}.
\end{proposition}

\begin{proof}
  For $g \in C(\S^{d-1} \times \II)$,
  we have by the definition in \eqref{def:V_measure}
  \begin{align*} 
    \langle \U\mu, g \rangle
    &=
    \int_{\S^{d-1} \times \II} 
    g(\bpsi,t) \, 
    \d T_\# (u_{\S^{d-1}} \times \mu)(\bpsi,t)
    \\
    &=
    \int_{\S^{d-1}} \int_{\S^{d-1}} 
    g(\bpsi, \inn{\bpsi,\bxi}) 
    \d u_{\S^{d-1}}(\bpsi) \d \mu(\bxi)
    = 
    \langle \mu, \U^*g \rangle.
  \end{align*}
  For $g \in C(\II)$, and fixed $\bpsi \in \S^{d-1}$,
  \begin{equation*} 
    \langle \U_\bpsi\mu, g \rangle
    =
    \int_{\II} 
    g(t) \, 
    \d (\cS_\bpsi)_\#  \mu (t)
    =
    \int_{\S^{d-1}}  
    g(\inn{\bpsi,\bxi}) 
    \d \mu(\bxi)
    = 
    \langle \mu, \U_\bpsi^*g \rangle. \qedhere
  \end{equation*}
\end{proof}

The last proposition provides an alternative way of defining $\U$ for measures, similarly to what was done for the Radon transform, e.g.\ in \cite{BomLin09}.

For absolutely continuous measures with respect to the surface measure $\sigma_{\S^{d-1}}$,
the measure- and function-valued transforms coincide.

\begin{proposition}[Absolutely continuous measures]
  \label{cor:abs1}
  For $f \in L^1(\S^{d-1})$,
  we have
  \begin{equation*}
    \U[f \sigma_{\S^{d-1}}] = (\U f) \, \sigma_{\S^{d-1} \times \II}
    \quad\text{and}\quad
    \U_\bpsi[f \sigma_{\S^{d-1}}] = (\U_\bpsi f) \, \sigma_{\II}
    \qquad \forall \bpsi\in\S^{d-1}.
  \end{equation*}
  In particular,
  the transformed measures are again absolutely continuous.
\end{proposition}

\begin{proof}
  Let $g \in C(\T\times\II)$.
  The first identity follows from \autoref{prop:V-as-adj} by
  \begin{align*}
    \langle \U[f \sigma_{\S^{d-1}}], g \rangle
    =
    \langle f \sigma_{\S^{d-1}}, \U^* g \rangle
    &=
    \int_{\S^{d-1}} f(\bxi)\, \U^* g(\bxi) \d\sigma_{\S^{d-1}}(\bxi)
    \\
    &=
    \int_{-1}^{1} \int_{\S^{d-1}} \U f(\bpsi,t)\, g(\bpsi,t) \d\sigma_{\S^{d-1}}(\bpsi) \d t
    =
    \langle (\U f)\, \sigma_{\T \times \II}, g \rangle.
  \end{align*}
  The identity for $\U_\bpsi$ follows analogously.
\end{proof}

%-------------------------------------------------
\subsection{Spherical Sliced Wasserstein Distance} \label{sec:SOT}
%-------------------------------------------------

For $p \in [1, \infty)$ and $\mu, \nu \in \P(\S^{d-1})$, the \emph{parallel-sliced spherical Wasserstein distance}
\begin{equation} 
    \label{eq:SW}
    \PSW_p^p(\mu,\nu)
    \coloneqq
    \int_{\S^{d-1}} 
    \WS_p^p( \U_{\bpsi}\mu , \U_{\bpsi}\nu ) \d u_{\S^{d-1}} (\bpsi)
\end{equation}
is the mean value
of Wasserstein distances on the unit interval $\II$.
Since the geodesic distance \eqref{eq:Sd-dist} on the sphere is rotationally invariant,
i.e., $d(\zb Q \zb\xi, \zb Q \zb\eta) = d(\zb\xi,\zb\eta)$ for all rotations $\zb Q \in \SOd$,
the spherical Wasserstein distance inherits this property, 
$\WS_p(\mu,\nu) = \WS_p(\mu\circ \zb Q,\nu\circ \zb Q)$.

\begin{theorem}[Metric] \label{thm:SSW}
  For every $p \in [1, \infty)$,
  the sliced spherical Wasserstein distance $\PSW_p$
  is a metric on $\P(\S^{d-1})$,
  which induces the same topology as the spherical Wasserstein distance $\WS_p$.
  There exist constants $c_{d,p},C_{d,p}>0$
  such that 
  \begin{equation} \label{eq:W-SW-equiv}
    c_{d,p}\PSW_p(\mu,\nu)
    \le
    \WS_p(\mu,\nu)  
    \le
    C_{d,p}
    \PSW_p(\mu,\nu)^{\frac{1}{p(d+1)}},
    \qquad \forall \mu,\nu\in \P(\S^{d-1}).
  \end{equation}
  It is
  rotationally invariant in the sense that
  for every $\mu, \nu \in \P(\S^2)$ and $\zb Q \in \SOd$, we have
  \begin{equation}
    \PSW_p(\mu, \nu) 
    =
    \PSW_p(\mu \circ \zb Q, \nu \circ \zb Q).
  \end{equation}
\end{theorem}

The proof is given in \autoref{sec:proof_SSW}.
Even though the Wasserstein $\WS_p$ and sliced Wasserstein metric $\PSW_p$ on $\P(\S^{d-1})$ are topologically equivalent,
they are not bilipschitz equivalent.
As $\S^{d-1}$ is a compact set, the $p$-Wasserstein metrics for all $p\in[1,\infty)$ induce the same topology on $\M(\S^{d-1})$, see \cite[§\,5.2]{San15}.
We conclude this section by mentioning another slicing approach.

\begin{remark}[Semicircular slices] \label{rem:SSW}
A different approach due to \cite{Bon22,Bon23}
uses slicing along semicircles
and maps to the one-dimensional torus $\T\coloneqq \R/(2\pi)$.
We mention here the case $d=3$ following \cite{QueBeiSte23}.
We denote the components of the inverse of the bijective spherical coordinate transform \eqref{eq:sph} by
$
\sph^{-1}(\bxi)
=
(\azi(\bxi), \zen(\bxi)),
$
{where $\eul$ denotes the transformation in Euler angles $\eqref{eq:Q}$.}
For $\bpsi = \Phi(\varphi,\theta)\in\S^2$, we define the slicing operator
$$
\cA_\bpsi
\colon \S^2\to \T,\ 
\bxi\mapsto \azi(\eul(\varphi,\theta,0)^\top \bxi).
$$
The semicircular sliced Wasserstein distance
\begin{equation}
    \SSW_p^p(\mu,\nu)
    \coloneqq
    \int_{\S^2}
    \WS_p^p((\cA_{\bpsi})_\# \mu, (\cA_{\bpsi})_\# \nu) \d u_{\S^2}(\bpsi)
\end{equation}
is also a rotationally invariant metric, see \cite{QueBeiSte23},
but an equivalence as in \eqref{eq:W-SW-equiv} is not known.
Furthermore, it requires solving the one-dimensional OT on the torus, which is more difficult than on an interval, see \cite{DelSalSob10,RabDelGou11,RabPeyDelBer12}.\qed
\end{remark}

%---------------------------------------------------------------------
\section{Sliced OT on SO(3)}\label{sec:SO}
%---------------------------------------------------------------------

We present an approach to generate sliced Wasserstein distances on $\X=\SO$.
We denote the angle of the rotation $\bQ\in\SO$ by
$$
\angle(\bQ)
\coloneqq
\arccos\frac{-1+\operatorname{trace}(\bQ)}{2}
\in[0,\pi].
$$
We take $\D = \SO$ and the slicing operator 
\begin{equation} \label{eq:SO_dist}
  d_\bQ\colon \SO\to [0,\pi],\quad
  \bP\mapsto
  \angle(\bQ^\top \bP)
  ,\qquad \forall\bQ\in\SO.
\end{equation}
The respective slice $d_\bQ^{-1}(\omega)$ can be parameterized as follows.

\begin{proposition} [Parameterization] \label{lem:SO_sphere}
  Let $\bQ\in\SO$ and $\omega\in [0,\pi]$.
  Then
  \begin{equation} \label{eq:SO_sphere}
    d_\bQ^{-1}(\omega)
    =
    \left\{\bA \in \SO\mid d_{\bQ}(\bA) = \omega\right\}
    =
    \{ \bQ \rR_{\bxi}(\omega) \mid \bxi\in\S^2\},
  \end{equation}
  where $\rR_{\bxi}(\omega)$ is the rotation with axis $\bxi$ and angle $\omega$, see \eqref{eq:axan}.
\end{proposition}
\begin{proof}
  We have 
  $
  \angle(\rR_{\zb\eta}(\omega)) = \omega
  $
  for all $\zb\eta\in\S^2$ and $\omega\in [0,\pi]$.
  Let $\bA\in \SO$.
  We write $\bQ^\top \bA$ in axis--angle form \eqref{eq:axan} as $\bQ^\top \bA = \rR_{\bxi}(\sigma)$ with $\sigma=\angle(\bQ^\top \bA)$ and some $\bxi\in\S^2$.
  Then we have $\bA\in d_\bQ^{-1}(\omega)$ if and only if
  $
  \omega
  =
  \angle(\bQ^\top \bA)
  =
  \sigma,
  $
  which shows the claim.
\end{proof}

%---------------------------------------------------
\subsection{A Two-Dimensional Radon Transform on SO(3)}
%---------------------------------------------------
Let $f\in L^1(\SO)$.
We define the \emph{Radon transform} on the rotation group for any $\bQ\in\SO$ and $\omega\in[0,\pi]$ by
\begin{equation} \label{eq:T}
    \cT f(\bQ,\omega)
    \coloneqq
    \frac{1}{4\pi^2}(1-\cos(\omega))\,
    \int_{\S^2} 
    f(\bQ \rR_{\bxi}(\omega)) 
    \d\sigma_{\S^2}(\bxi),
\end{equation}
and its restriction 
$\cT_\bQ f(\omega)
\coloneqq 8\pi \cT(\bQ,\omega)$.
By \autoref{lem:SO_sphere}, the domain of integration of $f$ is the slice $d_\bQ^{-1}(\omega)$.
With \eqref{eq:axan-int}, we obtain 
\begin{equation}
  \int_{\SO} f(\bA) \d\sigma_{\SO}(\bA)
  = 
  \int_{0}^{\pi}
  \cT_\bQ f(\omega)
  \d\omega
  ,\qquad \forall f\in L^1(\SO),
\end{equation}
which serves as analogue to \eqref{eq:V_int}.
By Fubini's theorem, the last equation implies that
$\cT_\bQ f \in L^1([0,\pi])$ and 
$\cT f \in L^1(\SO\times [0,\pi])$
if $f\in L^1(\SO)$.

\begin{theorem}[Singular value decomposition] \label{thm:svd}
  The Radon transform \eqref{eq:T} is one-to-one
  $$
  \cT \colon
  L^2(\SO) \to 
  L^2(\SO\times[0,\pi])
  $$
  and has
  the singular value decomposition
  \begin{equation} \label{eq:svd}
    \cT \widetilde D_n^{j,k}
    =
    \lambda_n^{\cT}\,
    F_n^{j,k}
    ,\qquad \forall n\in\NN,\, j,k\in\{-n,\dots,n\}
  \end{equation}
  with the orthonormal basis $\widetilde D_n^{j,k}$ of $L^2(\SO)$, see \eqref{eq:Dn},
  the singular values
  \begin{equation} 
    \lambda_0^{\cT}
    =
    \sqrt{\frac32}\, \pi^{-1/2}
    \quad\text{and}\quad
    \lambda_n^{\cT}
    \coloneqq
    \frac{1}{(2n+1)\sqrt{\pi}}
    \qquad \forall n\in\N,
  \end{equation}
  and the set of orthonormal functions on $L^2(\SO\times[0,\pi])$ defined by
  $$
  F_n^{j,k}(\bQ,\omega)
  \coloneqq
  \begin{cases}
    \tfrac2{\sqrt\pi}\, \widetilde D_n^{j,k}(\bQ)\, \sin\!\left((n+\tfrac12)\omega\right)\, \sin(\tfrac\omega2),
    & n\neq0,
    \\
    \sqrt{\tfrac{8}{3\pi}}\, \widetilde D_0^{0,0}(\bQ) \left( \sin(\tfrac\omega2) \right)^2, & n=0,
  \end{cases}
  \qquad (\bQ,\omega)\in\SO\times[0,\pi].
  $$
\end{theorem}

The proof is postponed to \autoref{sec:proof_T}.
Restricting $\cT(\cdot,\omega)$ to a fixed radius $\omega$, we obtain the following injectivity result, which is of a similar structure as for the spherical cap \cite{Ung54} or spherical slice transform \cite{Sch69} on $\S^2$.

\begin{corollary}[Injectivity for fixed $\omega$]
  For fixed $\omega\in(0,\pi)$,
  the Radon transform $\cT(\cdot,\omega)$
  is injective as operator $L^2(\SO)\to\ L^2(\SO)$ if and only if
  $ \omega / \pi = p/q$ 
  where $p/q$ is a reduced fraction
  and $p\in\N$ is even and $q\in\N$ is odd.
\end{corollary}
\begin{proof}
  As a direct consequence of \autoref{thm:svd} and the orthonormality of the rotational harmonics,
  the restricted transformation has the eigenvalue decomposition
  \begin{equation} \label{eq:TD}
    \cT \widetilde D_n^{j,k}(\bQ,\omega)
    =
    \frac{2}{(2n+1)\pi}\,
    {\sin((n+\tfrac12) \omega)}\, {\sin(\tfrac\omega2)}\,
    \widetilde D_n^{j,k}(\bQ)
    ,\qquad \forall \bQ\in\SO.
  \end{equation}
  It is injective if and only if all eigenvalues
  $
  \frac{2}{(2n+1)\pi}
  \sin((n+\tfrac12)\omega)
  \sin(\tfrac\omega2)
  $
  are non-zero.
  The $n$-th eigenvalue
  is zero if and only if
  $(n+1/2)\omega \in \pi\N$,
  i.e., there exists $k\in\N$ such that 
  $\omega/\pi = 2k / (2n+1)$,
  the quotient of an even and an odd integer.
  Since this fraction can only be reduced with an odd factor, 
  also the reduced fraction consists of an even divided by an odd integer.
\end{proof}

\begin{theorem}[Adjoint] \label{thm:adjoint}
  Let $\bQ\in\SO$.
  The adjoint of $\cT_{\bQ} \colon L^2(\SO) \to L^2([0,\pi])$ is given by
  $$
  \cT_{\bQ}^* g(\bA)
  =
  g(d_\bQ(\bA))
  ,\qquad \forall \bA\in\SO,
  $$
  and the adjoint of 
  $\cT \colon L^2(\SO) \to L^2(\SO\times[0,\pi])$
  is 
  $$
  \cT^*(\bA)
  =
  \frac{1}{8\pi^2}
  \int_{\SO} g(\bQ,d_\bQ(\bA))
  \d \sigma_{\SO}(\bQ),
  \qquad \forall \bA\in\SO.
  $$
\end{theorem}
\begin{proof}
  Let $f\in C(\SO)$, $g\in C([0,\pi])$, and $\bQ\in\SO$.
  We have with the substitution $\bA = \bQ \rR_{\bxi}(\omega)$ and \eqref{eq:axan-int}
  \begin{align*}
    \int_{\SO}
    f(\bA)
    g(d_\bQ(\bA))
    \d\sigma_{\SO} (\bA)
    &=
    2
    \int_{0}^{\pi}
    \int_{\S^2}
    f(\bQ\rR_{\bxi}(\omega))
    g(\omega)
    (1-\cos(\omega))
    \d\sigma_{\S^2} (\bxi) \d\omega 
    \\& 
    =
    \int_{0}^{\pi}
    \cT_\bQ f(\omega)
    g(\omega)
    \d\omega .
  \end{align*}
  The second claim follows analogously by considering $g\in C(\SO\times[0,\pi])$ and integration over $\bQ\in\SO$.
\end{proof}

\subsection{Slicing of Measures}
We generalize the definition \eqref{eq:T} to a measure $\mu\in\M(\SO)$ 
via the pushforward of the slicing operator \eqref{eq:SO_dist}, i.e.,
\begin{align}
\cT_\bQ\mu &\coloneqq (d_{\bQ})_\#\mu \in \P([0,\pi])
\quad\text{and} 
\\ \label{eq:Tmu}
  \cT\mu
  &\coloneqq
  T_\#(u_{\SO} \times\mu)
  \in\P(\SO\times[0,\pi]),
  \qquad \text{with} \quad
  T(\bQ,\bA) 
  =
  (\bQ, d_\bQ(\bA)).
\end{align}

\begin{proposition}[Connection with adjoint] \label{thm:T_adj_dual}
  Let $\bQ\in\SO$ and $\mu\in\M(\SO)$.
  Then
  \begin{align}
    \inn{\cT_\bQ \mu, g}
    &=
    \inn{\mu, \cT_\bQ^*g}
    ,\qquad \forall g\in C([0,\pi]),
    \\
    \inn{\cT \mu, g}
    &=
    \inn{\mu, \cT^*g}
    ,\qquad \forall g\in C(\SO\times[0,\pi]).
  \end{align}
\end{proposition}

\begin{proposition}[Absolutely continuous measures]
  Let $\mu\in\M(\SO)$ be absolutely continuous,
  i.e., there exists a density function $f\in L^1(\SO)$ such that 
  $\mu = f \sigma_{\SO}$.
  Then
  \begin{align*}
    \cT_\bQ \mu
    &=
    (\cT_\bQ f) \sigma_{[0,\pi]}, 
    \quad \text{and}
    \\
    \cT \mu
    &=
    (\cT f) \sigma_{\SO\times[0,\pi]}.
  \end{align*}
\end{proposition}

The last two propositions can be proven analogously to Propositions \ref{prop:V-as-adj} and \ref{cor:abs1}.

\begin{theorem}[Injectivity] \label{thm:T_inj}
  The Radon transform
  $\cT \colon \M(\SO)\to \M(\SO\times [0,\pi])$ 
  is injective. 
\end{theorem}

The proof, which uses the singular value decomposition, is given in \autoref{sec:proof_T}.
Although the slicing operator and therefore the transform $\cT$ could be easily generalized to $\SOd$, this does not apply for the singular value decomposition \eqref{eq:svd} as the harmonic analysis on $\SOd$ becomes much more evolved for $d>3$, cf.\ \cite[Chapter IX]{Vil68}.

\subsection{Sliced Wasserstein Distance on SO(3)}\label{sec:sosw}

Let $p\in[1,\infty)$.
We define the sliced Wasserstein distance between measures $\mu,\nu\in \P(\SO)$ by
\begin{equation} \label{eq:SW_SO}
  \SOSW_p^p(\mu,\nu)
  \coloneqq
  \int_{\SO} \WS_p^p\big(\cT_{\bQ} \mu, \cT_{\bQ} \nu\big) \d u_{\SO}(\bQ),
\end{equation}
which is the mean of Wasserstein distances on the interval $[0,\pi]$.

\begin{theorem}
  Let $p\in[1,\infty)$.
  The sliced Wasserstein distance \eqref{eq:SW_SO} is a metric on $\P(\SO)$ that is invariant to rotations, i.e., 
  for any $\bA\in\SO$ and $\mu,\nu\in\P(\SO)$, we have 
  $$
  \SOSW_p^p(\mu(\bA\cdot), \nu(\bA\cdot))
  =
  \SOSW_p^p(\mu, \nu).
  $$
\end{theorem}
\begin{proof}
  The positive definiteness is due to \autoref{thm:T_inj},
  while the symmetry and triangular inequality follow from the respective properties of the Wasserstein distance on $[0,\pi]$.
  By definition in \eqref{eq:SW_SO}, we have
  \begin{equation*} 
    \SOSW_p^p(\mu(\bA\cdot), \nu(\bA\cdot))
    =
    \int_{\SO} \WS_p^p\big( \mu\circ \bA\circ d_\bQ^{-1}, \nu\circ \bA\circ d_\bQ^{-1}\big) \d u_{\SO}(\bQ).
  \end{equation*}
  Let $\omega\in[0,\pi]$.
  We have 
  $\bB\in \bA\circ d_\bQ^{-1}(\omega)$
  if and only if $\omega = d_\bQ(\bA^\top \bB) = d_{\bA\bQ}(\bB)$,
  cf.\ \eqref{eq:SO_dist}.
  Hence
  \begin{equation*} 
    \SOSW_p^p(\mu(\bA\cdot), \nu(\bA\cdot))
    =
    \int_{\SO} \WS_p^p\big((d_{\bA\bQ})_\# \mu, (d_{\bA\bQ})_\# \nu\big) \d u_{\SO}(\bA\bQ)
    =
    \SOSW_p^p(\mu, \nu). \qedhere
  \end{equation*}
\end{proof}

In \autoref{sec:SO-S3}, we provide a relation between the sliced Wasserstein distances on $\SO$ and on $\S^3$.

%-------------------------------------------------
\section{Barycenter algorithms} \label{sec:bary}
%-------------------------------------------------

There exist two approaches to compute barycenters of measures using 1D Wasserstein distances along projected measures, namely sliced Wasserstein barycenters and Radon Wasserstein barycenters, cf.\ \cite{BonRabPeyPfi15}.
We adapt them to our slicing  in Section \ref{sec:bary-sliced} and \ref{sec:bary-radon}, respectively.
 
\subsection{Sliced Wasserstein barycenters}
\label{sec:bary-sliced}
For {sliced Wasserstein barycenters}, we replace in \eqref{eq:W-bary} the Wasserstein distance $\WS_2$ by its sliced counterpart.
In particular, with the general notion of slicing in \eqref{eq:SW-general},
the \emph{sliced Wasserstein barycenter}
of given measures $\mu_i\in\P(\X)$, $i\in\ii{M}$, and $\blambda\in\Delta_M$,
is defined by
\begin{equation} \label{eq:PSW-bary}
    \bary^{\SW}_{\X}(\mu_i, \lambda_m)_{i=1}^{M}
    \coloneqq
    \argmin_{\nu\in\P(\X)}
    \sum_{i=1}^{M}
    \lambda_i\, \SW_2^2(\nu,\mu_i).
\end{equation}

\begin{remark} \label{rem:antipo_diracs}
    Although the different slicing approaches often yield similar barycenters, as we will see in the numerics,
    they differ considerably in the extreme case of two antipodal point measures on the sphere $\S^2$.
    Denote by $\delta_{\bxi}$ the Dirac measure at $\bxi\in\S^2$.
    The Wasserstein barycenter $\bary_{\S^{2}}^{\WS}$ of $\delta_{\be^3}$ and $\delta_{-\be^3}$ with equal weights $\lambda_1=\lambda_2=1/2$ consists of the measures $\nu\in\P(\S^2)$ with support on the equator. 
    However, all measures in $\P(\S^2)$ are parallelly sliced Wasserstein barycenters $\bary_{\S^2}^{\PSW}$.
    For the semicircular slices of \autoref{rem:SSW},
    we can show that the uniform distribution on the equator is a candidate for the barycenter $\bary_{\S^{2}}^{\SSW}$, while $u_{\S^2}$ is not.
    The details are provided in \autoref{sec:particular}.\qed
\end{remark}

We consider two types of discretization to compute sliced Wasserstein barycenters. Free-support barycenters are based on a Lagrangian discretization: measures are represented by samples, and the minimization is carried out over the coordinates of those samples, see \cite{RabPeyDelBer12,BonRabPeyPfi15,PeyCut19}. Fixed-support barycenters are based on a Eulerian discretization: a fixed grid is considered for all the measures which are represented by the weights given to each grid point, and the minimization is carried out over those weights, cf.\ \cite{BLNS22,Bor22}.

\subsubsection{Free-support discretization}\label{sec:disc-free}

For ${X = (X_k)_{k=1}^{N} \in \X^N}$, 
we note $\mu_X \coloneqq \frac{1}{N}\sum_{k = 1}^N\delta_{X_k}$. 
We consider $M$ discrete measures $\mu_{Y^{(i)}} \in \P(\X)$ with $Y^{(i)} \in \X^N$ for all $i\in \ii{M}$. The aim is to compute a discrete barycenter of these measures, i.e.,
we want to minimize the functional
\begin{equation} \label{eq:SW-bary}
    \cE\colon \X^ N \to \R,\qquad X \mapsto \displaystyle\sum_{i \in \ii{M}} \lambda_i \SW_2^2(\mu_X, \mu_{Y^{(i)}}),
\end{equation}
which is not convex in general,
via a stochastic gradient descent,
as it has been applied in the Euclidean space $\X=\R^d$ in \cite[sect.\ 3.2]{RabPeyDelBer12} and \cite[sect.\ 10.4]{PeyCut19}.
Since $X$ is on a manifold $\X$, the gradient is in its tangent space $T_x\X$ at $x\in \X$.
By Whitney’s embedding theorem \cite[thm.\ 6.15]{Lee12},
we can assume $\X$ to be embedded in Euclidean space $\R^d$ for sufficiently large $d$.
If $f\colon D\to\R$ is differentiable on an open set $D\subset\R^d$ with $D\supset \X$,
then the Riemannian gradient of the restriction of $f$ to $\X$ is given by $\nabla f(x) = \proj_{T_x\X} (\nabla_{\R^d} f(x))$,
where $\nabla_{\R^d}$ denotes the gradient in Euclidean space and
$\proj_{T_x\X}$ the orthogonal projection to the tangent space, see \cite[sect.\ 3.6.1]{ABS2008}.
The gradient of the functional \eqref{eq:SW-bary} can be expressed as follows.
    
\begin{theorem}
    Let $\X$ be a smooth submanifold of $\R^d$
    and $X, Y\in\X^N$ consist of pairwise distinct points $X_k$, $k\in \ii{N}$. 
    Assume that the slicing operator $\cS_\bpsi\colon \X\to\R$ is differentiable for all $\bpsi \in \D$,
    that $(\bx, \bpsi) \mapsto \cS_\bpsi(\bx)$ is bounded on $\X\times \D$, and that $\bpsi \mapsto \nabla \cS_\bpsi(\bx)$ is integrable on $\D$ uniformly for every $\bx\in\X$. 
    Furthermore, assume that $\cS_\bpsi(\bx_1)\neq\cS_\bpsi(\bx_2)$ for $u_\D$-almost every $\bpsi\in\D$ if $\bx_1\neq \bx_2$.
    Then the gradient of the sliced Wasserstein distance between the two measures $\mu_X$ and $\mu_Y$ with respect to $X$ reads 
    \begin{equation} \label{eq:grad-SW}
        \nabla_{X_k}\SW_2^2(\mu_X, \mu_Y) 
        = \frac{2}{N}\int_{\D} 
            \left[
                \cS_\bpsi\left(X_k\right) - \cS_\bpsi\left( Y_{\sigma_{Y, \bpsi}\circ \sigma_{X, \bpsi}^{-1}(k)} \right)
            \right] \nabla \cS_\bpsi(X_k) \d u_{\D}(\bpsi),
    \end{equation}
    where $\sigma_{X,\bpsi}\colon \ii{N}\to\ii{N}$ is a permutation which sorts $\cS_\bpsi(X_k)$, i.e.,
    $$\cS_\bpsi(X_{\sigma_{X, \bpsi}(1)}) \leq \cS_\bpsi(X_{\sigma_{X, \bpsi}(2)}) \leq ... \leq \cS_\bpsi(X_{\sigma_{X, \bpsi}(N)}).$$ 
\end{theorem}

\begin{proof}
    Let $\bpsi\in \D$. For the sake of simplicity, we write $\cS_\bpsi(X) \coloneqq (\cS_\bpsi(X_k))_{1\leq k \leq N}\in \R^N$.
    The pseudo-inverse of the cumulative density function of $\mu_{\cS_\bpsi(X)}$ is written, for $r\in(0, 1)$,
    \begin{equation}
    \begin{split}
        F_{\mu_{\cS_\bpsi(X)}}^{-1} (r)
        &= \min\left\{x\in\R\cup\{-\infty\}\ \left|\ \frac{1}{N}\sum_{k = 1}^N\mathds{1}_{\left[\cS_\bpsi(X_k), 1\right]}(x) \geq r\right.\right\}\\
        &= \cS_\bpsi\left(X_{\sigma_{X, \bpsi}(\lceil rN\rceil)}\right).
    \end{split}
    \end{equation}
    Then
    \begin{align*}
        \WS_2^2\big((\cS_\bpsi)_\#\mu_X, (\cS_\bpsi)_\#\mu_Y\big)
        &= \WS_2^2\left(\mu_{\cS_\bpsi(X)}, \mu_{\cS_\bpsi(Y)}\right)\\
        &= \int_{[0, 1]}\left|F_{\mu_{\cS_\bpsi(X)}}^{-1}(r) - F_{\mu_{\cS_\bpsi(Y)}}^{-1}(r)\right|^2 \d r\\
        &= \sum_{k = 1}^N \frac{1}{N}\Big|\cS_\bpsi\left(X_{\sigma_{X, \bpsi}(k)}\right) - \cS_\bpsi\left(Y_{\sigma_{Y,\bpsi}(k)}\right)\Big|^2 \\
        &= \frac{1}{N}\sum_{k = 1}^N \left|\cS_\bpsi\left(X_k\right) - \cS_\bpsi\left( Y_{\sigma_{Y, \bpsi}\circ \sigma_{X, \bpsi}^{-1}(k)}\right)\right|^2,
    \end{align*}
    is bounded and, thus, integrable with respect to $\bpsi$ on $(\D, u_\D)$. Its gradient is given by
    \begin{equation}\label{eq:grad-W}
        \nabla_{X_k}\WS_2^2\big((\cS_\bpsi)_\#\mu_X, (\cS_\bpsi)_\#\mu_Y\big) = \left[\cS_\bpsi\left(X_k\right) - \cS_\bpsi\big(Y_{\sigma_{Y, \bpsi}\circ \sigma_{X, \bpsi}^{-1}(k)}\big)\right]\nabla \cS_\bpsi(X_k)
    \end{equation}
    for $u_\D$-almost every $\bpsi \in \D$. Indeed, we assume that $X_k$, $k\in \ii{N}$, are pairwise distinct, so  $\cS_\bpsi(X_k)$ are $u_\D$-almost surely pairwise distinct, so $\sigma_{X, \bpsi}$ is $u_\D$-almost surely uniquely defined, and constant in the neighborhood of $X$. Hence, \eqref{eq:grad-W} is integrable on $\D$ since $\nabla \cS_\psi$ is and the rest is bounded by assumption. 
    Therefore, similarly to \cite{BonRabPeyPfi15}, we have
    \begin{align}
        \SW_p^p(\mu_Y, \mu_X) 
        &= \int_\D \WS_p^p\big((\cS_\bpsi)_\#\mu_X, (\cS_\bpsi)_\#\mu_Y\big)\d u(\bpsi)\\
        &= \frac{1}{N}\int_\D \sum_{k = 1}^N \left|\cS_\bpsi\left(X_k\right) - \cS_\bpsi\left( Y_{\sigma_{Y, \bpsi}\circ \sigma_{X, \bpsi}^{-1}(k)}\right)\right|^p\d u(\bpsi)
        \intertext{and}
        \nabla_{X_k}\SW_2^2(\mu_Y, \mu_X)
        &= \int_\D \nabla_{X_k}\WS_2^2\left(\mu_{\cS_\bpsi(Y)}, \mu_{\cS_\bpsi(X)}\right)\d u(\bpsi) \label{eq:int_grad}\\
        &= \frac{2}{N}\int_\D \left[\cS_\bpsi\left(X_k\right) - \cS_\bpsi\left(Y_{\sigma_{Y, \bpsi}\circ \sigma_{X, \bpsi}^{-1}(k)}\right)\right]\nabla \cS_\bpsi(X_k) \d u(\bpsi). \qedhere\label{eq:grad_sig}
    \end{align}
\end{proof}
    
We discretize \eqref{eq:grad-SW} over $\bpsi$ via considering $(\bpsi_q)_{q=1}^{P}$ distributed according to the uniform measure $u_\D$ to get a numerical approximation. 
This enables us to devise a stochastic gradient descent algorithm with initialization $X^0\in\X^N$ and whose step $l\in\N$ is given by
\begin{equation}\label{eq:sgd}
\begin{split}
    X_k^{l+1}
    &= \exp^{\X}_{X_k^l}(- \tau_l\nabla\cE(X^l)_k)\\
    &= \exp^{\X}_{X_k^l} \left(- \tau_l
        \sum_{i = 1}^M\frac{2\lambda_i}{N P}
            \sum_{q = 1} ^P
                \left[
                    \cS_{\bpsi_q}\big(X^l_k\big)
                    - \cS_{\bpsi_q}\left(Y^{(i)}_{\sigma_{Y^{(i)}, {\bpsi_q}} \circ \sigma_{X^l, {\bpsi_q}}^{-1}(k)}\right) 
                \right]\nabla \cS_\bpsi(X_k^l)
        \right)
\end{split}
\end{equation}
for every $k\in\ii{N}$,
where $\exp^{\X}_x$ denotes exponential map, which maps a subset of the tangent space $T_x\X$ to the manifold $\X$, and $\tau_l>0$ is the step size, also known as the learning rate.

\begin{remark*}\label{rem:generalization_diffnpts}
    The last theorem can be extended to the case where both measures have different numbers of points. 
    The squared Wasserstein distance between $\mu_x$ and $\mu_y$, with $x = (x_i)_{1\leq i \leq N} \in \R^N$ and $y = (y_j)_{1\leq j\leq M} \in \R^M$ two sorted lists of real numbers, is not anymore $\frac{1}{N}\|x - y\|_{\R^N}^2$ but is of the shape $\sum_{i,j}\pi_{i,j}(x_i - y_j)^2$ with the transport plan $\pi \in \R^{N\times M}$ that depends only on $N$ and $M$ (not $x$ nor $y$). Because the matrix $\pi$ is sparse with support close to the diagonal $\frac{i}{j}\approx \frac{N}{M}$, we can generalize our algorithm while keeping its complexity $\mathcal O((N+M)\log(N+M))$.\qed
\end{remark*}

\paragraph{Application to $\PSW$ on the sphere} 
We look at the stochastic gradient descent step for the sphere with the parallel slicing operator  \eqref{eq:slice}.
Let $\bx,\bpsi\in\S^{d-1}$.
The projection to the tangential plane $T_\bx{\S^{d-1}} = \{\bv\in \R^d\mid \inn{\bv, \bx} = 0\}$ is 
\begin{equation}
    \proj_{T_\bx{\S^{d-1}}}\colon \R^d\to T_\bx{\S^{d-1}},\ \bv\mapsto \bv - \inn{\bx, \bv}\bx,
\end{equation}
and we have $\nabla\cS_\bpsi(\bx) = \proj_{T_\bx{\S^{d-1}}}(\bpsi)$.
As a consequence, the stochastic gradient descent step \eqref{eq:sgd} is 
\begin{equation} \label{eq:free-psb-sgd}
    X_k^{l+1} = \exp^{\S^{d-1}}_{X_k^l}\left[- \tau_l \proj_{T_{X_k^l}{\S^{d-1}}}\left(
        \sum_{i = 1}^M\frac{2\lambda_i}{N P}
            \sum_{q = 1} ^P
                \inn{\bpsi_q,\ 
                     X^l_k - Y^{(i)}_{\sigma_{Y^{(i)}, {\bpsi_q}}\circ\sigma_{X^l, {\bpsi_q}}^{-1}(k)}}
            \bpsi_q\right)
        \right]
\end{equation}
with the exponential map $\exp^{\S^{d-1}}_\bx (\bv) \coloneqq \cos(\|\bv\|)\bx + \sin(\|\bv\|)\frac{\bv}{\|\bv\|}$.

The numerical complexity of this step is $\mathcal O(MN(\log(N)+d)P)$ with $M$ the number of measures, $N$ the number of points in each measure and $P$ the number of directions (i.e.\ of projections or slices), since the sorting has complexity $\mathcal O(N\log(N))$ and the dimension $d$ comes in only due to the computation of the inner product in $\R^d$ and the generation of uniform samples on $\S^{d-1}$. 

\paragraph{Application to $\SOSW$ on the rotation group}
\label{sec:disc-so3} 
Let us now look at {the} case $\X = \SO$ with the slicing operator $\cS^{\SO}_\bpsi(\bR) 
=\trace(\bR^\top\bpsi)$ for $\bpsi \in\D = \SO$ and $\bR \in \SO$.
{In order to avoid numerical instabilities due to the unboundedness of the derivative of the arccosine, we take here} a monotone transformation of the slicing operator \eqref{eq:SO_dist} in order to simplify computation while keeping the properties of $\SOSW$.
{Drawing random directions $\bpsi=\eul(\alpha,\beta,\gamma)\in\SO$ can be done by randomly generating Euler angles $\alpha,\gamma\sim u_{[0,2\pi]}$ and $\beta=\arccos t$ with $t\sim u_{[-1,1]}$.}
The gradient descent step \eqref{eq:sgd} is similar to the case of the sphere. 
The tangent space is ${T_\bR\SO = \{\bA \in \R^{3\times 3}\mid \bR^\top\bA = - \bA^\top\bR\}}$. Utilizing that the orthogonal projection on the set of skew-symmetric matrices is given by $\bA\mapsto \frac{1}{2}(\bA - \bA^\top)$, one can show similarly that the projection to ${T_\bR \SO}$ is given by
\begin{equation}
    \proj_{T_\bR \SO}\colon \R^{3\times3}\to T_\bR \SO,\quad \bA\mapsto \tfrac{1}{2}(\bA - \bR\bA^\top\bR).
\end{equation}
The exponential map is given by  \cite[3.37]{War83}
\begin{equation}
    \exp^{\SO}_\bR\colon  T_\bR \SO\to \SO,\quad \bA\mapsto \bR\,\exp(\bR^\top \bA).
\end{equation}
In order to avoid the computation of the matrix exponential, one could replace the exponential map by a retraction, see \cite[Sect.\ 4.1]{ABS2008}.

\subsubsection{Fixed-support discretization}\label{sec:disc-fixed}
Fixed-support sliced Wasserstein barycenters correspond to a Eulerian discretization. As opposed to \autoref{sec:disc-free}, the support is the same, fixed set for all measures (including the barycenter), and we minimize over the weights of the Dirac measures. 
We first study the 1-dimensional case of fixed-support OT.

\begin{theorem} \label{thm:fixed-1d}
    Let $\{t_j\mid j\in\ii{N}\}\subset\R$ with $t_1< t_2< ... < t_N$.
    Further, let $\bw, \bv \in \Delta_N $
    and the discrete probability measures $\mu_\bw = \sum_{j = 1}^N w_j\delta_{t_j}$ and $\mu_\bv = \sum_{j = 1}^N v_j\delta_{t_j}$.
    We introduce the partial sums $\tilde \bw = (\sum_{j = 1}^k w_j)_{k=1}^{N}$ and $\tilde \bv = (\sum_{j = 1}^k v_j)_{k=1}^{N}$ in $\R^N$
    as well as the vectors $\by=(\tilde w_1, ..., \tilde w_{N-1}, \tilde v_1, ..., \tilde v_{N-1}) \in \R^{2N-2}$ and $\bz = (u_{\sigma(j)})_{j=1}^{2N-2}$ with $\sigma$ a permutation such that $u_{\sigma(1)} \leq u_{\sigma(2)} \leq ... \leq u_{\sigma(2N-2)}$.
    We set $$a_j = \big| \smash{ t_{\min\{k\mid \tilde w_k \geq z_{j+1}\}} - t_{\min\{k\mid \tilde v_k \geq z_{j+1}\}}}\big|^p\quad \text{for } j \in \ii{2N-3},$$ and $a_0 = a_{2N-2} = 0$. 
    Then the gradient of $\WS_p^p(\mu_\bw, \mu_\bv)$ in $\Delta_N$ with respect to $\bw$ is given almost everywhere 
    by 
    \begin{equation} \label{eq:fixed-1d-grad}
        \nabla_{\bw} \WS_p^p(\mu_\bw, \mu_\bv) = \proj_H \left( \left(\sum_{k = j}^{N-1} \big(a_{\sigma^{-1}(k) - 1} - a_{\sigma^{-1}(k)}\big) \right)_{j=1}^N 
        \right),
    \end{equation}
    where $\proj_H\colon \R^N\to H$, $\bx\mapsto \bx - \inn{\bx,\mathds1}\mathds1$ is the orthogonal projection on the hyperplane $H = \{ \bx\in \R^N\mid \inn{\bx, \mathds{1}} = 0\}$.
\end{theorem}

\begin{proof}
    The pseudo-inverse of the cumulative distribution function of $\mu_\bw$ is 
    $$F_{\mu_\bw}^{-1}(r) = \min\left\{s\in\R \left|\ \sum_{i = 1}^N w_i\mathds{1}_{[t_i, +\infty)}(s) \geq r\right.\right\} = t_{\min\{k\mid \tilde w_k \geq r\}}, \quad \forall r \in (0, 1),$$ 
    and analogously for $\bv$.
    We thus have
    \begin{align*}
        \WS_p^p(\mu_\bw, \mu_\bv) 
        &= \int_0^1 \left| F_{\mu_\bw}^{-1}(r) - F_{\mu_\bv}^{-1}(r)\right|^p \d r\\
        &= \int_0^1 \left| t_{\min\{k\mid \tilde w_k \geq r\}} - t_{\min\{k\mid \tilde v_k \geq r\}}\right|^p \d r\\
        &= \sum_{j = 1}^{2N - 3} \left| t_{\min\{k\mid \tilde w_k \geq z_{j+1}\}} - t_{\min\{k\mid \tilde v_k \geq z_{j+1}\}}\right|^p (z_{j+1} - z_{j})\\
        &= \sum_{j = 1}^{2N - 3} a_j (z_{j+1} - z_{j}),
    \end{align*}
    where we have ignored the segment $(0, z_1]$ on which the integrand is 0, and the segment $[z_{2N-2}, 1)$ on which the integrand is 0 or which is empty if $z_{2N-2} = 1$. 
    
    We extend the map $\bw\mapsto \WS_p^p(\mu_\bw, \mu_\bv)$ from $\Delta_N$ to some neighborhood in the Euclidean space $\R^N$ by making it constant in the last component $w_N$, then  $\frac{\mathrm d}{\mathrm d {w_N}} \WS_p^p(\mu_\bw, \mu_\bv) = 0.$
    For $k \in \ii{ N-1}$, since $z_j = \tilde w_{\sigma(j)}$, we have 
    \begin{equation*}
        \frac{\mathrm d}{\mathrm d \tilde w_k} \WS_p^p(\mu_\bw, \mu_\bv) = a_{\sigma^{-1}(k) - 1} - a_{\sigma^{-1}(k)}, \quad \text{a.e.}
    \end{equation*}
    Indeed, $a_j$ is almost everywhere constant with respect to $\tilde w_k$. It is justified using the fact that the $z_j$ are almost surely pairwise distinct. So, if $\tilde w_k = z_{j+1}$, then $z_{j+1}$ ``moves'' (a.s.) with $\tilde w_k$, such that $\{k'\mid \tilde w_{k'} \geq z_{j+1}\}$ does not change. If $\tilde w_k \neq z_{j+1}$, it is obvious that $\{k'\mid \tilde w_{k'}\geq z_{j+1}\}$ is locally constant with respect to $\tilde w_k$.    
    The definition of $\tilde\bw$ gives a bijection between $(w_j)_{j=1}^{N-1}$ and $(\tilde w_k)_{k=1}^{N-1}$, so
    we obtain
    for any $j\in\ii{N-1}$
    \begin{equation*}
        \frac{\mathrm d}{\mathrm d {w_j}} \WS_p^p(\mu_\bw, \mu_\bv) 
        = \sum_{k = 1}^{N} \frac{\mathrm d \tilde w_k}{\mathrm d w_j} \frac{\mathrm d}{\mathrm d {\tilde w_k}} \WS_p^p(\mu_\bw, \mu_\bv)
        = \sum_{k = j}^{N-1} \frac{\mathrm d}{\mathrm d \tilde w_k} \WS_p^p(\mu_\bw, \mu_\bv),
    \end{equation*}
    The projection onto the hyperplane $H$ yields the Riemannian gradient in $\Delta_N$.
\end{proof}

We now come to the sliced OT on $\X\in\{\S^{d-1}, \SO\}$. 
We consider a family $X = (\bx_j)_{j=1}^{N} \in \X^N$ of points on $\X$ representing the fixed support and measures of the form
\begin{equation*} \label{eq:fixed}
    \mu_\bw = \sum_{j = 1}^ N w_j \delta_{x_j}
\end{equation*}
with some weight vector $\bw \in \Delta_N$.
Let $\bv^{(i)} \in \Delta_N$ for $i \in \ii{M}$ be given probability vectors of the measures whose barycenter we want to compute.
Therefore, we minimize 
\begin{equation*}
    \cE\colon \Delta_N \to \R, \qquad \bw \mapsto \displaystyle\sum_{i =1}^ M \lambda_i \SW_p^p(\mu_\bw, \mu_{\bv^{(i)}}).
\end{equation*}
We have, for $i \in \ii{M}$, 
\begin{align*}
\SW_p^p(\mu_w, \mu_{v^{(i)}}) 
&= \int_\D \WS_p^p\big((\cS_\bpsi)_\#\mu_\bw, (\cS_\bpsi)_\#\mu_{\bv^{(i)}}\big)\d u_\D(\bpsi) \\
&= \int_\D \WS_p^p\left(\sum_{j = 1}^N w_j\delta_{\cS_\bpsi(\bx_j)}, \sum_{j = 1}^N v^{(i)}_j\delta_{\cS_\bpsi(\bx_j)}\right) \d u_\D(\bpsi),
\end{align*}
and therefore
\begin{equation*}
    \nabla_\bw\SW_p^p(\mu_\bw, \mu_{\bv ^{(i)}}) = \int_\D \nabla_\bw \WS_p^p\left(\sum_{j = 1}^ N w_j\delta_{\cS_\bpsi(\bx_j)}, \sum_{j = 1}^ N v_j^{(i)}\delta_{\cS_\bpsi(\bx_j)}\right) \d u_\D(\bpsi).
\end{equation*}
As the integrand of the last equation is handled in \autoref{thm:fixed-1d},
we can thus compute $\nabla \cE$ and devise a stochastic gradient descent, as synthesized in \autoref{alg:fixed}, where we use $P$ projections per descent step as above.

\begin{algorithm}[htp]
    \caption{Fixed support sliced Wasserstein Barycenter algorithm}\label{alg:fixed}
    \begin{algorithmic}
        \Require Support $X=(\bx_j)_{1\leq j\leq N} \in\X^N\subset \R^{N\times d}$, weights $\bv^{(i)}\in \Delta_N, \forall i \in \ii{M}$, step size $\tau>0$, initialization $\bw^0 \in \Delta_N$, number $P$ of slices
        \For{$l = 1, ...$}
            \State $(\bpsi_q)_{1\leq q\leq P} \gets \texttt{generate\_uniform\_samples\_on\_}\,\D(P)$
            \For{$i \in\ii{M}$ and $q\in\ii{P}$}
                \State $(s_j)_{1\leq j\leq N} \gets (\cS_{\bpsi_q}(\bx_j))_{1\leq j\leq N}$
                \vspace{4mm}
                \State $\zb\rho \gets \texttt{argsort}\left((s_j)_{1\leq j\leq N}\right)$
                \State $\zb t \gets \left(s_{\rho(j)}\right)_{1\leq j\leq N}$ \Comment{sorted}
                \State $\tilde \bw \gets \texttt{cumsum} \left(w^l_{\rho(j)}\right)_{1\leq j\leq N-1}$ \Comment{cumulative sum}
                \State $\tilde \bv \gets \texttt{cumsum} \left(v^{(i)}_{\rho(j)}\right)_{1\leq j\leq N-1}$
                \vspace{4mm}
                \State $\bu \gets \texttt{concatenate} (\tilde \bw, \tilde \bv)$
                \State $\texttt{mw} \gets (1, ..., 1, 0, ..., 0)\in \R^{2N-2}$
                \State $\zb\sigma \gets \texttt{argsort}(\bu)$
                \State $\bz \gets (u_{\sigma(j)})_{1\leq j\leq 2N-2}$
                \State $\texttt{mw} \gets (\texttt{mw}_{\sigma(j)})_{1\leq j\leq 2N-2}$
                \vspace{4mm}
                \State $\texttt{idx\_w} \gets \texttt{cumsum}(\texttt{mw}) + 1$
                \State $\texttt{idx\_v} \gets \texttt{cumsum}(1 - \texttt{mw}) + 1$
                \State $\ba \gets (|t_{\texttt{idx\_w}(j)} - t_{\texttt{idx\_v}(j)}|^p)_{1\leq j\leq 2N-2}$
                \State $a_0 \gets 0$
                \vspace{4mm}
                \State $\tilde \bg \gets (a_{\sigma^{-1}(k) - 1} - a_{\sigma^{-1}(k)})_{1\leq k\leq N-1}$
                \State $\bg \gets \texttt{cumsum\_adj}(\tilde \bg)$ \Comment{cumulative sum from right to left}
                \State $g_N = 0$
                \State $\texttt{grad}_{i,q} = (g_{\rho^{-1}(j)})_{1\leq j\leq N}$
            \EndFor
            \State $\texttt{grad}_{\R^N} \gets \frac{1}{P}\sum_{i = 1}^M\lambda_i \sum_{q = 1}^P \texttt{grad}_{i,q}$
            \State $\texttt{grad}_{\Delta_N} \gets \proj_H(\texttt{grad}_{\R^N}) \qquad= \texttt{grad}_{\R^N} - \frac{1}{\sqrt{N}}\inn{\texttt{grad}_{\R^N}, \mathds{1}_N}$
            \State $w^{l+1} \gets \proj_{\Delta_N}(w^l - \tau_l \texttt{grad}_{\Delta_N})$
        \EndFor
        \State \textbf{Output:} $w^{l+1}$
    \end{algorithmic}
\end{algorithm}
    
The complexity of the computation of the gradient \eqref{eq:fixed-1d-grad} is $\mathcal O(N\log N)$, as we need to sort the points $(t_j)_{j=1}^{N}$ and the vector $\bu$, and all other operations are done in linear time.
The projection $\proj_{\Delta_N}\colon \R^N\to\Delta_N$ on the probability simplex can be computed in complexity $\mathcal O(N\log(N))$ using the algorithm from \cite{WanCar13}, see also \cite{Con16} for further numerical approaches.
Therefore, one iteration of \autoref{alg:fixed} has the arithmetic complexity $\mathcal O(MN (\log(N)+d)P)$. However, we want to point out that the number $N$ of points of a fixed grid generally grows with the dimension $d$.

\subsection{Radon Wasserstein barycenters}
\label{sec:bary-radon}

\emph{Radon Wasserstein barycenters} are obtained by first computing the 1D barycenter \eqref{eq:cdt-bary} for every slice,
stacking them together,
and then applying the pseudoinverse of the respective Radon transform,
cf.\ \cite{BonRabPeyPfi15,KolNadSimBadRoh19}.
Denoting by $\cZ\colon \P(\X) \to \P(\D\times\II)$ the generalized slicing transformation, we set for $\mu_m\in\Pac(\X)$ the Radon barycenter
\begin{equation} \label{eq:disc-svd}
    \bary_{\X}^{\cZ}(\mu_m, \lambda_m)_{m=1}^{M}
    \coloneqq
    \cZ^\dagger \left( \left( \bary_\II(\lambda_m,(\cS_\bpsi)_\#\mu_m)_{m=1}^{M} \right)_{\bpsi\in\D} \right),
\end{equation}
where $\cZ^\dagger$ is the pseudoinverse whose argument is viewed as a density function on $\D\times\II$.
In general, it is not clear if the pseudoinverse yields again a nonnegative function, which then gives a probability density.
On $\S^{d-1}$, it is fulfilled for the parallel slicing $\V$ by \autoref{thm:pos},
but not for the semicircular slicing, see \cite[sect.\ 6.2]{QueBeiSte23}.

The discretization is based on a fixed support.
We describe the parallel slicing case $\cZ=\U$ analogously to the semicircular case $\cZ=\cW$ from \cite{QueBeiSte23}.
Let $\bpsi_p \in \S^{d-1}$, $p\in\ii{P}$, be the nodes of a quadrature rule with weights $\bw\in\Delta_P$,
and some grid $t_\ell $, $\ell\in\ii{L}$, on the interval $\II$.
We denote the density function of $\mu_m$ by $f^{\mu_m}$
and assume that we are given $f^{\mu_m}(\bpsi_p)$ for $p\in\ii{P}$.
Firstly, we approximate $\U_{\bpsi_p}(t_\ell)$ via the singular value decomposition \eqref{eq:M-svd} for fixed truncation degree $D\in\N$ by
\begin{equation*}
    \U_{\bpsi_p} f^{\mu_m}(t_\ell)
    \approx
    \sum_{n=0}^D
    \sum_{j=1}^{N_{n, d}}
    \lambda_{n, d}^{\U}
    Y_{n, d}^j(\bpsi_p) \widetilde P_{n, d}(t_\ell)
    \sum_{i=1}^{P} f^{\mu_m}(\bpsi_i) \overline{Y_{n, d}^j(\bpsi_i)} w_i.
\end{equation*}
Secondly, we compute the density of the one-dimensional barycenters \eqref{eq:cdt-bary} of the measures $\U_{\bpsi_p}\mu_m$.
In particular, we set $g(\bpsi_p,\cdot)$ as the density function of
$$
\CDT^{-1}_{\omega}\left(
\sum_{m=1}^{M} \lambda_m \CDT_\omega [\U_{\bpsi_p}\mu_m] \right).
$$
Using again the singular value decomposition \autoref{thm:svd},
we note that the \emph{Moore--Penrose pseudoinverse} \cite{EnHaNe96} of $\U$ is given by
\begin{equation} \label{eq:Vdag}
  \U^\dagger \colon \operatorname{Range}(\U)\oplus \operatorname{Range}(\U)^\perp \to L^2(\S^{d-1})
  ,\quad
  \U^\dagger g 
  = \sum_{n=0}^{\infty}
  \sum_{k=1}^{N_{n,d}}
  \frac{1}{\lambda_{n,d}^{\U}} \,
  \inn{ g, Y_{n,d}^k\, \widetilde P_{n,d} } \,
  Y^k_{n,d}.
\end{equation}
Finally, we discretize the Moore--Penrose pseudoinverse to approximate the density of the desired barycenter $\bary_{\S^{d-1}}^{\U}$ by 
\begin{equation*}
    \U^\dag g(\bpsi_p)
    \approx
    \sum_{n=0}^D
    \sum_{j=1}^{N_{n, d}}
    \frac{1}{\lambda_{n,d}^{\U}}
    Y_{n, d}^j(\bpsi_p) 
    \frac{1}{L} \sum_{i=1}^{P} \sum_{\ell=1}^{L} g(\bpsi_i, t_\ell) \overline{Y_{n, d}^j(\bpsi_i)} \widetilde P_{n, d}(t_\ell) w_i.
\end{equation*}
We analyze the complexity for $\S^2$.
The sums over $n$ and $j$ constitute a nonuniform spherical Fourier transform, and the sum over $i$ its adjoint,
which can both be computed efficiently in $\mathcal{O}(D^2 \log (D)+P)$ steps, see \cite{kupo02} and \cite[sect.\ 9.6]{PlPoStTa23}.
The CDT \eqref{eq:cdt} and its inverse \eqref{eq:icdt} have linear complexity and can be computed with the algorithm \cite{KolParRoh16}.
Therefore, we obtain an overall complexity of $\mathcal{O}((D^2 \log (D)+P)LM)$.
Since the number of points is usually $P\sim D^2$, 
the complexity grows slower than for the algorithms of \autoref{sec:bary-sliced}.

%-------------------------------------------------
\section{Numerical results} \label{sec:numerics}
%-------------------------------------------------

In this section, we present numerical computations of sliced barycenters between two measures on the sphere. and visualize them $\S^2$. On general $\S^{d-1}$, we only look at free-support barycenters, which do not require grids that become unhandy in higher dimension, and test their convergence behavior and execution times.
We compare two notions of slicing:
our parallel slicing \eqref{eq:slice}
and the semicircular slicing of \autoref{rem:SSW},
each for the two different notions of free and fixed support sliced Wasserstein barycenters from \autoref{sec:bary-sliced} and the Radon barycenters from \autoref{sec:bary-radon}.
Further, we compare the fixed-support and Radon barycenters with the entropy-regularized version \cite{PeyCut19} of the Wasserstein barycenter \eqref{eq:W-bary} computed with PythonOT \cite{POT}. Finally, we consider the barycenters on $\SO$ in \autoref{sec:numerics-SO}. Our code is available online.\footnote{\url{https://github.com/leo-buecher/Sliced-OT-Sphere}}

\subsection{Free-support sliced Wasserstein barycenters on the sphere}

For the free-support discretization of \autoref{sec:disc-free},
we compute the parallelly sliced Wasserstein barycenter (PSB) $\bary^{\PSW}_{\S^2}$ with the stochastic gradient descent \eqref{eq:free-psb-sgd},
and the semicircular sliced Wasserstein barycenter (SSB) $\bary^{\SSW}_{\S^2}$, see \autoref{rem:SSW},
with the algorithm \cite{Bon22}, which
uses a similar gradient descent scheme.\footnote{See the code \url{https://github.com/clbonet/Spherical_Sliced-Wasserstein}}

The von Mises--Fisher (vMF) distribution with center $\zb\eta\in\S^2$
and concentration $\kappa > 0$ has the density function
\begin{equation} \label{eq:vmf}
    f(\bxi) 
    = \frac{\kappa}{4\pi \sinh\kappa}\, \e^{\kappa \inn{\bxi,\zb\eta}}
    ,\qquad \bxi\in\S^2.
\end{equation}
In the first setting, we consider the sliced Wasserstein barycenters of two vMF distributions with centers on the equator shifted by $90^\circ$ 
and concentration $\kappa = 100$, represented by $N=200$ samples each. 
We use 1000 iterations, $P=500$ slices, a step size of $\tau=40$ for the PSB algorithm and 80 for the SSB algorithm, and we take samples of the uniform distribution $u_{\S^2}$ as initialization $X^0$.  
The computed sliced Wasserstein barycenters are displayed in \autoref{fig:sphere-shape-comp-2vmf}. We observe that the SSB is slightly more extended toward the poles. 
    
\begin{figure}[!ht]
    \centering
    \includegraphics[height=4cm,trim={20 75 50 70},clip]{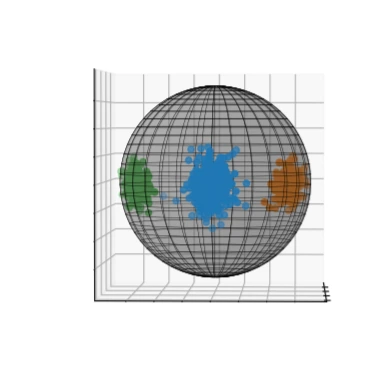}
    \qquad
    \includegraphics[height=4cm,trim={30 80 50 70},clip]{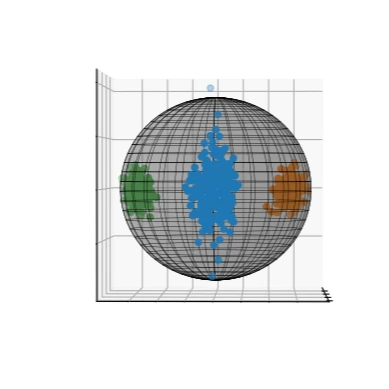}
    \caption{PSB (left) and SSB (right) of two vMF distributions. The green and orange points represent the two input measures, and the blue points the barycenter.}
    \label{fig:sphere-shape-comp-2vmf}
\end{figure}

In the second setting, we consider vMF distributions that are highly concentrated near the poles with $\kappa = 400$, to illustrate the observations of \autoref{rem:antipo_diracs}. 
The resulting sliced Wasserstein barycenters are shown in \autoref{fig:sphere-shape-comp-2diracs}. The PSB is seemingly uniform on the sphere, which corresponds to the initial distribution $X^0$. This is coherent with the observation that all measures on the sphere are PSB of two antipodal Dirac measures. Conversely, the SSB is supported on a ring around the equator.
\begin{figure}[!ht]
    \centering
    \includegraphics[height=4cm,trim={0 20 10 25},clip]{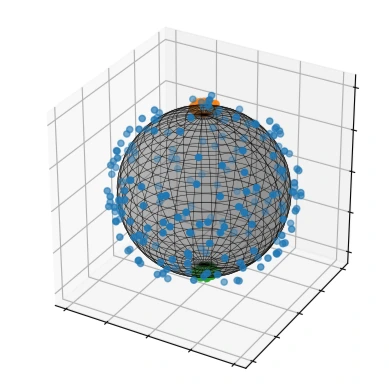}
    \qquad
    \includegraphics[height=4cm,trim={0 30 20 60},clip]{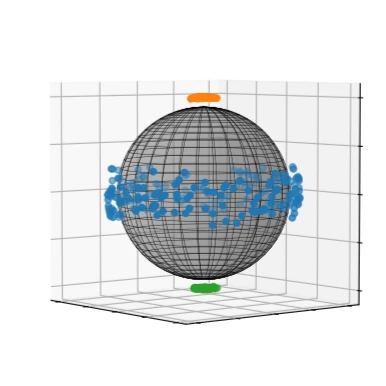}
    \caption{PSB (left) and SSB (right) of highly concentrated measures. The green and orange points represent the two input measures, and the blue points the barycenter.}
    \label{fig:sphere-shape-comp-2diracs}
\end{figure}

The third setting is with two ``croissant'' measures spanning from the South Pole to the North Pole, and rotated from each other by an angle of $120^\circ$. The resulting PSB and SSB in \autoref{fig:sphere-shape-comp-2crois} are quite similar, which illustrates the fact that in many cases the two notions of barycenters seem more or less to coincide.
\begin{figure}[!ht]
\centering
\includegraphics[height=3.6cm,trim={40 70 40 70}, clip]{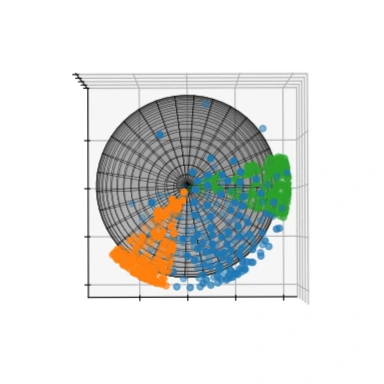}
\qquad
\includegraphics[height=3.6cm,trim={40 70 40 70}, clip]{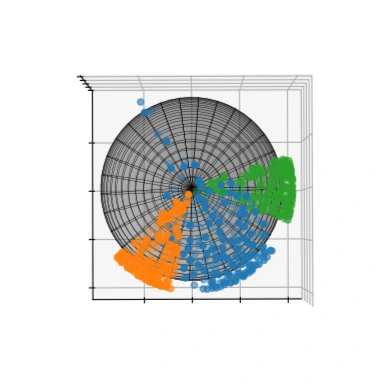}
\\
\includegraphics[height=4cm,trim={10 15 10 20}, clip]{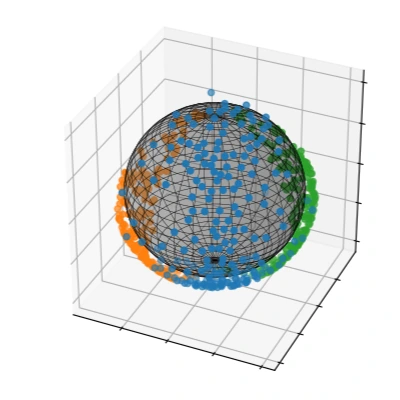}
\qquad
\includegraphics[height=4cm,trim={10 5 10 20}, clip]{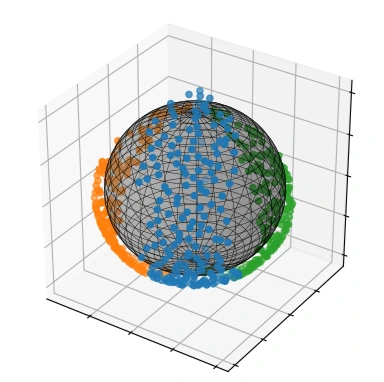}
\caption{PSB (left) and SSB (right) for 2 ``croissants'' measures. The green and orange points represent the two input measures, and the blue points represent the barycenter.}
\label{fig:sphere-shape-comp-2crois}
\end{figure}

\paragraph*{Convergence}
We study the convergence of the algorithms using the same measures as in \autoref{fig:sphere-shape-comp-2vmf}, but with $N=50$ samples. \autoref{fig:sphere-conv} shows the evolution of the loss function \eqref{eq:SW-bary} and the step norm, which is the $L^2$ norm of the step $X^{l+1} - X^{l}$, depending on the iteration. We show the results on $\S^{d-1}$ for $d=3$ and $d=10$, where we use $\tau=20$ for PSB and $\tau=50$ for SSB.
The two losses are rescaled so that the initial loss coincides, as they cannot be compared in absolute value. Despite
this, these loss evolutions remain difficult to compare, as they highly depend on the chosen step
size $\tau$. However, we observed that for appropriate step sizes $\tau$ (i.e. that avoid oscillatory behavior around the minimizer), the PSB algorithm converges faster on $\S^2$. On $\S^{9}$, both algorithms show a similar speed of convergence. This change in dimension might be explained by the fact that a uniform measure on $\S^{d-1}$ (which is our initial distribution) is projected by the semicircular slicing to a uniform measure on the circle, while the parallel slicing  becomes more concentrated around 0 on the unit interval $\II$ for larger $d$.
\begin{figure}[!ht]
    \centering
    \includegraphics[height=4.2cm,trim={52 20 40 40},clip]{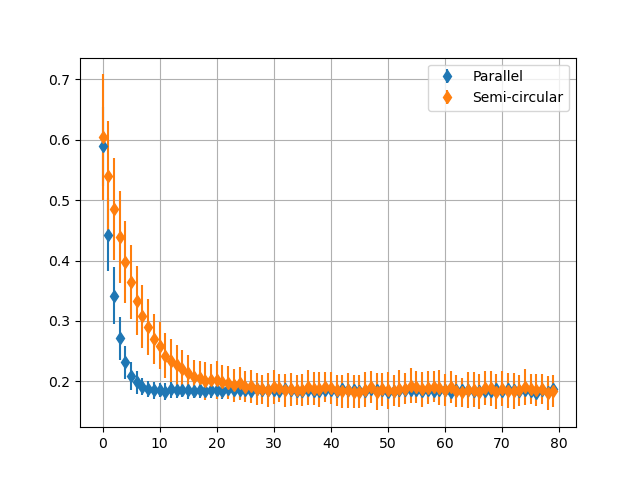} 
    \qquad
    \includegraphics[height=4.2cm,trim={52 20 40 40},clip]{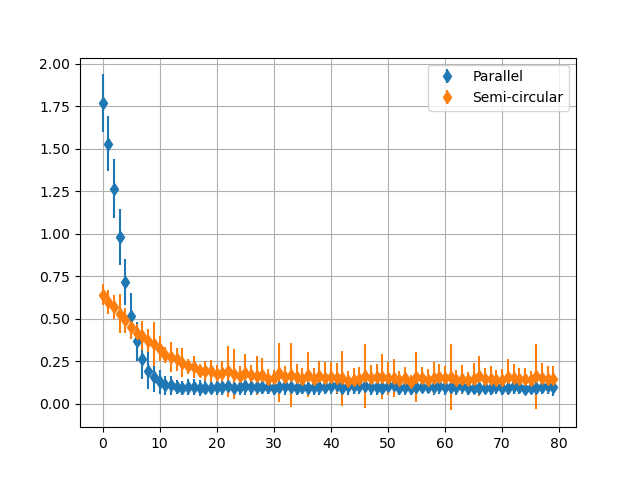}\\
    \includegraphics[height=4.2cm,trim={52 20 40 40},clip]{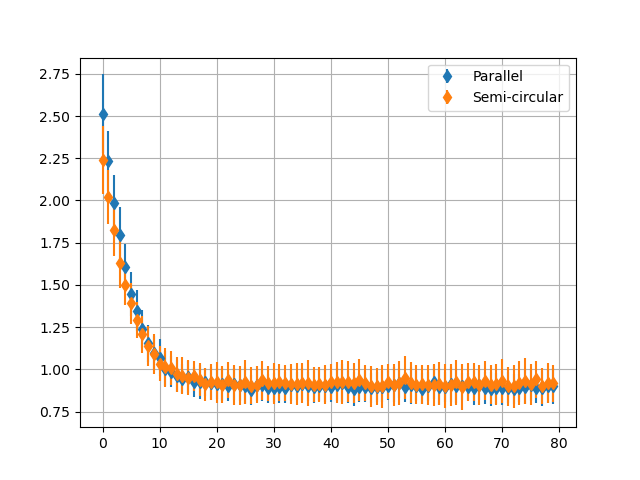} 
    \qquad
    \includegraphics[height=4.2cm,trim={52 20 40 40},clip]{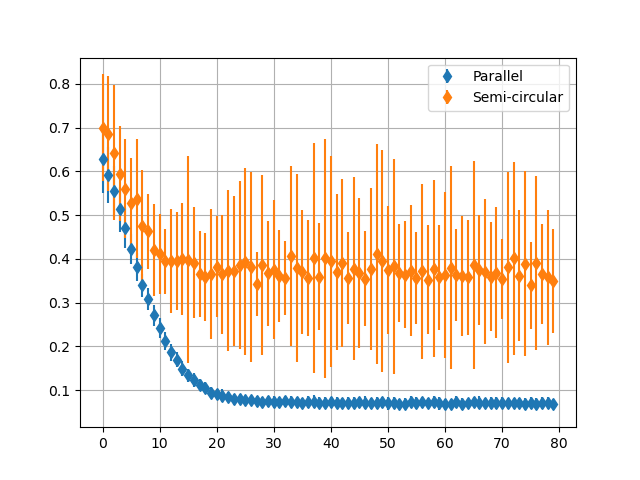}
    \caption{Loss function (left) and step norm (right) on $\S^{d-1}$ for $d=3$ (top row) and $d=10$ (bottom row) depending on the number of iterations. Reported are the mean and standard deviation (vertical lines) over 100 runs.}
    \label{fig:sphere-conv}
\end{figure}

\paragraph*{Complexity}
All tests were performed on an Intel Core i7-10700 with 32 GB memory.
\autoref{fig:sphere-time} shows the execution time depending on the number $N$ of points of each given measure and the number $P$ of slices,
both for a fixed number of 20 iterations without stopping criterion. 
We observe that the PSB algorithm is between 40 and 100 times faster than the SSB algorithm. 
This comes from the fact that the SSB requires to solve an OT problem on the torus $\T$, which is more difficult than on the real line.
In terms of the evolution along $N$, the two algorithms (PSB, SSB) have a complexity of $O(P N(\log(N)+d))$, which is coherent with our observations.
The dependence on $P$ is linear, but experiments with bigger values of $P$ would be required to confirm this.
\begin{figure}[!ht]
    \centering
    \includegraphics[height=4.2cm,trim={28 20 40 40},clip]{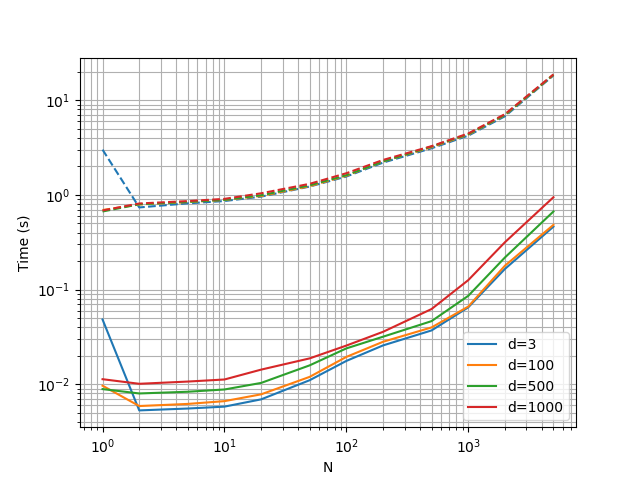}
    \qquad
    \includegraphics[height=4.2cm,trim={28 20 40 40},clip]{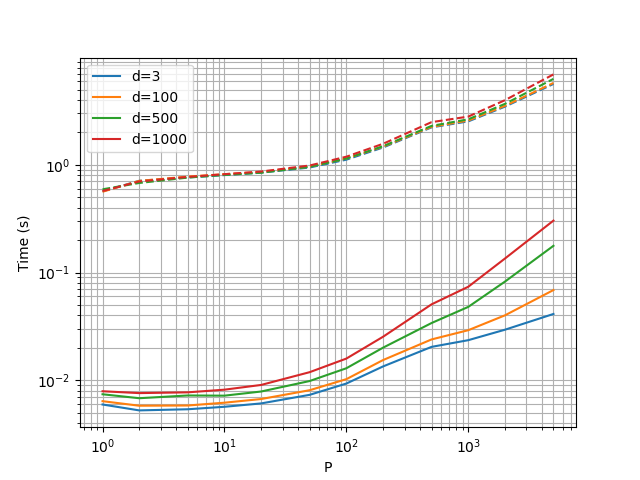}
    \caption{Execution time of the PSB (solid line) and SSB (dashed) algorithms with 20 iterations and different dimensions $d$.\\
    Left: Dependence on the number $N$ of points with $P = 200$ slices.\\
    Right: Dependence on the number $P$ of slices with $N=40$ points.
    \label{fig:sphere-time}}
\end{figure}

The execution times barely grow with increasing dimension $d$ of the sphere.
The only parts of the algorithm that depend on $d$ are the inner products in $\R^d$, which are well parallelized, and the generation of uniform samples on $\S^{d-1}$, which explains the stronger effect for larger number of slices $P$.
While in \autoref{fig:sphere-time} the relative differences are higher for the generally faster PSB alrorithm, the absolute differences are usually larger for the SSB.

\subsection{Fixed-support sliced Wasserstein barycenters}

We test the fixed-support sliced barycenter algorithm from \autoref{sec:disc-fixed} and compare the resulting barycenters with the free-support ones.  As input measures, we take a vMF distribution with $\kappa = 30$ and a ``smiley'' distribution, see \autoref{fig:sphere-fixed}. We use the gradient descent \autoref{alg:fixed} with a grid of $150\times 50$ 
points on the sphere, $P=100$ slices, 500 iterations, the uniform distribution as initialization, and an empirically chosen step size $\tau = 0.005\, (1 + k/20)^{-1/2}$ in the $k$-th iteration. 
We noticed that the results highly depend on the step size.
    
The resulting fixed-support PSB is displayed in \autoref{fig:sphere-fixed}, along with the free-support PSB and SSB from \autoref{sec:disc-free} and the regularized Wasserstein barycenter from PythonOT with regularization parameter 0.05.
For the free-support sliced Wasserstein barycenters (PSB and SSB), the two input measures are sampled $N=200$ points and the resulting barycenters are represented using kernel density estimation \cite{HalWatCab87} with the density of the vMF distributions \eqref{eq:vmf} as kernel function.
We notice that all barycenters look similarly.

\begin{figure}[!ht]
    \centering
    \begin{subfigure}[t]{0.3\textwidth}
        \centering
        \includegraphics[width=\textwidth, trim={90 25 40 40}, clip]{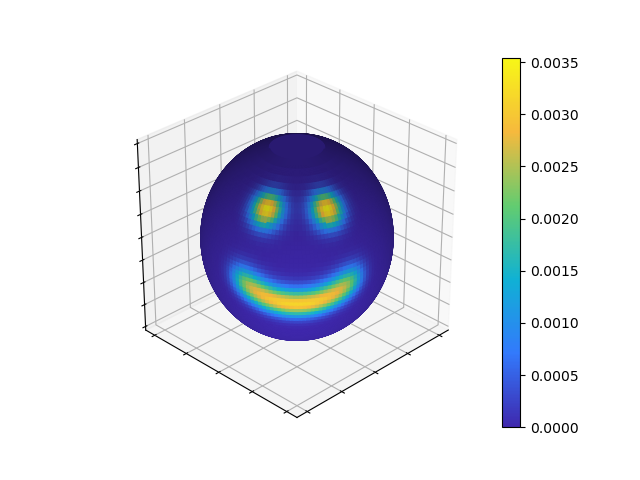} %left, bottom, right, top
        \caption{First input measure}
        \label{fig:sphere-fixed-input1}
    \end{subfigure}\hfill
    \begin{subfigure}[t]{0.3\textwidth}
        \centering
        \includegraphics[width=\textwidth, trim={90 25 40 40}, clip]{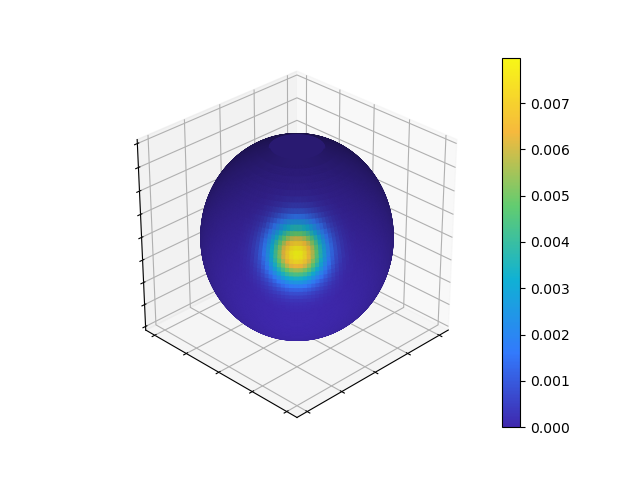} %left, bottom, right, top
        \caption{Second input measure}
        \label{fig:sphere-fixed-input2}
    \end{subfigure}\hfill
    \begin{subfigure}[t]{0.3\textwidth}
        \centering
        \includegraphics[width=\textwidth, trim={90 25 40 40}, clip]{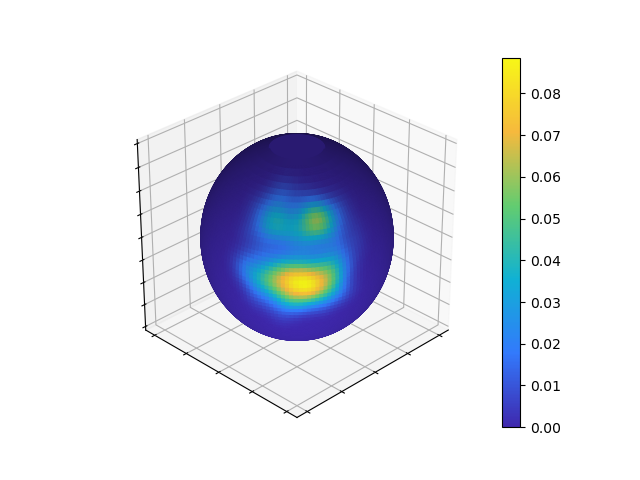} %left, bottom, right, top
        \caption{Density estimation of the free-support PSB (0.05\,s)}
        \label{fig:sphere-fixed-LPSB}
    \end{subfigure}
    \begin{subfigure}[t]{0.3\textwidth}
        \centering
        \includegraphics[width=\textwidth, trim={90 25 40 30}, clip]{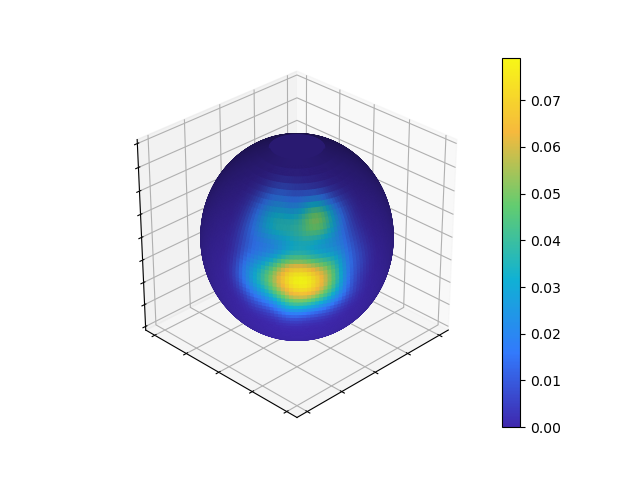} %left, bottom, right, top
        \caption{Density estimation of the free-support SSB (15\,s)}
        \label{fig:sphere-fixed-LSSB}
    \end{subfigure}\hfill
    \begin{subfigure}[t]{0.32\textwidth}
        \centering
        \includegraphics[width=\textwidth, trim={80 25 40 30}, clip]{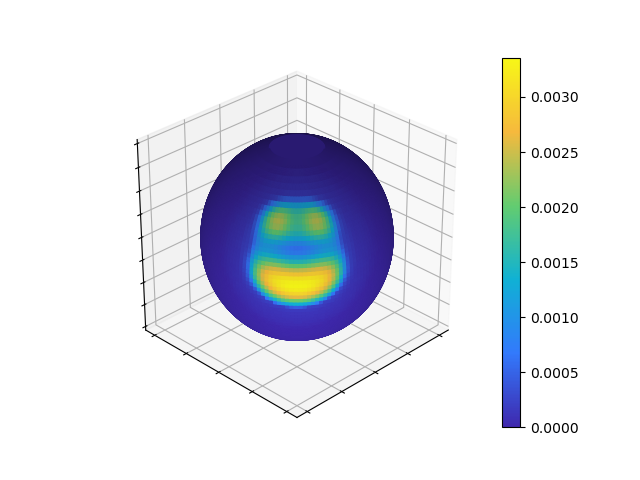} %left, bottom, right, top
        \caption{Fixed-support PSB (107\,s)}
        \label{fig:sphere-fixed-EPSB}
    \end{subfigure}\hfill
    \begin{subfigure}[t]{0.32\textwidth}
        \centering
        \includegraphics[width=\textwidth, trim={80 25 40 30}, clip]{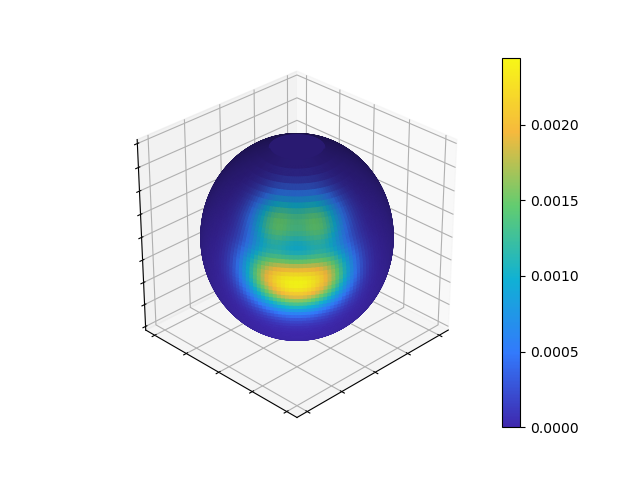} %left, bottom, right, top
        \caption{Regularised Wasserstein barycenter (3.8\,s)}
        \label{fig:sphere-fixed-regwb}
    \end{subfigure}
    \caption{Fixed-support and free-support barycenters between two input measures.}
    \label{fig:sphere-fixed}
\end{figure}

\subsection{Radon Wasserstein barycenters on the sphere}
We compare the Radon PSB $\bary^{\U}_{\S^2}$ of \autoref{sec:bary-radon}
with the Radon SSB $\bary^{\cW}_{\S^2}$.
For the latter, we apply the algorithm \cite{QueBeiSte23}.
Both use the truncation degree $N=120$ in \eqref{eq:disc-svd} and $P=29282$ slices, which equals the number of points on the $121\times242$ grid on $\S^2$.
\autoref{fig:radon-2diracs} shows the barycenters of vMF distributions concentrated at the poles.
As above, the Wasserstein barycenter is computed with the POT library with the regularization parameter 0.05 (or 0.01 for the ``smiley'' test).
The regularized Wasserstein barycenter is somehow ``between'' the Radon PSB and SSB.
Different from \autoref{fig:sphere-shape-comp-2diracs},
the Radon is concentrated on a ring around the equator,
which might be explained by the fact that, other than in \autoref{rem:antipo_diracs}, the measures are not Dirac measures.
Furthermore, the computation of the Radon PSB is much faster than the Radon SSB. 
Moreover, we compare the Radon barycenters of the ``croissant'' shape in \autoref{fig:radon-2crois}
as well as the ``smiley'' in \autoref{fig:radon-smiley}.
Again, the Wasserstein barycenter seems to be between the Radon PSB and SSB.
\begin{figure}[!ht]
    \centering
    \begin{subfigure}[t]{.24\textwidth}
    \includegraphics[width=\textwidth]{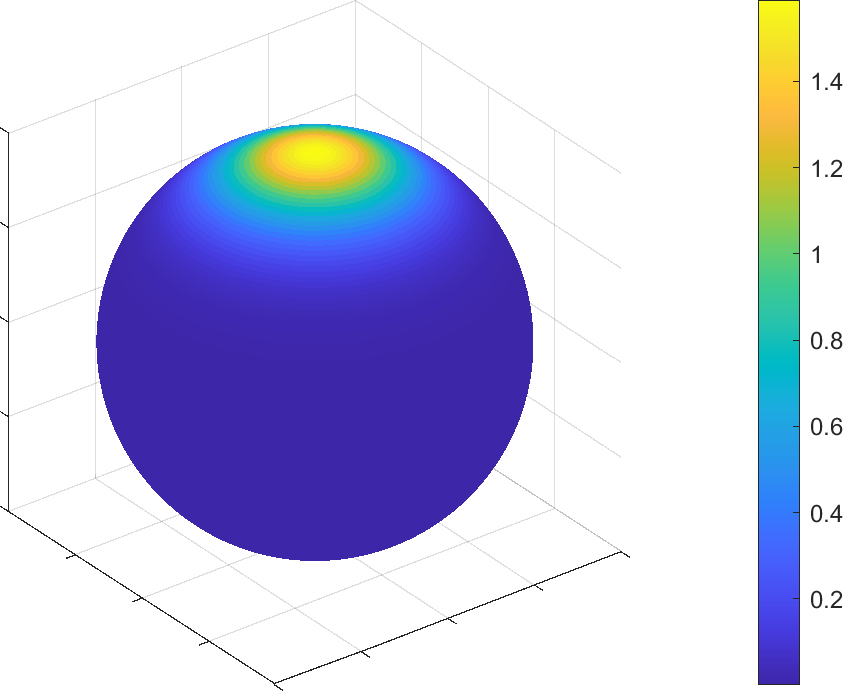}
    \caption{First input distribution}
    \end{subfigure}\hfill
    \begin{subfigure}[t]{.24\textwidth}
    \includegraphics[width=\textwidth]{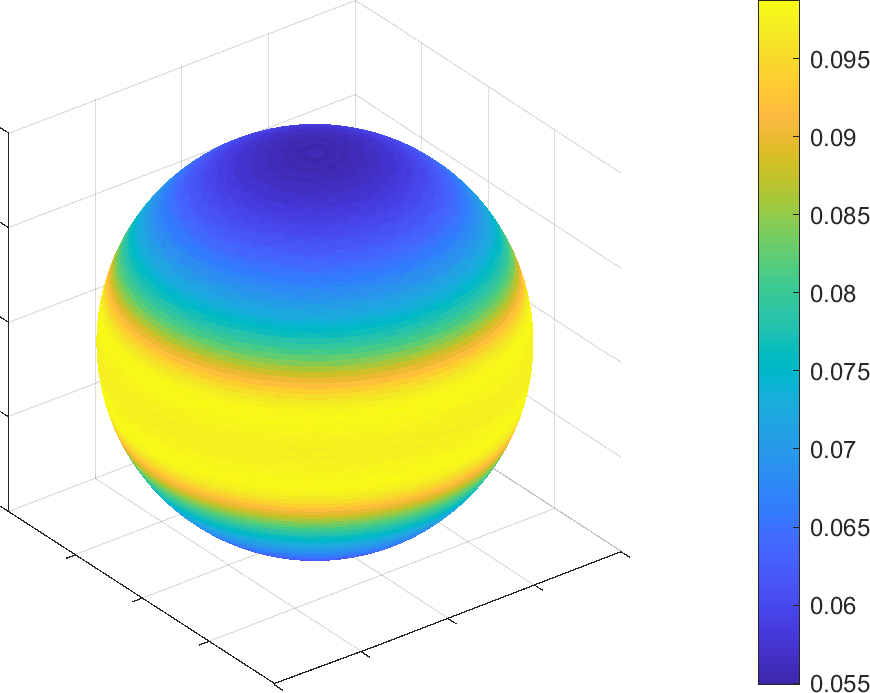}
    \caption{Radon PSB (0.7\,s)}
    \end{subfigure}\hfill
    \begin{subfigure}[t]{.24\textwidth}
    \includegraphics[width=\textwidth]{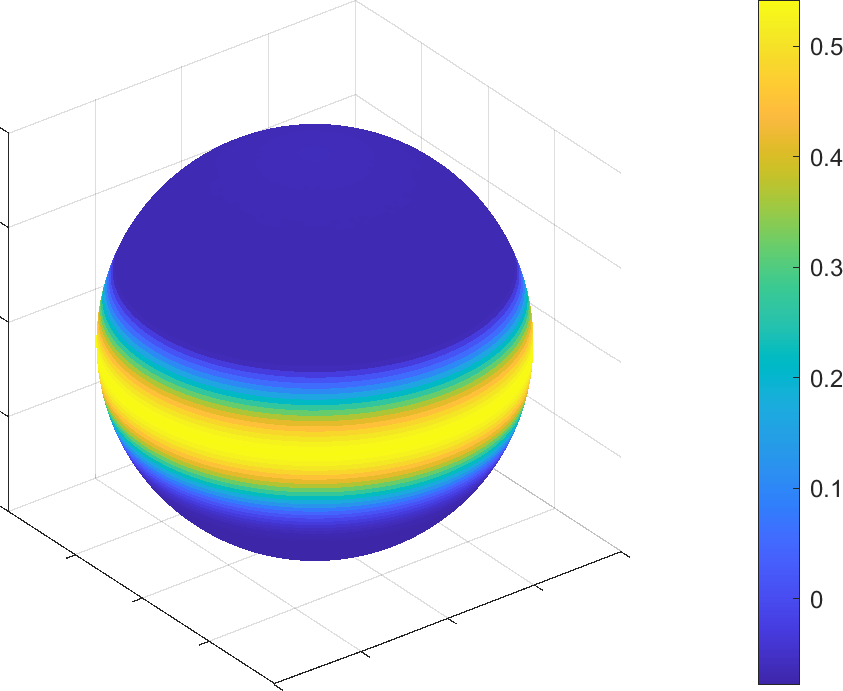}
    \caption{Radon SSB (89\,s)}
    \end{subfigure}\hfill
    \begin{subfigure}[t]{.24\textwidth}
    \includegraphics[width=\textwidth]{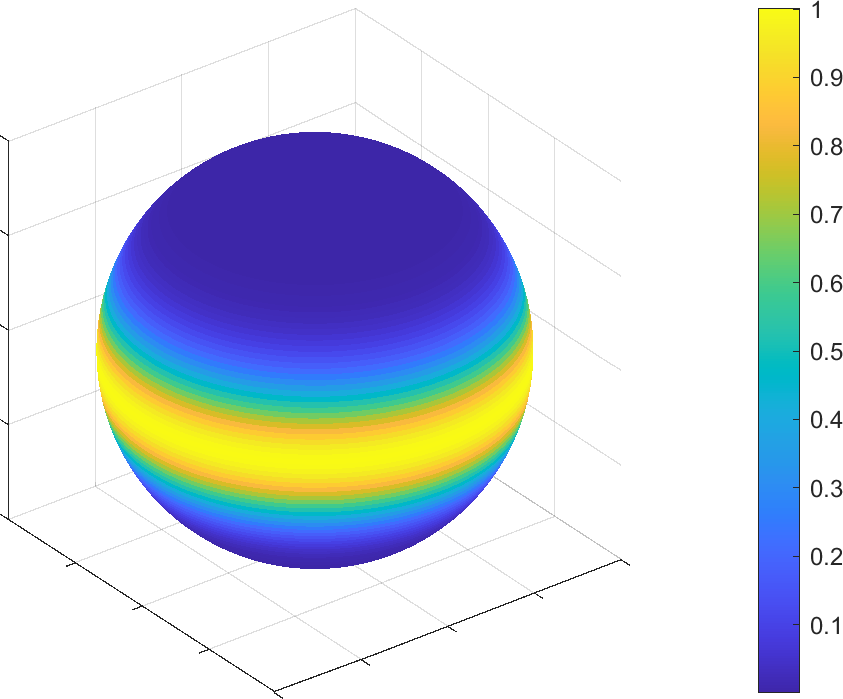}
    \caption{Wasserstein barycenter (364\,s)}
    \end{subfigure}
    \caption{First input measure (the second is the antipodal) and the Radon barycenters.
    }
    \label{fig:radon-2diracs}
\end{figure}

\begin{figure}[!ht]
    \centering
    \begin{subfigure}[t]{.192\textwidth}
    \includegraphics[width=\textwidth]{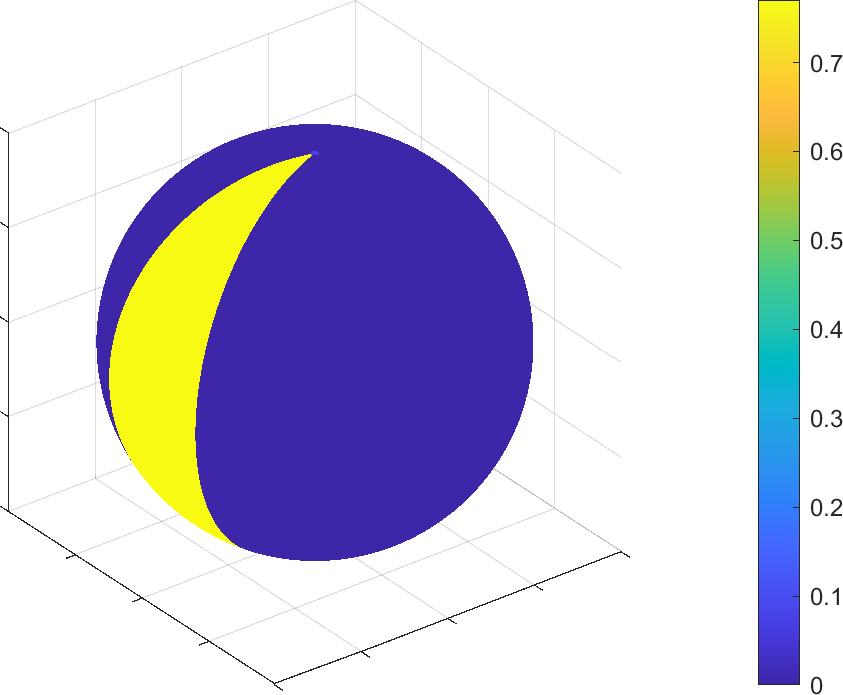}
    \caption{First input distribution}
    \end{subfigure}\hfill
    \begin{subfigure}[t]{.192\textwidth}
    \includegraphics[width=\textwidth]{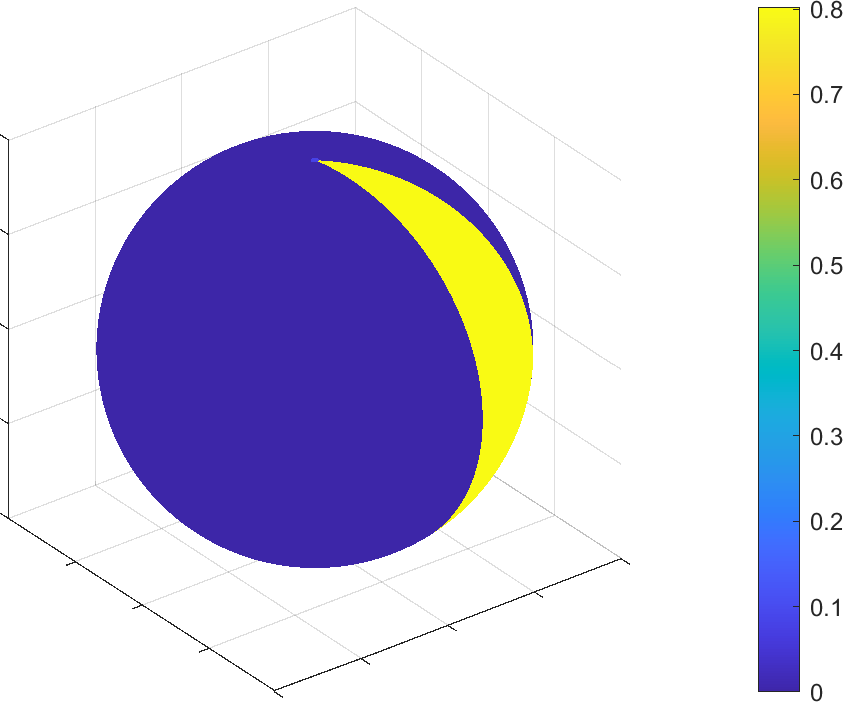}
    \caption{Second input distribution}
    \end{subfigure}\hfill
    \begin{subfigure}[t]{.192\textwidth}
    \includegraphics[width=\textwidth]{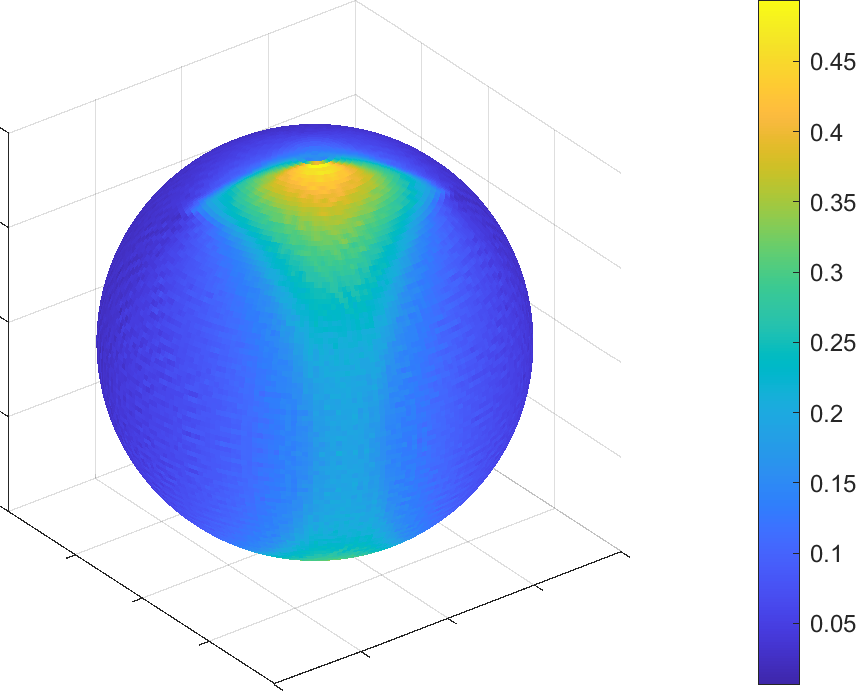}
    \caption{Radon PSB }
    \end{subfigure}\hfill
    \begin{subfigure}[t]{.192\textwidth}
    \includegraphics[width=\textwidth]{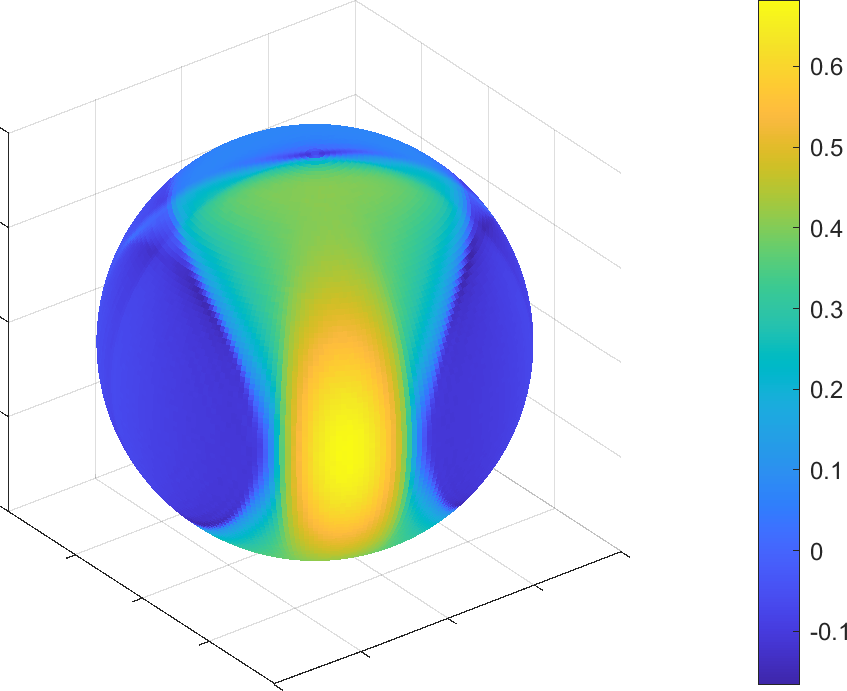}
    \caption{Radon SSB}
    \end{subfigure}\hfill
    \begin{subfigure}[t]{.192\textwidth}
    \includegraphics[width=\textwidth]{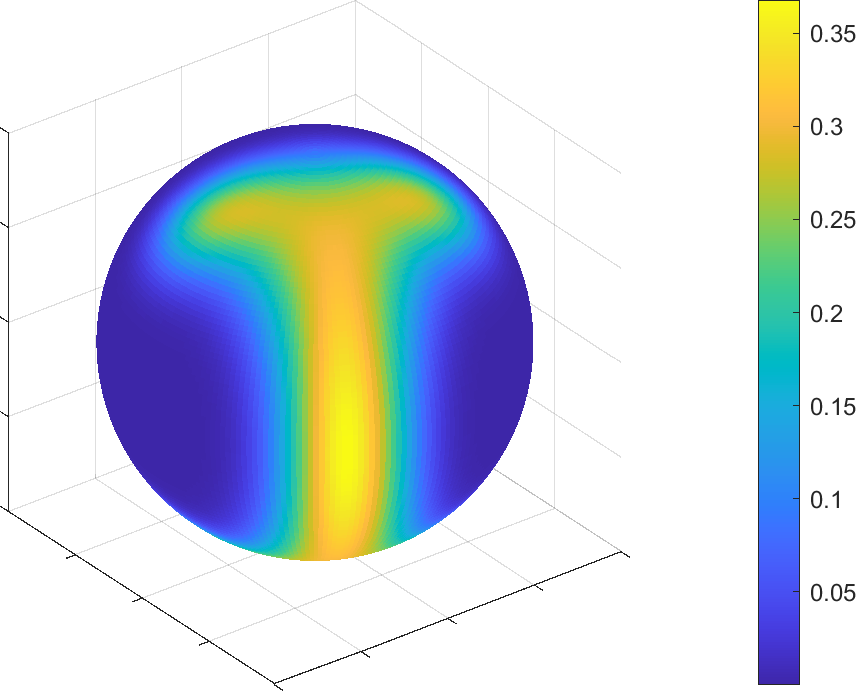}
    \caption{Wasserstein barycenter}
    \end{subfigure}
    \caption{Input measures and their Radon barycenters.
    }
    \label{fig:radon-2crois}
\end{figure} 

\begin{figure}[!ht]
    \centering
    \begin{subfigure}[t]{.192\textwidth}
    \includegraphics[width=\textwidth]{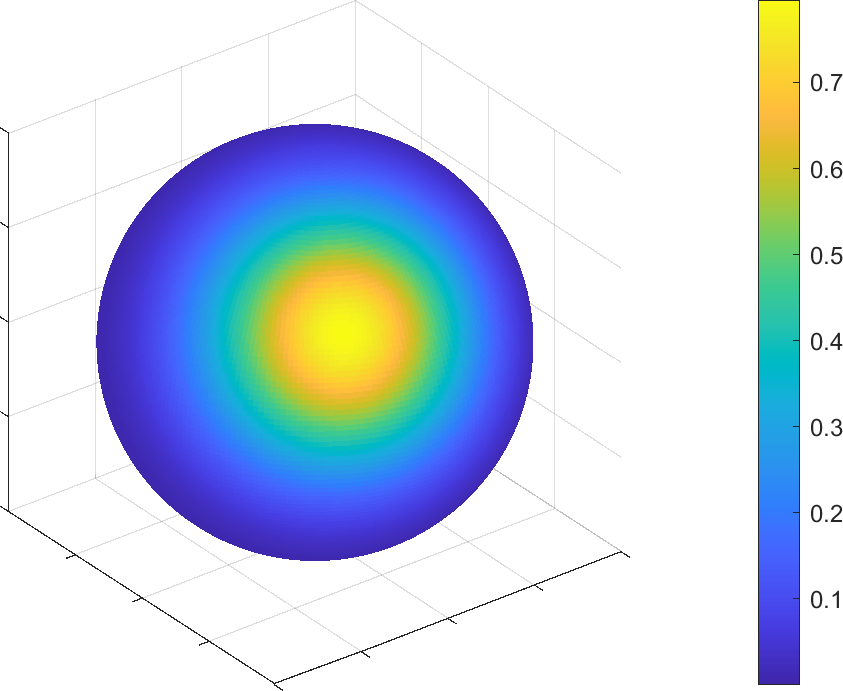}
    \caption{First input distribution}
    \end{subfigure}\hfill
    \begin{subfigure}[t]{.192\textwidth}
    \includegraphics[width=\textwidth]{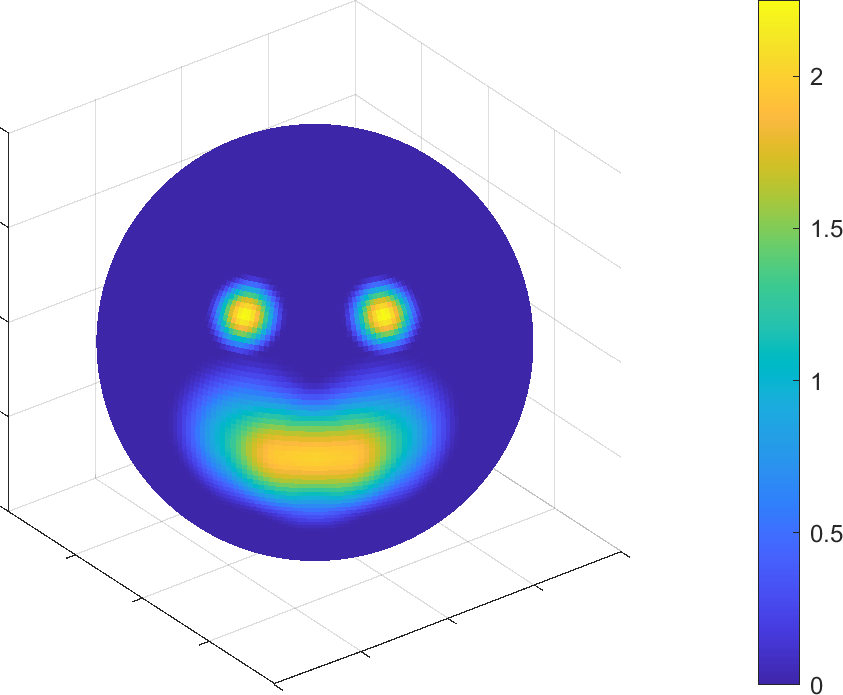}
    \caption{Second input distribution}
    \end{subfigure}\hfill
    \begin{subfigure}[t]{.192\textwidth}
    \includegraphics[width=\textwidth]{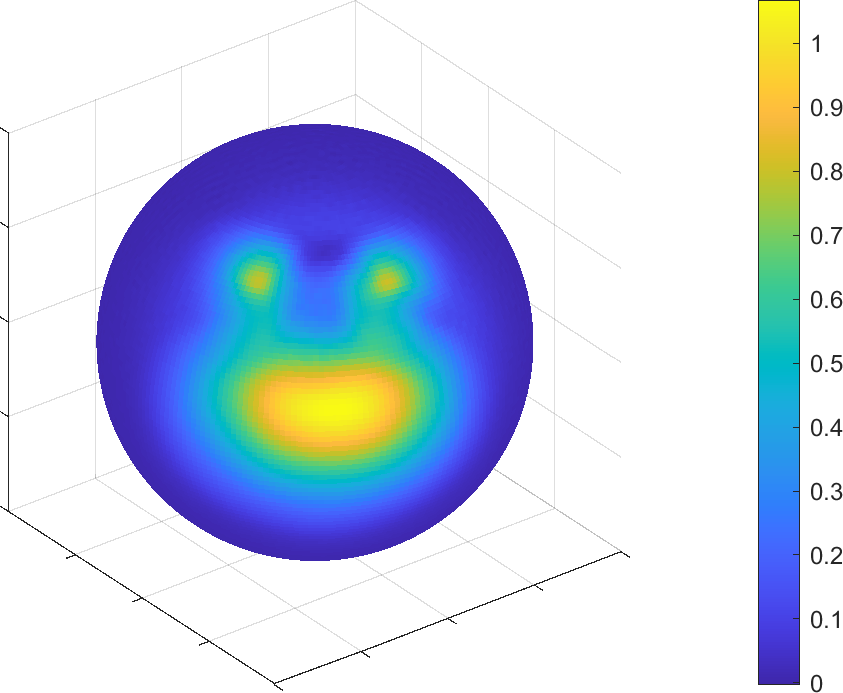}
    \caption{Radon PSB }
    \end{subfigure}\hfill
    \begin{subfigure}[t]{.192\textwidth}
    \includegraphics[width=\textwidth]{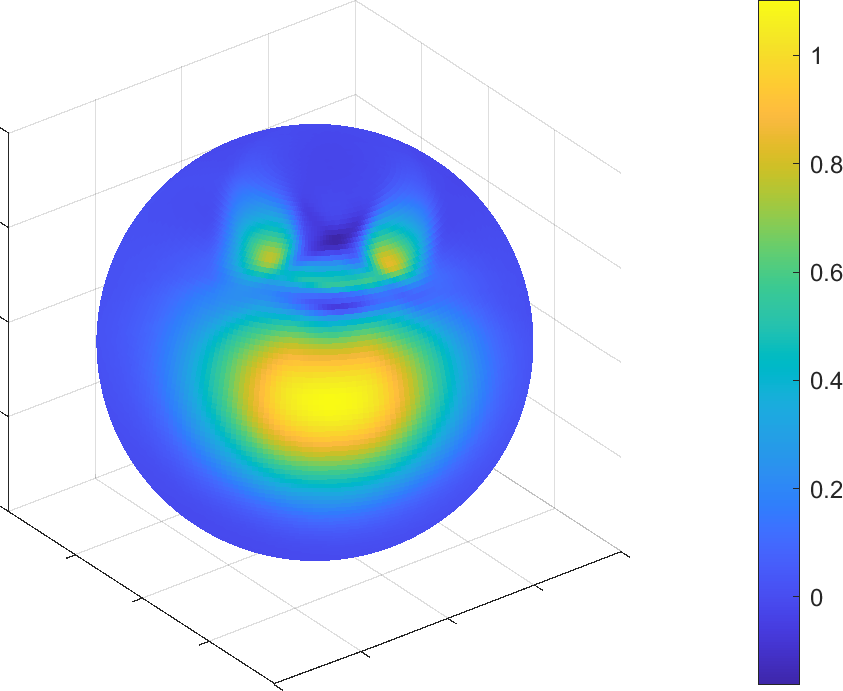}
    \caption{Radon SSB}
    \end{subfigure}\hfill
    \begin{subfigure}[t]{.192\textwidth}
    \includegraphics[width=\textwidth]{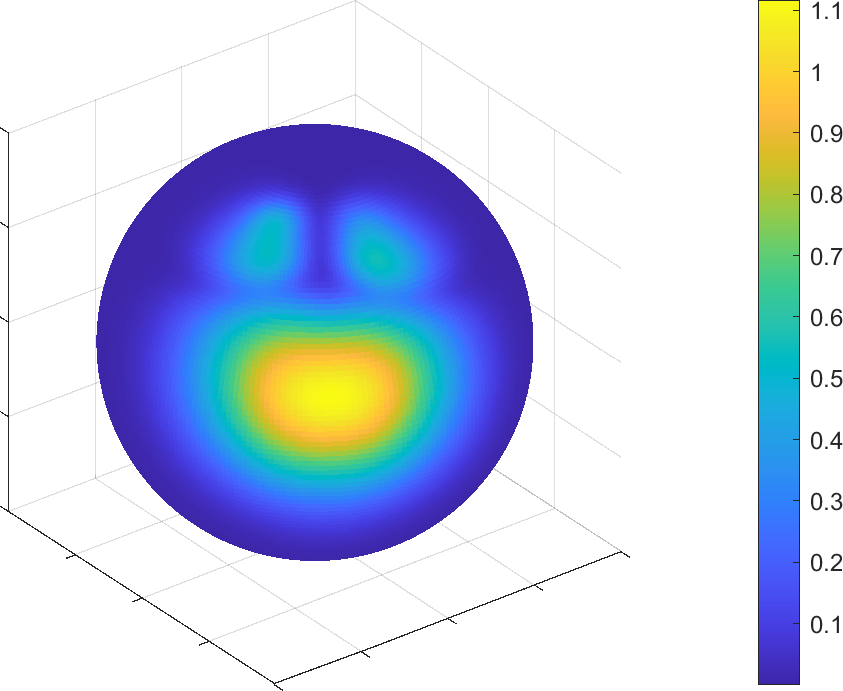}
    \caption{Wasserstein barycenter}
    \end{subfigure}
    \caption{Input measures and their Radon barycenters.
    }
    \label{fig:radon-smiley}
\end{figure}

\subsection{Sliced barycenters on the rotation group}
\label{sec:numerics-SO}

We compute free-support sliced Wasserstein barycenters on $\SO$ with the stochastic gradient descent method outlined in \autoref{sec:disc-so3}.
Similar to the spherical setting, we start with two given measures with $N=100$ points each.
Since there is no other slicing approach on $\SO$, we compare it with a Wasserstein barycenter that is obtained by computing an optimal transport map for the 1-Wasserstein distance \eqref{eq:Wp} using PythonOT and applying the logarithm map to project it to the manifold $\SO$.

We visualize a rotation $Q=\rR_{\bn}(\omega)\in\SO$, see \eqref{eq:axan}, as the point $\tan(\frac\omega4)\, \bn\in\R^3$, where $\omega\in[0,\pi]$ is the angle and $\bn\in\S^2$ is the rotation axis, cf.\ \cite[p.\ 1633]{EhlGraNeuSte21}.
The resulting free-support barycenters between two input measures are depicted in \autoref{fig:so3}.
In (a) and (b), the measures are closer together, the angle between their centers is approximately 92 degree, and notions of barycenters are similar.
In (c) and (d), the input measures are almost opposite to each other with a distance of 164 degrees, then the sliced barycenter is supported on a curve with high clustering between the inputs, and the projected Wasserstein barycenter is stronger localized around two locations.

\begin{figure}[!ht]
    \begin{subfigure}[t]{.24\textwidth}
    \includegraphics[width=\textwidth]{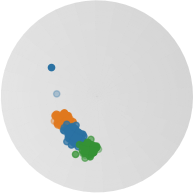}
    \caption{Sliced Wasserstein barycenter}
    \end{subfigure}\hfill
    \begin{subfigure}[t]{.24\textwidth}
    \includegraphics[width=\textwidth]{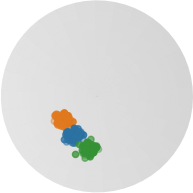}
    \caption{Wasserstein barycenter}
    \end{subfigure}\hfill
    \begin{subfigure}[t]{.24\textwidth}
    \includegraphics[width=\textwidth]{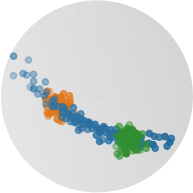}
    \caption{Sliced Wasserstein barycenter}
    \end{subfigure}\hfill
    \begin{subfigure}[t]{.24\textwidth}
    \includegraphics[width=\textwidth]{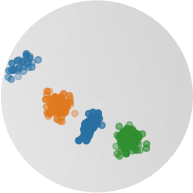}
    \caption{Wasserstein barycenter}
    \end{subfigure}
    \caption{Free support barycenters on $\SO$. Blue points represent the barycenter, while green and orange points represent the two input measures, which are closer together in (a) and (b) or further away in (c) and (d).}
    \label{fig:so3}
\end{figure}

%-------------------------------------------------
\section{Conclusions} \label{sec:conc}
%-------------------------------------------------

We investigated a new approach for sliced Wasserstein distance of spherical measures
and proved that this parallel slicing provides a rotational invariant metric on $\P(\S^{d-1})$ that induces the same topology as the Wasserstein distance. 
We provided numerical algorithms for the computation of the respective sliced barycenters, both with free or fixed support,
that are considerably faster than for the semicircular slicing, while producing comparable results in most cases, except when the input measures are highly concentrated around antipodal points.

Extending our method to the rotation group $\SO$,
we proved the metric properties of the proposed sliced Wasserstein distance based on a new Radon transform on $\SO$ and its  singular value decomposition.
An extensive numerical evaluation of the latter transform is planned, but out of the scope of this paper.
Further, it will be interesting to incorporate our  slicing approach into gradient flows on $\mathbb S^{d-1}$ for $d \gg 2$.

\subsubsection*{Acknowledgements}

We gratefully acknowledge the funding by the German Research Foundation (DFG): STE 571/19-1,
project number 495365311, within the Austrian Science Fund (FWF) SFB 10.55776/F68: \enquote{Tomography Across the Scales}.
L.B. acknowledges funding from the TU Berlin, Institute of Mathematics,
during his internship in 2023.
For open access purposes, the author has applied a CC BY public copyright license to any author-accepted manuscript version arising from this submission.

%-------------------------------------------------
\appendix
%-------------------------------------------------

\section{Metric properties of the parallelly sliced Wasserstein distance}
\label{sec:proof_SSW}

In the following, we show bounds between the sliced Wasserstein distance $\PSW_p$ and the spherical Wasserstein distance $\WS_p$
by a relation to the Euclidean case \cite[Sect.\ 5.1]{Bon13}.

\begin{lemma} \label{lem:dist-bound}
  The geodesic distance \eqref{eq:Sd-dist} on the sphere and the Euclidean distance are related via
  \begin{equation}  \label{eq:dist-bound}
    \norm{\zb\xi-\zb\eta}
    \le
    d(\zb\xi,\zb\eta)
    \le \frac{\pi}{2} \norm{\zb\xi-\zb\eta}
    ,\qquad \forall \zb\xi,\zb\eta \in \S^{d-1}.
  \end{equation}  
\end{lemma}
\begin{proof}
  We first show that
  \begin{equation} \label{eq:acos-bound}
    \sqrt{2-2x}
    \le
    \arccos(x)
    \le
    \frac{\pi}{2} \sqrt{2-2x}
    ,\qquad \forall x\in\II.
  \end{equation}
  For $x\in\II$,
  we have by \cite[4.4.2]{abst}
  \begin{equation} \label{eq:acos-bound-low}
    \arccos(x)
    =
    \int_{x}^{1} \frac{1}{\sqrt{1-t^2}} \d t
    =
    \int_{x}^{1} \frac{1}{\sqrt{1+t}\, \sqrt{1-t}} \d t
    \ge
    \int_{x}^{1} \frac{1}{\sqrt{2-2t}} \d t
    =
    \sqrt{2-2x},
  \end{equation}
  which is the first inequality of \eqref{eq:acos-bound}.
  Analogously, we have for $x\in[0,1]$ that
  $$
  \arccos(x)
  =
  \int_{x}^{1} \frac{1}{\sqrt{1-t^2}} \d t
  \le
  \int_{x}^{1} \frac{1}{\sqrt{1-t}} \d t
  =
  2\sqrt{1-x}.
  $$
  Because $\arccos$ is convex on $[-1,0]$ and $\arccos(-1)=\pi$, 
  we obtain the second inequality of \eqref{eq:acos-bound} for all $x\in[-1,1]$.
  The assertion follows from the fact that 
  $$
  \norm{\zb\xi-\zb\eta}
  =
  \sqrt{\norm{\zb\xi}^2+\norm{\zb\eta}^2-2\inn{\zb\xi,\zb\eta}}
  =
  \sqrt{2-2\inn{\zb\xi,\zb\eta}}
  $$
  and 
  $
  d(\zb\xi,\zb\eta)
  =
  \arccos(\inn{\zb\xi,\zb\eta})
  $
  for all $\zb\xi,\zb\eta\in\S^{d-1}$ 
  by \eqref{eq:Sd-dist}.
\end{proof}

We extend a spherical measure $\mu\in \P(\S^{d-1})$ to a measure $\tilde\mu\in\P(\R^d)$ that is supported on $\S^{d-1}$ by setting 
\begin{equation} \label{eq:meas_SR}
\tilde\mu(B) \coloneqq \mu(B\cap\S^{d-1}),
\qquad \forall 
B\in\cB(\R^d).
\end{equation}

\begin{lemma} \label{thm:W-Euclidean}
  Let $\mu,\nu\in \P(\S^{d-1})$ be spherical measures with extensions $\tilde\mu,\tilde\nu\in\P(\R^d)$.
  Then the following inequalities between the spherical Wasserstein distance $\WS_p(\mu,\nu)$ and the Euclidean Wasserstein distance $\WS_p(\tilde\mu,\tilde\nu)$ on $\R^d$ holds:
  $$
  \WS_p(\tilde\mu,\tilde\nu)
  \le
  \WS_p(\mu,\nu)
  \le
  \frac{\pi}{2} \WS_p(\tilde\mu,\tilde\nu).
  $$
\end{lemma}
\begin{proof}
The Euclidean Wasserstein distance on $\X=\R^d$ is given in \eqref{eq:Wp} with $d(\zb x,\zb y) = \norm{\zb x- \zb y}$.
Any transport plan $\tilde\gamma \in \Pi(\tilde\mu,\tilde\nu)$ is supported only on  $\S^{d-1}\times\S^{d-1}$.
Hence, its restriction to $\S^{d-1}\times\S^{d-1}$ yields a transport plan in $\Pi(\mu,\nu)$.
Conversely, a transport plan $\gamma \in \Pi(\mu,\nu)$ can be extended to $\R^d$ by setting it to zero outside the sphere.
Hence, we have
\begin{equation*}
  \WS_p^p(\mu, \nu)
  =
  \inf_{\gamma\in\Pi(\mu, \nu)}
  \int_{\S^{d-1}\times\S^{d-1}} d(\bx,\by) \d\gamma(\bx,\by)
  =
  \inf_{\tilde\gamma\in\Pi(\tilde\mu,\tilde\nu)}
  \int_{\R^{d}\times\R^{d}} d(\bx,\by) \d\tilde\gamma(\bx,\by).
\end{equation*}
The claim follows by \autoref{lem:dist-bound}.
\end{proof}

\begin{proof}[Proof of \autoref{thm:SSW}]
We first show \eqref{eq:W-SW-equiv} via applying the respective result from the Euclidean space.
We briefly recall sliced OT on $\R^d$, see \cite{BonRabPeyPfi15},
with the slicing operator
$
  \cS^{\R^d}_{\bpsi} 
  \colon \R^d\to\R$, 
$\bx \mapsto \inn{\bpsi,\bx}$
for
$\bpsi\in\S^{d-1}$
and
the {sliced Wasserstein distance} 
\begin{equation} \label{eq:RSW}
  \RSW_p^p(\mu,\nu)
  \coloneqq
  \int_{\S^{d-1}} \WS_p^p\left((\cS^{\R^d}_{\bpsi})_\# \mu,(\cS^{\R^{d}}_{\bpsi})_\# \nu\right) \d u_{\S^{d-1}} (\bpsi),
  \qquad \mu,\nu\in\P(\R^d).
\end{equation}
Comparing the Euclidean sliced Wasserstein distance $\RSW$ with the spherical sliced Wasserstein distance \eqref{eq:SW}, we see that 
\begin{equation} \label{eq:SSW-RSW}
\PSW_p(\mu,\nu)
=
\mathrm{RSW}_p(\tilde\mu,\tilde\nu).
\end{equation}
By \cite[thm.\ 5.1.5]{Bon13},
there exist constants $\tilde c_{d,p}, \tilde C_{d,p}$ such that
for all measures $\tilde\mu,\tilde\nu \in \P(\R^d)$ which are supported in a ball of fixed radius $R>0$ we have
\begin{equation}
  \mathrm{RSW}_p(\tilde\mu,\tilde\nu)
  \le
  \tilde c_{d,p} \WS_p(\tilde\mu,\tilde\nu)  
  \le
  \tilde C_{d,p} R^{1-\frac{1}{p(d+1)}}\,
  \mathrm{RSW}_p(\tilde\mu,\tilde\nu)^{\frac{1}{p(d+1)}}.
\end{equation}
As $\tilde\nu$ and $\tilde\mu$ are by construction supported in a ball of radius $1+\varepsilon$ for any $\varepsilon>0$,
the validity of \eqref{eq:W-SW-equiv} follows by invoking \eqref{eq:SSW-RSW} and \autoref{thm:W-Euclidean}.

Next we prove the metric properties.
The symmetry and the triangle inequality follow from
the corresponding properties of the Wasserstein distance
and the $p$-norm on $\S^{d-1}$.
The positive definiteness and the equivalence to the spherical Wasserstein distance follow from \eqref{eq:W-SW-equiv}.
The rotational invariance of $\PSW$ follows since $\U$ is rotationally invariant.
\end{proof}

\section{Proofs from \autoref{sec:SO}}
\label{sec:proof_T}

\begin{proof}[Proof of \autoref{thm:svd}]
  Let $n\in\NN$ and $j,k\in\{-n,\dots,n\}$.
  Using the product identity \cite[cor.\ 2.11]{hiediss07}
  \begin{equation} \label{eq:D_prod}
  D_n^{j,k}(\bP\bQ) 
  = 
  \sum_{\ell=-n}^{n} D_n^{j,\ell}(\bP) D_n^{\ell,k}(\bQ)
  \qquad \forall \bP,\bQ\in\SO,
  \end{equation}
  we have 
  \begin{equation*}
    \cT D_n^{j, k}(\bQ, \omega)
    =
    \frac{1}{4\pi^2}(1-\cos(\omega))\,
    \sum_{\ell=-n}^{n}
    D_n^{j, \ell}(\bQ)
    \int_{\S^2}
    D_n^{\ell, k}(\rR_{\bxi}(\omega))
    \d\sigma_{\S^2}(\bxi).
  \end{equation*}
  We write $\bxi=\sph(\varphi,\vartheta)$ with the spherical coordinates \eqref{eq:sph}.
  Since
  \begin{equation}
    \rR_{\sph(\varphi,\vartheta)}(\omega)
    =
    \rR_{\be^3}(\varphi)
    \rR_{\be^2}(\vartheta)
    \rR_{\be^3}(\varphi)
    \rR_{\be^3}(\omega)
    \rR_{\be^2}(-\vartheta)
    \rR_{\be^3}(-\varphi)
    =
    \eul(\varphi,\vartheta,0)
    \eul(\omega,-\vartheta,-\varphi),
  \end{equation}
  we have by \eqref{eq:D} and \eqref{eq:D_prod}
  \begin{equation} \label{eq:D_axan}
    D_n^{j,k}(\rR_{\sph(\varphi,\vartheta)}(\omega))
    =
    \sum_{\ell=-n}^{n}
    D_n^{j,\ell}(\eul(\varphi,\vartheta,0))\,
    \e^{-\i \ell \omega}\,
    D_n^{\ell,k}(\eul(0,-\vartheta,-\varphi)),
  \end{equation}
  cf.\ \cite[§~4.5]{Varsha88}.
  Hence, we obtain 
  \begin{multline*}
    \cT D_n^{j,k}(\bQ,\omega)
    =
    \frac{1}{4\pi^2}(1-\cos(\omega))
    \\
    \sum_{\ell=-n}^{n}
    D_n^{j,\ell}(\bQ)
    \sum_{m=-n}^{n}
    \int_{0}^{\pi} \int_{\T}
    D_n^{\ell,m}(\eul(\varphi,\vartheta,0))
    \e^{-\i m \omega}
    D_n^{m,k}(\eul(0,-\vartheta,-\varphi))
    \d\varphi
    \sin(\vartheta) \d\vartheta 
    .
  \end{multline*}
  With the symmetry 
  $
  D_n^{m,k}(\bQ)
  =
  \overline{D_n^{k,m}(\bQ^\top)}
  $
  and \eqref{eq:D},
  we see that
  \begin{multline*}
    \cT D_n^{j,k}(\bQ,\omega)
    =
    \frac{1}{4\pi^2}(1-\cos(\omega))
    \\
    \sum_{\ell=-n}^{n}
    D_n^{j,\ell}(\bQ)
    \sum_{m=-n}^{n}
    \e^{-\i m \omega}
    \int_{0}^{\pi} \int_{\T}
    \e^{-\i \ell\varphi}
    d_n^{\ell,m}(\cos(\vartheta))
    d_n^{k,m}(\cos(\vartheta))
    \e^{\i m \varphi}
    \d\varphi
    \sin(\vartheta) \d\vartheta 
    .
  \end{multline*}
  With the orthogonality of the exponentials and the d-functions in \eqref{eq:d_ortho},
  we obtain
  \begin{equation*}
    \cT D_n^{j, k}(\bQ, \omega)
    =
    (1-\cos(\omega))\,
    \frac{1}{(2n+1)\pi}
    D_n^{j, k}(\bQ)
    \sum_{\ell=-n}^{n}
    \e^{-\i \ell \omega}
    .
  \end{equation*}
  The expansion relation of the Dirichlet kernel 
  \begin{equation*}
    \sum_{\ell=-n}^{n}
    \e^{-\i \ell \omega}
    =
    \frac{\sin((n+\frac12) \omega)}{\sin(\frac\omega2)}
  \end{equation*}
  and the half angle formula 
  $
  (1-\cos(\omega))
  =
  2\sin(\omega/2)^2 
  $ 
  yield
  \begin{equation*}
    \cT D_n^{j, k}(\bQ, \omega)
    =
    \frac{2}{(2n+1)\pi}
    D_n^{j, k}(\bQ)\,
    \sin\left((n+\tfrac12) \omega\right)\,
    \sin(\tfrac{\omega}{2})
    ,
  \end{equation*}
  which implies \eqref{eq:svd}.
  The orthogonality of $F_n^{j,k}$ follows from the orthonormality \eqref{eq:D_ortho} of the rotational harmonics $\widetilde D_n^{j,k}$.
  Using the identity $(\sin(\frac\omega2))^2 = (1+\cos(\omega))/2$ and the orthogonality of the cosine, we obtain
  \begin{equation*}
    \int_{0}^{\pi}
    \left(\sin\left((n+\tfrac12) \omega\right)\right)^2\,
    \left(\sin(\tfrac\omega2)\right)^2
    \d\omega
    =
    \begin{cases}
      \pi/4
      , & n \in\N, \\
      3\pi/8
      , & n=0.
    \end{cases}\qedhere
  \end{equation*}
\end{proof}

\begin{proof}[Proof of \autoref{thm:T_inj}]
This proof uses a similar structure as \cite[Thm.\ 3.7]{QueBeiSte23}.
Let $\mu,\nu \in \M(\SO)$ such that $\cT\mu=\cT\nu$. 
For $g\in C(\SO\times \II)$, we have by \autoref{prop:V-as-adj} 
\begin{equation}
  \inn{\mu, \cT^*g}
  =
  \inn{\nu, \cT^*g}.
\end{equation}
We show that $\{\cT^*g: g\in C(\SO\times \II)\}$ is a dense subset of $C(\SO)$, which implies $\mu=\nu$.

The {Sobolev space} $H^s(\SO)$ with $s\ge0$ 
is defined as the completion of $C^\infty(\SO)$ with respect to the Sobolev norm
\begin{equation} \label{eq:SO_Sobolev}
  \norm{f}_{H^s(\SO)}^2
  \coloneqq
  \sum_{n=0}^\infty \left( n+\tfrac12\right)^{2s} \sum_{j,k=-n}^{n}
  \frac{8\pi^2}{2n+1} | \langle f, D_n^{j,k}\rangle|^2.
\end{equation}
Let $s>2$,
then $H^s(\SO)$ is dense in $C(\SO)$, cf.\ \cite[Lem.\ 2.22]{hiediss07}.
Let $f\in H^s(\SO)$.
Since $\cT$ is injective on $L^2(\SO)$ by \autoref{thm:svd},
we have 
$f = \cT^* g$
if and only if
$\cT f = \cT\cT^* g$.
In the following, we show that 
$$
g
\coloneqq
(\cT\cT^*)^{-1} \cT f
$$
is in $C(\SO\times [0,\pi])$,
then we obtain $f = \cT^* g$, which shows that $H^s(\SO)\subset \cT^*(C(\SO\times[0,\pi])$ and therefore the assertion.

Since $\cT^*$ has the same singular functions as $\cT$ and the conjugate singular values,
we obtain by the singular value decomposition \eqref{eq:svd} that
\begin{equation} \label{eq:TT}
  (\cT \cT^*)^{-1} \cT f
  =
  \sum_{n=0}^{\infty} \sum_{j,k=-n}^{n}
  \frac{1}{\lambda_n^{\cT}}\, \inn{f, D_n^{j,k}}_{L^2(\SO)}\, F_n^{j,k}.
\end{equation}
We want to show that the right-hand side of \eqref{eq:TT} converges uniformly on $C(\SO\times [0,\pi])$.
Let $(\bQ,\omega) \in \SO\times [0,\pi]$ and $N\in\N$.
Inserting $\lambda_n^{\cT}$ from \autoref{thm:svd}, we have
\begin{align} 
  &
  \abs{\sum_{n=0}^{\infty} \sum_{j,k=-n}^{n}
  \frac{1}{\lambda_n^{\cT}}\, \inn{f, D_n^{j,k}}_{L^2(\SO)}\, F_n^{j,k}(\bQ,\omega)
  -\sum_{n=0}^{N-1} \sum_{j,k=-n}^{n}
  \frac{1}{\lambda_n^{\cT}}\, \inn{f, D_n^{j,k}}_{L^2(\SO)}\, F_n^{j,k}(\bQ,\omega)
  }
  \\
  &\le
  \frac12 \sum_{n=N}^{\infty} \sum_{j,k=-n}^{n}
  \left(n+\tfrac12\right)\, \abs{\inn{f, D_n^{j,k}}_{L^2(\S^2)}} \abs{\widetilde D_n^{j,k}(\bQ)}
  \\
  &\le
  \frac12 \sqrt{\sum_{n=N}^{\infty} \sum_{j,k=-n}^{n}
  \left(n+\tfrac12\right)^{2s}\, \abs{\inn{f, D_n^{j,k}}_{L^2(\S^2)}}^2}
  \sqrt{\sum_{n=N}^{\infty} 
  \left(n+\tfrac12\right)^{2-2s} \sum_{j,k=-n}^{n}\abs{\widetilde D_n^{j,k}(\bQ)} },
\end{align}
where we made use of the Cauchy--Schwarz inequality.
In the last equation,
the first part is bounded by the Sobolev norm \eqref{eq:SO_Sobolev}.
For the second part, the addition theorem \cite[Thm.\ 2.14]{hiediss07} yields
\begin{equation}
  \sum_{n=N}^{\infty} 
  \left(n+\tfrac12\right)^{2-2s}
  \sum_{j,k=-n}^{n} \abs{\widetilde D_n^{j,k}(\bQ)}
  =
  \sum_{n=N}^{\infty}
  \left(n+\tfrac12\right)^{2-2s}
  \frac{(2n+1)(n+1)}{8\pi^2}
  <\infty
\end{equation}
since $s>2$.
Hence, the right-hand side of \eqref{eq:TT}
converges uniformly to a continuous function on $\SO\times [0,\pi]$,
which finally implies that $g$ is continuous.
\end{proof}

\section{Relation of the rotation group with the 3-sphere}
\label{sec:SO-S3}

We show a relation of the sliced Wasserstein distance \eqref{eq:SW_SO} with an analogue of $\PSW$ on the sphere $\S^3$,
making use of the representation of $\SO$ via unit quaternions, see \cite[§ 2.6]{Mor04}.
The algebra of quaternions $\H$ consists of vectors $\bq = (q_0,q_1,q_2,q_3) = (q_0,\bq')\in\R^4$,
where $\bq'=(q_1,q_2,q_3)$ is called the vector part,
with the standard addition and the multiplication 
\begin{equation} \label{eq:H-mult}
\bq\diamond\br
\coloneqq
(r_0q_0 - \bq' \cdot\br', q_0\br' + r_0\bq' + \bq'\times\br'),
\end{equation}
where $\cdot$ is the scalar product and $\times$ is the cross product in $\R^3$.
The unit quaternions $\{q\in\H\mid q_0^2+q_1^2+q_2^2+q_3^2=1\}$ can be identified with $\S^3$.
The inverse of $\bq\in\S^3$ with respect to $\diamond$ is $\bar\bq = (q_0,-q_1,-q_2,-q_3)$.
The map 
\begin{equation} \label{eq:phi}
\phi\colon \S^3\to\SO,\quad 
\bq \mapsto \rR_{{\bq'}/{\sqrt{1-q_0^2}}}(2\arccos(q_0))
\end{equation}
is surjective and satisfies $\phi^{-1}(\phi(\bq))=\{\bq,-\bq\}$ for all $\bq\in\S^3$.
It is a homomorphism in the sense that $\phi(\bq\diamond\br) = \phi(\bq)\phi(\br)$.
By \cite[(2.6)]{hiediss07}, the integral on $\SO$ is transformed via
\begin{equation} \label{eq:SO_int2}
\begin{split}
    \int_{\SO} f(\bQ) \d \sigma_{\SO}(\bQ)
    &=
    \int_{-1}^{1} f\circ\phi(q_0,\bq') 4 \sqrt{1-q_0^2} \d q_0 \d\sigma_{\S^2}\left(\sqrt{1-q_0^2}\bq' \right)
    \\&\overset{\eqref{eq:Sd_measure}}{=}
    4 \int_{\S^3} f\circ\phi(\bq) \d\sigma_{\S^3}(\bq).
    \end{split}
\end{equation}

We denote the set of even probability measures on $\SO$ by
$$
\Pe(\S^3) \coloneqq
\{\mu\in\P(\S^3) \mid \mu(B) = \mu(-B) \,\forall B\in\cB(\S^3)\}.
$$

\begin{theorem} \label{thm:SOSW-S3}
    Let $\mu,\nu\in \P(\SO)$ and $p\in[1,\infty)$.
    We define 
    $
    c\colon \II \to  [0,\pi],
    $
    $
    c(t) \coloneqq 2\arccos \abs{t}.
    $
    Then
    \begin{equation} \label{eq:SOSW-S3}
    \SOSW_p^p(\mu,\nu)
    =
    \int_{\S^3}
    \WS_p^p \left(c_\#\U_{\bq}(\mu\circ\phi), c_\#\U_{\bq}(\nu \circ \phi)\right) \d u_{\S^3}(\bq).
    \end{equation}
\end{theorem}

\begin{proof}
Let $\bxi\in\S^3$, $\bQ\in\SO$ and $q\in\phi^{-1}(\SO)$.
By the multiplication-invariance of $\phi$, we have
\begin{equation}
    d_\bQ\circ \phi(\bxi)
    =
    \angle(\bQ^\top \phi(\bxi))
    =
    \angle(\phi(\bar\bq \diamond \bxi)).
\end{equation}
We note that since the pushforward $\phi_\#$ is bijective from $\Pe(\S^3)$ to $\P(\SO)$,
and its inverse is given by $\mu\circ\phi \in\Pe(\S^3)$. 
Since $\angle(\phi(\br)) = 2\arccos\abs{r_0}$ for any $\br\in\S^3$,
we obtain by \eqref{eq:H-mult}
\begin{equation} \label{eq:dq}
    d_\bQ\circ \phi(\bxi)
    =
    2\arccos\abs{q_0 \xi_0+\bq'\cdot\bxi'}
    =
    2\arccos\abs{\bq\cdot\bxi}
    =
    c\circ \cS_\bq(\bxi).
\end{equation}
Since $\phi_\# (\mu\circ\phi) = \mu$, we obtain 
\begin{align}
    \SOSW_p^p(\mu,\nu)
    &\overset{\eqref{eq:SW_SO}}=
    \frac{1}{8\pi^2}
    \int_{\SO} \WS_p^p\left( (d_\bQ\circ\phi)_\# (\mu\circ\phi), (d_\bQ\circ\phi)_\# (\nu\circ\phi) \right) \d\sigma_{\SO}(\bQ)
    \\
    &\overset{\eqref{eq:SO_int2}}=
    \frac{4}{8\pi^2} \int_{\S^3} \WS_p^p\left( (c\circ \cS_\bq)_\# (\mu\circ\phi), (c\circ \cS_\bq)_\# (\nu\circ\phi) \right) \d\sigma_{\S^3}(\bq). \qedhere
\end{align}
\end{proof}

The right-hand side of \eqref{eq:SOSW-S3} mimics the parallelly sliced Wasserstein distance \eqref{eq:SW} on $\S^3$ between $\mu\circ\phi$ and $\nu\circ\phi$,
except for the additional transformation $c$.
We want to point out that this equivalence in \autoref{thm:SOSW-S3} holds only for even measures on $\S^3$, since we always have $\mu\circ\phi\in\Pe(\S^3)$.

\section{Sliced Wasserstein distances for antipodal point measures} \label{sec:particular}
We study sliced Wasserstein barycenters of two antipodal Dirac measures $\mu_1 = \delta_{\be^3}$ and $\mu_2 = \delta_{-\be^3}$ on the sphere $\S^2$, as presented in \autoref{rem:antipo_diracs}. This case exhibits some differences between the Wasserstein distance, the parallelly sliced Wasserstein distance and the semicircular sliced Wasserstein distance. We define the equator $C \coloneqq \{\bxi\in \S^2\mid \xi_3=0\}$.
    
\begin{proposition}
    The 2-Wasserstein barycenters of the two antipodal Dirac measures $\mu_1$ and $\mu_2$ are the probability measures $\nu \in \P(\S^2)$ whose support is included in the equator $C$. 
\end{proposition}
    
\begin{proof}
    The Wasserstein barycenters of $\mu_1$ and $\mu_2$ on the sphere are the measures in $\P(\S^2)$ minimizing $\cE_{\WS}(\nu) \coloneqq \frac{1}{2}\WS_2^2(\nu, \mu_1) + \frac{1}{2}\WS_2^2(\nu, \mu_2)$. 
    Let $\nu \in \P(\S^2)$. 
    With the map $\zen \colon \S^2\to[0,\pi]$, $\bxi\mapsto \arccos(\xi_3)$,
    we have
    \begin{equation*}
        \WS_2^2(\nu, \delta_{\be^3}) 
        = \int_{\S^2} d_{\S^2}(\bpsi, \be^3)^2 \d \nu(\bpsi)
        = \int_0^\pi t^2 \d(\zen)_\#\nu(t)
    \end{equation*}
    and similarly
    \begin{equation*}
        \WS_2^2(\nu, \delta_{-\be^3}) = \int_0^\pi (\pi - t)^2 \ \d(\zen)_\#\nu(t).
    \end{equation*}
    Hence
    \begin{equation*}
        \cE_{\WS}(\nu)
        = \frac{1}{2}\int_0^\pi \big[t^2 + (\pi - t)^2\big]\d(\zen)_\#\nu(t).
    \end{equation*}
    The integrand has a unique minimizer $t=\frac{\pi}{2}$ and its minimum is $\frac{\pi^2}{4}$. Therefore, $\cE_{\WS}(\nu)\ge\frac{\pi^2}{4}$ for all $\nu\in\P(\S^2)$ with equality if and only if $(\zen)_\#\nu(dt) = \delta_{\frac{\pi}{2}}$, i.e., if $\nu(C) = 1$.
\end{proof}

However, this observation does not apply to the parallelly sliced Wasserstein barycenters.

\begin{proposition}
    \label{thm:antipod_par}
    All probability measures on the sphere are parallelly sliced Wasserstein barycenters of $\mu_1$ and $\mu_2$. 
\end{proposition}

\begin{proof}
    Let $\nu \in \P(\S^2)$.
    We have
    \begin{align*}
        \cE_V(\nu) 
        \coloneqq{}& \frac{1}{2}\PSW_2^2(\nu, \delta_{\be^3}) + \frac{1}{2}\PSW_2^2(\nu, \delta_{-\be^3})\\
        ={}& \frac{1}{2}\int_{\S^2}\Big[\WS_2^2\big(\U_\bpsi\nu, \U_\bpsi\delta_{\be^3}\big) + \WS_2^2\big(\U_\bpsi\nu, \U_\bpsi\delta_{-\be^3}\big)\Big]\d u_{\S^2}(\bpsi).
        \intertext{Using that $\U_\bpsi \delta_{\pm\be^3} = \delta_{\inn{\bpsi, \pm\be^3}} = \delta_{\pm\psi_3}$ and $\U_\bpsi \nu = (\cS_\bpsi)_\#\nu$ for all $\bpsi\in\S^2$, we have}
        \cE_V(\nu)
        ={}& \frac{1}{2}\int_{\S^2} \left[\int_{-1}^1\abs{t - \inn{\bpsi, \be^3}}^2 \d(\cS_{\bpsi})_\#\nu(t) + \int_{-1}^1\abs{t + \inn{\bpsi, \be^3}}^2 \d(\cS_{\bpsi})_\#\nu(t)\right]\d u_{\S^2}(\bpsi)\\
        ={}& \frac{1}{2}\int_{\S^2} \int_{-1}^1\Big[2t^2 + 2 \inn{\bpsi, \be^3}^2\Big] \d(\cS_{\bpsi})_\#\nu(t)\ u_{\S^2}(d\bpsi).
    \end{align*}
    Using \eqref{eq:Sd_measure} and the rotation invariance of the spherical integral, we have 
    \begin{equation} \label{eq:int_inn2}
        \int_{\S^2}\inn{\bpsi, \bxi}^2\d u_{\S^2}(\bpsi)
        = \frac{1}{2}\int_{-1}^1 t^2\d t = \frac{1}{3}
        \qquad \forall \bpsi\in\S^2.
    \end{equation}
    Hence, we obtain
    \begin{equation*}
        \cE_V(\nu)
        = \int_{\S^2} \int_{\S^2} \left[ \inn{\bpsi, \bxi}^2 + \inn{\bpsi, \be^3}^2 \right] \d u_{\S^2}(\bpsi) \d\nu(\bxi)
        = \frac{2}{3}.\qedhere
    \end{equation*}
\end{proof}

Let us now consider the case of the semicircular sliced Wasserstein distance. Such slicing operator is much harder to manipulate. Therefore, we did not manage to determine the semicircular sliced Wasserstein barycenters of $\mu_1$ and $\mu_2$, but we can show that the observation made in \autoref{thm:antipod_par} does not hold in the case of the $\SSW$ distance.

\begin{proposition}
    The uniform probability measure on the sphere is not a semicircular sliced 2-Wasserstein barycenter. In particular, considering $\chi$, the uniform probability measure on the equator $C$, and $u_{\S^2}$ the uniform probability measure on $\S^2$, we have
    $$\tfrac{1}{2}\SSW_2^2(\mu_1, \chi) + \tfrac12 \SSW_2^2(\mu_2, \chi) < \tfrac{1}{2}\SSW_2^2(\mu_1, u_{\S^2}) + \tfrac12 \SSW_2^2(\mu_2, u_{\S^2}).$$
\end{proposition}

\begin{proof}\newcommand{\SP}{\S^2} 
    Recall the circle $\T=\R / 2\pi\Z$.
    For any $x \in \R$, we note $[x] = x + 2\pi\Z$ the equivalence class of $x$ and for any $\gamma\in \T$ we note $\tilde\gamma$ its representative in $[0, 2\pi[$.
    Let $\mu\in\P(\S^2)$.
    The semicircular sliced Wasserstein barycenters are the minimizers of the functional
    \begin{equation*}
        \cE_{S}(\mu) \coloneqq \frac 12 \SSW_2^2(\mu, \delta_{\be^3}) + \frac 12 \SSW_2^2(\mu, \delta_{-\be^3}),
    \end{equation*}
    where
    \begin{equation*}
        \SSW_2^2(\mu, \delta_{\be^3}) = \int_{\S^2} \W_2^2(\cA_{\bpsi\#}\mu, \cA_{\bpsi\#}\delta_{\be^3})\d u_{\S^2}(\bpsi)
    \end{equation*}
    and the slicing operator $\cA_\bpsi$ is given in \autoref{rem:SSW}.
    Let $\bpsi = \Phi(\varphi, \theta) \in \SP$.
    Since we integrate over $\S^2$, we can assume $\bpsi \notin \{\pm\be^3\}$.
    Then we have $(\cA_\bpsi)_\#\delta_{\be^3} = \delta_{[\pi]}$ and $(\cA_\bpsi)_\# \delta_{-\be^3} = \delta_{[0]}$. Therefore, 
    \begin{equation} \label{eq:WsA}
        \W_2^2(\cA_{\bpsi\#}\mu, \cA_{\bpsi\#}\delta_{\be^3})
        = \int_\T |\tilde \gamma - \pi|^2 \d \cA_{\bpsi\#}\mu (t)
        = \int_{\S^2} \left|\tilde\cA_\bpsi(\bxi) - \pi \right|^2 \d\mu(\bxi).
    \end{equation}
    Using spherical coordinates $\bxi = \sph(\alpha,\beta)$, we see that
    $$
    \cA_{\bpsi}(\bxi)
    =
    \azi\big(\eul(0,-\theta,-\varphi) \sph(\alpha,\beta)\big)
    =
    \azi\big(\eul(0,-\theta,-\varphi+\alpha) \sph(0,\beta)\big)
    =
    \cA_{\sph(\varphi-\alpha,\theta)} (\sph(0,\beta)).
    $$
    Hence, with the substitution $\varphi\mapsto\varphi-\alpha$, we have
    \begin{align*}
        \SSW_2^2(\mu,\delta_{\be^3})
        &=
        \frac{1}{4\pi} \int_{\S^2} \int_{0}^{\pi} \int_{0}^{2\pi}
        \abs{\tilde\cA_{\sph(\varphi-\alpha,\theta)} (\sph(0,\beta)) - \pi}^2 \sin(\theta) \d\varphi \d\theta \d\mu(\sph(\alpha,\beta))\\
        &=
        \frac{1}{4\pi} \int_{\S^2} \int_{0}^{\pi} \int_{0}^{2\pi}
        \abs{\tilde\cA_{\sph(\varphi,\theta)} (\sph(0,\beta)) - \pi}^2 \sin(\theta) \d\varphi \d\theta \d\mu(\sph(\alpha,\beta))\\
        &=
        \frac{1}{4\pi} \int_{\S^2} \int_{0}^{\pi} \int_{0}^{2\pi}
        \abs{\tilde\cA_{\sph(\varphi,\theta)} (\sph(0,\beta)) - \pi}^2 \sin(\theta) \d\varphi \d\theta \d\zen_\#\mu(\beta).
    \end{align*}
    Introducing the functions
    \begin{equation*}
        F_1\colon [0, \pi]\to \R, \quad \beta \mapsto \frac{1}{4\pi} \int_0^{2\pi}\int_0^\pi \left|\tilde\cA_{\Phi(-\varphi, \theta)}(\Phi(0, \beta)) - \pi\right|^2 \sin(\theta) \d\theta \d\varphi,
    \end{equation*}
    and
    \begin{equation*}
        F \colon[0, \pi]\to \R, \quad\beta\mapsto \tfrac 12 F_1(\beta) + \tfrac 12 F_1(\pi-\beta),
    \end{equation*}
    we obtain by using the symmetry that
    \begin{equation*}
    \cE_S(\mu)
    =
    \tfrac12 \SSW_p^p(\mu, \delta_{\be^3})
    +
    \tfrac12 \SSW_p^p(\mu, \delta_{-\be^3}) 
    = 
    \int_0^\pi F(\beta) \d\zen_\#\mu(\beta).
    \end{equation*}
    Let $\beta\in[0,\pi]$.
    We study $f_\beta(\varphi, \theta) \coloneqq \tilde\cA_{\Phi(-\varphi, \theta)}(\Phi(0, \beta))$. 
    For $\theta \in (0, \pi)$ and $\varphi \in (0, 2\pi) \setminus \{\pi\}$, we have
    \begin{align*}
        f_\beta(\varphi, \theta) 
        &= \azi\big(\rR_{\be^2}(-\theta)\rR_{\be^3}(\varphi)\Phi(0, \beta)\big)\\
        &= \azi\left(
            \begin{bsmallmatrix}
                \cos(\theta) & 0 & -\sin(\theta) \\
                0 & 1 & 0 \\
                \sin(\theta) & 0 & \cos(\theta)
            \end{bsmallmatrix}
            \begin{bsmallmatrix}
                \cos(\varphi) & -\sin(\varphi) & 0 \\
                \sin(\varphi) & \cos(\varphi) & 0 \\
                0 & 0 & 1
            \end{bsmallmatrix}
            \begin{bsmallmatrix}
                \sin(\beta) \\ 0 \\ \cos(\beta)
            \end{bsmallmatrix}
            \right)\\
        &= \azi\left(
            \begin{bsmallmatrix}
                \cos(\theta)\cos(\varphi) \sin(\beta) - \sin(\theta)\cos(\beta) \\
                \sin(\varphi)\sin(\beta) \\
                \sin(\theta)\cos(\varphi)\sin(\beta) + \cos(\theta)\cos(\beta)
            \end{bsmallmatrix}
            \right).
    \end{align*}
    We identify some symmetries. For $\varphi \in (0, 2\pi) \setminus \{\pi\}$, and $\theta \in (0, \pi)$, we have 
    $$f_\beta(2\pi - \varphi, \theta) = 2\pi - f_\beta(\varphi, \theta) \quad\text{ or }\quad f_\beta(2\pi - \varphi, \theta) = f_\beta(\varphi, \theta).$$ 
    In both cases, $(f_\beta(2\pi - \varphi, \theta) - \pi)^2 = (f_\beta(\varphi, \theta) - \pi)^2$. Moreover, for $\varphi, \theta \in (0, \pi)$, 
    \begin{equation}\label{eq:fb_sym2}
        f_\beta(\varphi, \pi - \theta) = f_\beta(\pi -  \varphi, \theta)\quad\text{ and }\quad f_{\pi - \beta}(\varphi, \theta) = \pi - f_\beta(\pi - \varphi, \theta).
    \end{equation}
    Hence, we have
    \begin{align*}
        F_1(\beta)
        &= \frac{1}{\pi} \int_0^\pi \int_0^{\frac{\pi}{2}} (f_\beta(\varphi, \theta) - \pi)^2 \sin(\theta) \d\theta \d\varphi
    \qquad \text{and}\\
        F_1(\pi-\beta) &= \frac 1\pi \int_0^\pi \int_0^{\frac{\pi}{2}} f_\beta(\varphi, \theta)^2 \sin(\theta) \d\theta \d\varphi.
    \end{align*}
    Eventually, $F$ is given by 
    \begin{equation*}
        F(\beta) 
        = \frac 1\pi \int_0^\pi \int_0^{\frac{\pi}{2}} \left(f_\beta(\varphi, \theta) - \frac{\pi}2\right)^2 \sin(\theta) \d\theta \d\varphi \ + \frac{\pi^2}4,
    \end{equation*}
    where, for $\beta, \varphi, \theta \in (0, \pi)$, we have 
    \begin{equation*}
        f_\beta(\varphi, \theta) = \frac{\pi}2 - \arctan\left(\cos(\theta)\cot(\varphi) - \frac{\sin(\theta)}{\sin(\varphi)}\cot(\beta)\right).
    \end{equation*}

    Let us now focus on the two particular cases of the theorem. Let $\chi$ be a measure supported by the equator. By the symmetry \eqref{eq:fb_sym2}, we have
    \begin{align*}
        \cE_S(\chi)
        = F\left(\frac \pi 2\right)
        &= \frac 1\pi \int_0^\pi \int_0^{\frac{\pi}{2}} \left(f_{\frac{\pi}{2}}(\varphi, \theta) - \frac{\pi}2\right)^2 \sin(\theta) \d\theta \d\varphi \ + \frac{\pi^2}4\\
        &= \frac 2\pi \int_0^{\frac{\pi}2} \int_0^{\frac{\pi}{2}} \left(\arctan\left(\frac{\tan(\varphi)}{\cos(\theta)}\right) - \frac{\pi}2\right)^2 \sin(\theta) \d\theta \d\varphi \ + \frac{\pi^2}4.
    \end{align*}
    Since $\arctan\left(\frac{\tan(\varphi)}{\cos(\theta)}\right) > \varphi$ for any $\varphi, \theta \in (0, \frac\pi2)$, and $t \mapsto \left(t-\frac\pi2\right)^2$ is strictly decreasing on $[0, \frac\pi2]$,
    we obtain
    \begin{equation*}
        \cE_S(\chi)
        < \frac 2\pi \int_0^{\frac{\pi}2} \int_0^{\frac{\pi}{2}} \left(\varphi - \frac{\pi}2\right)^2 \sin(\theta) \d\theta \d\varphi \ + \frac{\pi^2}4
        =  \frac{\pi^2}{12} + \frac{\pi^2}4
        = \frac{\pi^2}{3}.
    \end{equation*}
    For the uniform measure $u_{\S^2}$ on the sphere, we have $\cA_{\bpsi\#}u_{\S^2} = u_\T$ for any $\bpsi \in \SP$. From the first equality of \eqref{eq:WsA}, we obtain
    \begin{align*}
        \SSW_2^2(u_{\S^2}, \delta_{\be^3})
        &= \frac{1}{2\pi} \int_{\S^2} \int_0^{2\pi} |t - \pi|^2 \d t \d u_{\S^2}(\bpsi)
        = \frac{\pi^2}{3}.
    \end{align*}
    By symmetry, we have $\cE_S(u) = \frac{\pi^2}3$, and therefore $\cE_S(u) > \cE_S(\chi)$. 
\end{proof}

\small
\bibliographystyle{abbrvurl}
\bibliography{literature}

\end{document}